\documentclass[12pt, reqno]{amsart}
\usepackage{amsmath, amsthm, amscd, amsfonts, amssymb, graphicx, color,tikz}
\usepackage[bookmarksnumbered, colorlinks, plainpages]{hyperref}
\usepackage{mathrsfs}
\usepackage[utf8]{inputenc}
\usepackage{enumerate}
\usepackage{mathtools}
\usepackage{stmaryrd}

\newtheorem{theorem}{Theorem}[section]
\newtheorem{lemma}[theorem]{Lemma}
\newtheorem{proposition}[theorem]{Proposition}
\newtheorem{corollary}[theorem]{Corollary}
\newtheorem{fact}[theorem]{Fact}
\theoremstyle{definition}
\newtheorem{definition}[theorem]{Definition}
\newtheorem{example}[theorem]{Example}

\newtheorem{remark}[theorem]{Remark}
\newtheorem{question}[theorem]{Question}

\numberwithin{equation}{section}

\newcommand{\veps}{\varepsilon}
\newcommand{\CC}{\mathbb C}
\newcommand{\RR}{\mathbb R}
\newcommand{\NN}{\mathbb N}
\newcommand{\ZZ}{\mathbb Z}

\newcommand{\TT}{\mathbb T}

\newcommand{\essequal}{\varphi^{-1}(\mathcal B)=_{\mathrm{ess}}\mathcal B}
\newcommand{\sbt}{\,\begin{picture}(-1,1)(-1,-3)\circle*{3}\end{picture}\ }
\newcommand{\vertiii}[1]{{\left\vert\kern-0.25ex\left\vert\kern-0.25ex\left\vert #1 
    \right\vert\kern-0.25ex\right\vert\kern-0.25ex\right\vert}}

\def\Lomu{L_0(\Omega,\mu)}
\def\Lpmu{L_p(\Omega,\mu)}
\def\L2mu{L_2(\Omega,\mu)}
\def\X{\mathcal X}
\def\od{\mathfrak o}
\def\trans{\mathfrak t}
\def\odinv{\mathfrak p}

\allowdisplaybreaks[4]

\begin{document}
\setcounter{page}{1}

\title[Composition operators]
{On the dynamics of composition operators: supercyclicity, odometers and translations}
\date{\today}

\author[F. Bayart]{Frédéric Bayart}
\address{Laboratoire de Math\'ematiques Blaise Pascal UMR 6620 CNRS, Universit\'e Clermont Auvergne, Campus universitaire des C\'ezeaux, 3 place Vasarely, 63178 Aubi\`ere Cedex, France.}
\email{frederic.bayart@uca.fr}
\author[\'E. Matheron]{\'Etienne Matheron}
\address{Univ. Artois, UR 2462 - Laboratoire de Math\'{e}matiques de Lens (LML)\\ F-62300 Lens, France}
\email{etienne.matheron@univ-artois.fr}


\subjclass[2010]{}


\begin{abstract}
We study the dynamical properties of composition operators acting on Banach spaces of measurable functions. In particular, we study in some detail the composition operators induced by odometers, which  allows us to give a variety of new examples and counter-examples. We also get general statements about supercyclicity and frequent hypercyclicity of  
composition operators on $L_p\,$-$\,$spaces. 
\end{abstract}
\maketitle

\section{Introduction}

\subsection{Linear dynamics}
Linear dynamics,  the study of the dynamical properties of linear operators, has been the object of extensive research over the past forty years. We refer to \cite{BM09} and \cite{GePeBook} for detailed 
presentations of the area, and to \cite{Gil20} for a survey of more recent developments. Here, we just recall the few basic definitions which will be needed in this paper. We fix a complex Banach space $X$ and $T\in\mathfrak L(X)$, a bounded linear operator on $X$. 

The first definitions are related to the behaviour of $T$ with respect to the nonempty open subsets of $X$. We say that $T$ is {\bf topologically transitive} if for all nonempty open sets $U,V\subset X$, there exists $n\in\mathbb N$ such that $T^n(U)\cap V\neq\varnothing$; and that $T$ is {\bf topologically mixing} if for all $U,V\subset X$ nonempty open, there exists $n_0\in\mathbb N$ such that $T^n(U)\cap V\neq\varnothing$  for all $n\geq n_0$.

The next definitions are related to the orbits of $T$. We say that $T$ is  {\bf hypercyclic} if there exists $x\in X$ with a dense $T\,$-$\,$orbit, \textit{i.e.} $O(x,T):=\{T^nx:\ n\geq 0\}$ is dense in $X$. This can happen only if $X$ is separable, and when it is so, hypercyclicity is equivalent to topological transitivity. More generally, given a set $\Gamma\subset \CC$, we say that $T$ is {\bf $\Gamma\,$-$\,$supercyclic} if there exists $x\in X$ such that $\Gamma\cdot O(x,T)=\{zT^n x:\ z\in\Gamma,\ n\geq 0\}$ is dense in $X$. The vector $x$ is then called a $\Gamma$-supercyclic vector for $T.$ So  ``$\{1\}\,$-$\,$supercyclic'' means ``hypercyclic''; and $\CC\,$-$\,$supercyclic operators are simply called supercyclic.

Hypercyclic vectors have ``big'' orbits. At the other extreme, a vector $u\in X$ is called a {\bf periodic point} for $T$ if there exists $d\geq 1$ such that $T^d(u)=u.$ 
 The operator $T$ is said to be {\bf chaotic} if it is hypercyclic with a dense set of periodic points.

Finally, we say that $T$ is {\bf frequently hypercyclic} (resp. \textbf{$\mathcal U\,$-$\,$frequently hypercyclic}) if there exists $x\in X$ such that, for every nonempty open set $V\subset X$, the set
$\mathcal N_T(x,V):=\{n\in\NN:\ T^n x\in V\}$ 
has positive lower  density (resp. positive upper density). 

\smallskip
An important line of investigation in linear dynamics is the study of some concrete classes of operators, with a twofold aim: to find examples and counterexamples for various dynamical  properties and  the possible implications between them, and to understand better these classes of operators.
In this paper, we intend to follow this plan for \emph{composition operators} acting on Banach spaces of measurable functions, with a strong emphasis on composition operators  induced by \emph{odometers}.  The study of this class of operators has been pioneered in \cite{BDDD22}, and pursued very recently in \cite{DGV}. Let us first introduce the general framework of composition operators on Lebesgue spaces. 

\subsection{Composition operators}

Let $(\Omega,\mathcal B,\mu)$ be a $\sigma$-finite measure space (with $\mu\neq 0$), and let $\varphi:\Omega\to\Omega$ be a measurable transformation. We assume that the transformation $\varphi$ is \textbf{nonsingular}, which means that  $\mu(B)=0$ implies $\mu(\varphi^{-1}(B))=0$, for any $B\in\mathcal B$; in other words, the image measure $\mu\varphi^{-1}$ is absolutely continuous with respect to $\mu$. We say that $(\Omega,\mathcal B,\mu,\varphi)$ is a \emph{nonsingular measurable system}.

Denoting by $\Lomu$ the linear space of all measurable functions $f:\Omega\to\CC$, where two functions are identified if they are $\mu\,$-$\,$almost everywhere equal, the nonsingularity assumption ensures that the operator $C_\varphi:\Lomu\to \Lomu,\ f\mapsto f\circ \varphi$ is well defined. This operator $C_\varphi$ is called the  {\bf composition operator} with symbol $\varphi.$ It is well-known that $C_\varphi$ is bounded on $\Lpmu$, $1\leq p<\infty$ if and only if there exists some constant $K<\infty$ such that 
\begin{equation}\label{eq:continuitycompositionoperator}
\mu(\varphi^{-1}(B))\leq K\mu(B)\quad \textrm{ for all }B\in\mathcal B.
\end{equation}
 In that case $\|C_\varphi\|=\|f_\varphi\|_\infty^{1/p}$, where $f_\varphi$ is the Radon-Nykodim derivative  of the measure $\mu\varphi^{-1}$ with respect to $\mu$ (see  \cite[Theorem 2.1.1]{SinMan93}).
 
 We warn the reader that the definition of nonsingularity varies in the literature. In the present paper, we have chosen the definition which is most commonly used in the study of composition operators, whereas in many references on dynamical systems, a transformation $\varphi$ is said to be nonsingular when $\mu(\varphi^{-1}(B))=0$ \emph{if and only if} $\mu(B)=0$.


The nonsingular system 
$(\Omega,\mathcal B,\mu,\varphi)$ is said to be \textbf{invertible} if the transformation $\varphi$ is bijective, and if $\varphi^{-1}$ is measurable and nonsingular. (Note that if $(\Omega,\mathcal B)$ is a standard Borel space, then the measurablility of $\varphi^{-1}$ is for free.) 

\begin{definition}
The transformation $\varphi:(\Omega,\mathcal B,\mu)\to(\Omega,\mathcal B,\mu)$ is said to be
\begin{enumerate}
\item[-] {\bf conservative} if for each $B\in\mathcal B$ with $\mu(B)>0$ there exists $n\geq 1$ such that 
$\mu(B\cap \varphi^{-n}(B))>0$;
\item[-] {\bf dissipative} if there exists $W\in\mathcal B$ 
such that $\Omega=\bigcup_{n\in\ZZ}\varphi^n(W)$ mod$\,\mu$ and  $W$ is a {\bf wandering set} for $\varphi,$ \textit{i.e.} $\varphi^i(W)\cap\varphi^j(W)=\varnothing$ for all $i,j\in\ZZ,$ $i\neq j.$
\end{enumerate} 
\end{definition}

For example, if the measure $\mu$ is finite and $\varphi$ is measure-preserving, then $\varphi$ is conservative by Poincar\'e's recurrence Theorem; whereas if $\Omega=\ZZ$ and $\mu$ is the counting measure, then the shift map $\sigma(i)=i+1$ is dissipative. 

The importance of these notions stems from the fact that any invertible nonsingular system $(\Omega,\mathcal B,\mu,\varphi)$ can be decomposed into a conservative part and a dissipative part: there exist two {\rm (}possibly empty{\rm )} $\varphi\,$-$\,$invariant measurable sets $C$ and $D$ such that $C\cap D=\varnothing,$ $\Omega=C\cup D$, $\varphi_{|C}$ is conservative and $\varphi_{|D}$ is dissipative. This is (a special case of) {\bf Hopf's decomposition Theorem};  see e.g. \cite[Theorem 2.19]{Hawk} or \cite[Section 1.3]{Kr}.


\smallskip As it turns out, the study of composition operators induced by a nonsingular transformation seems to be easier when the transformation $\varphi$ is dissipative, especially if the technical assumption of being of \emph{bounded distorsion} is added.
The dynamical properties of $C_\varphi$ are then comparable to that of an associated weighted shift; see e.g.
\cite{DDM22} or \cite{DP21}. In the present paper, we investigate the composition operators induced by well-known examples of \emph{conservative} transformations, namely odometers. Before giving the relevant definitions, we briefly review some known results concerning the dynamics of general composition operators on $L_p\,$-$\,$spaces. 

\subsection{Dynamical properties of composition operators}

Topologically transitive and mixing composition operators on $\Lpmu$ have been completely characterized in \cite{BDP} (see also \cite{GG25}). These characterizations, which hold in fact in a context more general than $L_p\,$-$\,$spaces, will be recalled in Section \ref{sec:supercyclic}. 

\smallskip
Regarding \emph{supercyclicity} of composition operators, the situation is less clear. Recall first that $\RR\,$-$\,$supercyclicity and $\RR_+$-$\,$supercyclicity are always equivalent by \cite{BBP02}; and that for operators $T$ such that $T^*$ has no eigenvalue, supercyclicity and $\RR_+$-$\,$supercyclicity are equivalent by \cite{LeMu04}. In \cite{DM24}, 
$\mathbb R\,$-$\,$supercyclic composition operators on $\Lpmu$ are characterized, and supercyclic composition operators are characterized in the dissipative case. In Section \ref{sec:supercyclic}
of the present paper,  we will give (in a context more general than $L_p\,$-$\,$spaces), a characterization of composition operators which are $\Gamma_\alpha$-$\,$supercyclic for some $\alpha<\pi$, where $\Gamma_\alpha$ is the angular sector $\{ \vert\arg(z)\vert \leq\alpha\}$. This will imply in particular that for composition operators on $\Lpmu$, $\Gamma_\alpha$-$\,$supercyclicity and $\RR_+$-$\,$supercyclicity are always equivalent, and that they are equivalent to supercyclicity when the measure $\mu$ is non-atomic

\smallskip
As for \emph{frequent hypercyclicity}, again the situation is well understood for composition operators induced by  dissipative systems of bounded distorsion, see \cite{DP21}. Moreover it is also shown in \cite{DP21} that there is no hope to apply one of the main tools for proving frequent hypercyclicity (namely, the Frequent Hypercyclicity Criterion) in the conservative context, the case which interests us here. However, composition operators induced by odometers have plenty of periodic vectors (see Lemma \ref{periodicgeneral} below); and as shown in \cite{GMM21}, the periodic vectors can be very useful to check that a given operator is frequently hypercyclic or $\mathcal U$-frequently hypercyclic. We use this idea to give, in Sections \ref{sec:fhc} and \ref{sec:ufhc}, sufficient conditions for a general composition operator to be frequently hypercyclic or $\mathcal U$-frequently hypercyclic, which will prove to be rather efficient in the odometer context.

\subsection{Odometers}\label{subsec:odometers}

Let $(m_i)_{i\geq 1}$ be a sequence of integers with $m_i\geq 2$ for all $i\geq 1.$ Define $\Omega:=\prod_{i=1}^{\infty}\Omega_i$, where $\Omega_i=\mathbb Z/m_i\mathbb Z$. We endow $\Omega$ with the product topology, which turns it into a compact metrizable space. One possible compatible metric is given by $d(x,y):=2^{-i(x,y)}$, where $i(x,y)$ is the least index $i$ such that $x_i\neq y_i.$ 

The space $\Omega$ is a topological abelian group when addition is performed coordinatewise. However, we will consider another operation on $\Omega$, namely \emph{addition with carrying to the right}, which will be denoted by 
$\pmb +$ (boldface) and is defined as follows: for $x,y\in\Omega$ and $i\geq 1,$ 
\[ (x\pmb+y)_i=x_i+y_i+\veps_{i-1}\ [m_i]\] where the sequence $(\veps_i)_{i\geq 0}$ is defined inductively by
$\veps_0=0$ and, for $i\geq 1,$
\[\veps_i=\left\{
\begin{array}{ll}
0&\textrm{if }x_i+y_i+\veps_{i-1}<m_i\\
1&\textrm{otherwise.}
\end{array}\right.\]
Here, of course, we identify $\Omega_i=\mathbb Z/m_i\mathbb Z$ with $\llbracket 0,m_i-1\rrbracket$.

In other words, $(x\pmb+y)_1=x_1+y_1\ [m_1]$ and, for $i\geq 2$, $(x\pmb +y)_i=x_i+y_i\ [m_i]$ or $x_i+y_i+1\ [m_i]$, depending on whether one has to perform a carry due to the computation of $(x\pmb +y)_{i-1}$. This is indeed a (commutative) group law on $\Omega$ with neutral element $0=(0,0,0,\dots)$. If $x\in\Omega$, then $y=\pmb-x$ can be computed inductively: $y_1=0$ if $x_1=0$ and $y_1=m_1-x_1$ if $x_1>0$; and for $i\geq 2$, $y_{i}=m_{i}-(x_{i}+1)$ if the computation of $(x+y)_{i-1}$ has led to a carry, and otherwise $y_i=0$ if $x_i=0$ and $y_i=m_i-x_i$ if $x_i>0$. The group operations are clearly continuous, so $(\Omega, \pmb +)$ is a compact abelian group.

We consider the map $\od:\Omega\to\Omega$ defined as follows: 
\[ \od(x):=x\pmb+a\qquad\hbox{where $a=(1,0,0,0,\dots)$}.\]
In other words, $\od$ is the translation by $a$ in the group $(\Omega,\pmb+)$. By definition of the addition with carry $\pmb+$, we have $\od(m_1-1,m_2-1,\dots)=(0,0,\dots)$ and, for any $x\neq (m_1-1,m_2-1,\dots)$,
$$(\od (x))_i=\left\{
\begin{array}{ll}
0&\textrm{if }i<l(x)\\
x_i+1&\textrm{if }i=l(x)\\
x_i&\textrm{if }i>l(x)
\end{array}\right.
$$
where $l(x)$ is the smallest index $l$ such that $x_l\neq m_l-1.$ In other words, if $x=(m_1-1,\dots m_{l-1}-1, x_l, x_{l+1}, x_{l+2},\dots )$ with $x_l< m_l-1$ (where the initial sequence of $m_i-1$'s may be empty), then $\od(x)=(0,\dots ,0, x_l+1, x_{l+1}, x_{l+2}, \dots)$. The map $\od$ is a homeomorphism of $\Omega$; and in fact an isometry for the distance $d$ defined above since for any $l\geq 1$, the first $l$ coordinates of $\od(x)$ depends only on the first $l$ coordinates of $x$. 
For future reference, we quote the following basic fact concerning the iterates of $\od$. 
\begin{fact}\label{addition} Let $M_1=1$ and $M_i=m_1\cdots m_{i-1}$ for $i\geq 2$. Any integer $k\geq 0$ can be uniquely written as a finite sum $k=\sum_{i\geq 1} k_i M_i$, where $0\leq k_i\leq m_i-1$ for all $i$; and with this notation, we have for all $x\in \Omega$:
\[ \od^k(x)=x\pmb+(k_1, k_2, \dots ).\]
\end{fact}

\smallskip
Now, for each $i\geq 1,$ we consider a probability measure $\mu_i$  on $\Omega_i$.  Writing $\mu_i(j)$ instead of $\mu_i(\{ j\})$, we assume that $\mu_i(j)>0$ for all $j\in \Omega_i$, 
and we denote by $\mu$ the product measure $\otimes_{i\geq 1}\mu_i.$ We will always assume that the measure $\mu$ is \emph{non-atomic}, which means in the present context that $\mu(\{ x\})=0$ for every $x\in\Omega$.  Equivalently, \[ \prod_{i=1}^{\infty}\max\bigl\{\mu_i(j):\ j\in \Omega_i\bigr\}=0.\]

Denoting by $\mathcal B$ the Borel $\sigma$-algebra of $\Omega$, the system $(\Omega,\mathcal B,\mu,\od)$ is called an {\bf odometer}. It is well known that any odometer $(\Omega,\mathcal B,\mu,\od)$ is nonsingular, invertible and conservative;  see \cite[Theorem 9.9]{Hawk}. Since this is not obvious, and since nonsingularity is essential for the composition operator $C_\od$ to be well-defined, we will outline a proof below.

\smallskip
Let us now discuss the boundedness of the composition operator $C_\od$ on $\Lpmu$, $1\leq p<\infty$. 

Given $B_1\subset \Omega_1,\dots,B_n\subset \Omega_n,$ we denote by $[B_1,\dots,B_n]$ the {\bf cylinder set} 
\[ [B_1,\dots,B_n]:=\{x\in\Omega:\ x_1\in B_1,\dots, x_n\in B_n\}.\]
If $B_1=\{x_1\},\dots,B_n=\{x_n\}$, we simply write $[x_1,\dots,x_n]$, and  we say that  $[x_1,\dots,x_n]$ is a {\bf basic cylinder set}.
Observe that if $[x_1,\dots,x_n]$ is a basic cylinder set and $(x_1,\dots ,x_n)\neq (0,\dots ,0)$, then 
$$\od^{-1}([x_1,\dots,x_n])=[m_1-1,\dots,m_{l(x)-1}-1,x_{l(x)}-1,x_{l(x)+1},\dots,x_n]$$
where $l(x)$ is the smallest index $l$ such that $x_l>0$; whereas
$$\od^{-1}([0,\dots,0])=[m_1-1,\dots,m_n-1].$$
Hence, there exists some constant $K<\infty$ such that $\mu(\od^{-1}(B))\leq K\mu(B)$ for all basic cylinder sets $B$ if and only if
\begin{equation}\label{eq:continuityodometer}
\sup_{l\geq 1}\left[ \prod_{i=1}^{l-1}\frac{\mu_i(m_i-1)}{\mu_i(0)}\times \sup_{j\in \Omega_l} \frac{\mu_l(j-1)}{\mu_l(j)}\right]<\infty,
\end{equation}
where $j-1$ has to be understood as $m_l-1$ if $j=0.$ Since every open subset of $\Omega$ is a countable disjoint union of basic cylinder sets, and since (for a given $K$) the family of all sets $B\in\mathcal B$ for which inequality \eqref{eq:continuitycompositionoperator} holds true is closed under decreasing countable intersections, it follows from the outer regularity of the measures $\mu$ and $\mu\od^{-1}$ that \eqref{eq:continuityodometer} characterizes the boundedness of $C_\od$ on $\Lpmu$.

\medskip

In Section \ref{sec:odometers}, we study in some detail the dynamics of composition operators induced by odometers on $L_p$, in the spirit of  \cite{BDDD22} and \cite{DGV}. Let us recall two of the main results of \cite{BDDD22}, which  relate hypercyclicity and mixing of $C_\od$ to the asymptotic behaviour of the sequence $(\eta_i)_{i\geq 1}$ defined by  
\[ \eta_i:=\max\bigl\{\mu_i(j):\ j\in \Omega_i\bigr\}.\]


\begin{theorem}\label{EmmaUdayan} Assume that $C_\od$ is bounded on $\Lpmu$, $1\leq p<\infty$.
\begin{enumerate}
\item[\rm (1)] If $\limsup_{i\to\infty}  \eta_i=1$, then $C_\od$ is hypercyclic.
\item[\rm(2)] If $\eta_i\to 1$ as $i\to\infty$, then $C_\od$ is topologically mixing; and 
the converse is true provided the sequence $(m_i)$ is bounded.
\end{enumerate}
\end{theorem}


In the present paper, using probabilistic arguments, we shall provide another sufficient condition for hypercyclicity, which will allow us to completely characterize the hypercyclic composition operators induced by odometers with the same measure on each component.  We will also characterize topological mixing for $C_\od$ without any assumption on the sequence $(m_i)$, and this will enable us to give interesting new examples. Finally we shall  investigate frequent hypercyclicity, which  will lead in particular to examples of composition operators which are frequently hypercyclic, chaotic and topologically mixing, as well as examples that are frequently hypercyclic and chaotic but \emph{not} topologically mixing.




\subsection{Odometers, continued} For the sake of completeness, we outline a proof of the fact that any odometer $(\Omega,\mathcal B,\mu,\od)$ with a non-atomic measure $\mu$ is nonsingular, invertible and conservative. More precisely, we prove the following lemma, where $0=(0,0,\dots)\in\Omega$.

\begin{lemma} Let $\mu$ be an arbitrary product probability measure on $\Omega=\prod_{i=1}^\infty \Omega_i$. The system $(\Omega,\mathcal B,\mu,\od)$ is nonsingular if and only if $\mu(\{ 0\})>0$ or $\mu(\{ \od^{-1}(0)\})=0$; and it is $($nonsingular and$\,)$ invertible if and only if  $\mu(\{ 0\})=0=\mu( \{ \od^{-1}(0)\}$ or $\mu(\{ 0\})\,\mu( \{ \od^{-1}(0)\})>0$. Finally, if the measure $\mu$ is non-atomic then $(\Omega,\mathcal B,\mu,\od)$ is conservative.
\end{lemma}

\begin{proof} If $\mu(\{ 0\})=0$ and $\mu(\{\od^{-1}(0)\})>0$ then $(\Omega,\mathcal B,\mu,\od)$ is obviously singular. Conversely, assume that $\mu(\{ \od^{-1}(0)\})=0$ or $\mu(\{ 0\})>0$, and let us show that  $(\Omega,\mathcal B,\mu,\od)$ is nonsingular. We have to prove that the measure $\mu \od^{-1}$ is absolutely continuous with respect to $\mu$, which we do  by directly exhibiting the Radon-Nikodym derivative. 

Let $\odinv=\od^{-1}$. For any $x\neq 0= (0,0,\dots)$ in $\Omega$, denote again by $l(x)$ the smallest index $l$ such that $x_l>0$. Then \[\odinv(x)=(m_1-1,\dots m_{l(x)-1}-1, x_{l(x)}-1, x_{l(x)+1}, x_{l(x)+2},\dots ).\] It follows that 
\[ h(x):= \prod_{i=1}^\infty \frac{\mu_i\bigl( \odinv(x)_i\bigr)}{\mu_i(x_i)}\]
is well defined for all $x\neq 0$ since the product is finite in that case. It is also well-defined for $x=0$ if $\mu(\{ 0\})>0$, namely $h(0)=\frac{\mu (\{ \odinv(0)\})}{\mu(\{ 0\})}\cdot$  In any case, $h(x)$ is well-defined $\mu\,$-$\,$almost everywhere, and 
\[ \forall x\neq 0\;:\; h(x)=\prod_{i=1}^{l(x)-1}  \frac{\mu_i\bigl( m_i-1\bigr)}{\mu_i(x_i)}\times \frac{\mu_{l(x)}(x_{l(x)}-1)}{\mu_{l(x)}(x_{l(x)})}\cdot\]
So $h$ is a real-valued, nonnegative mesurable function. We show that $\mu\od^{-1}=h\mu$. 

We introduce some notation. For any $l\geq 1$, let $\Omega_{| l}=\prod_{i=1}^l \Omega_i$; and for $x\in \Omega$, let $x_{| l}=(x_1,\dots ,x_l)\in \Omega_{| l}$. Since the $l$ first coordinates of $\od(x)$ or $\odinv(x)$ depend only on the $l$ first coordinates of $x$, the maps $\od$ and $\odinv$ induce bijections 
$\od_l$ and $\odinv_l$  of $\Omega_{| l}$ onto itself, namely $\od_l(x_{| l})=\od(x)_{| l}$ and $\odinv_l(x_{| l})=\odinv(x)_{| l}$ for all $x\in\Omega$. Moreover, if we define 
\[ E_l=\{ x\in\Omega\,:\, l(x)\leq l\}=\{ x\in\Omega\,:\, x_{| l}\neq (0,\dots 0)\},\]
then \[ h(x)=\prod_{i=1}^l \frac{\mu_i\bigl( \odinv(x)_i\bigr)}{\mu_i(x_i)}=\frac{\mu_{| l}(\{ \odinv_l(x_{| l})\})}{\mu_{| l}(\{ x_{| l}\})}\quad\hbox{for all $x\in E_l$},\]
where $\mu_{| l}=\otimes_{i=1}^l \mu_i$. In particular, we may write $h(x)=h(x_{| l})$. 

Now let us show that $\mu\od^{-1}=h\mu$. For that, it is enough to show that if $f:\Omega\to \RR$ is a nonnegative measurable function depending on finitely many coordinates, then 
\[ \int_\Omega f h\, d\mu=\int_\Omega f\circ\od \, d\mu.\] 

Assume that $f(x)$ depends only on the first $N$ coordinates of $x$. Then, for any $l\geq N$, we may write $f(x)=f(x_{| l})$; and we have
\begin{align*}
\int_{E_l} fh\, d\mu&=\int_{E_l} f(x_{| l})h(x_{| l})\, d\mu(x)\\
&=\sum_{u\in \Omega_{| l}\setminus\{ 0_{| l}\}} f(u)h(u)\, \mu_{| l}(\{ u\})\\
&=\sum_{u\in \Omega_{| l}\setminus\{ 0_{| l}\}} f(u)\, \mu_{| l}(\{ \odinv_l(u)\})\\
&=\sum_{v\in \Omega_{| l}\setminus\{ \odinv(0)_{| l}\}} f(\od_l(v)) \mu_{| l}(\{ v\})\\
&=\int_{F_l} f\circ\od\, d\mu,
\end{align*}
where $F_l=\odinv(E_l)=\{ y\in\Omega\,:\, y_{| l}\neq (m_1-1,\dots ,m_l-1)\}$. 

Since the sequence $(E_l)$ is increasing with $\bigcup_{l\geq 1} E_l=\Omega\setminus\{ 0\}$, so that $\bigcup_{l\geq 1} F_l=\Omega\setminus\{ \odinv(0)\}$, 
it follows that \[ \int_{\Omega\setminus\{ 0\}} f h\, d\mu=\int_{\Omega\setminus\{ \odinv(0)\}} f\circ\od \, d\mu.\] 
Now, recall that we are assuming that either $\mu(\{ 0\})=0=\mu(\{\odinv(0)\}$ or $\mu(\{ 0\})>0$. If $\mu(\{ 0\})=0=\mu(\{\odinv(0)\}$, we directly get $\int_\Omega  f h\, d\mu=\int_{\Omega} f\circ\od \, d\mu$. If $\mu(\{ 0\})>0$, then $\int_{\{ 0\}} fh\, d\mu=f(0)h(0)\mu(0)=(f\circ\od)(\odinv(0))\mu(\{ \odinv(0)\})=\int_{\{ \odinv(0)\}} f\circ\od\, d\mu$; so we also get $\int_\Omega f h\, d\mu=\int_{\Omega} f\circ\od \, d\mu$.

\smallskip Since the map $\od:\Omega\to\Omega$ is a homeomorphism, it is bimeasurable; and one shows in exactly the same way as above that $\odinv=\od^{-1}$ is nonsingular if $\mu( \{ \od^{-1}(0)\}>0$ or $\mu(\{ 0\})=0$. Hence, 
$(\Omega,\mathcal B,\mu,\od)$ is nonsingular and invertible if $\mu(\{ 0\})=0=\mu( \{ \od^{-1}(0)\}$ or $\mu(\{ 0\})\,\mu( \{ \od^{-1}(0)\})>0$.

\smallskip Finally, let us show that $(\Omega, \mathcal B,\mu, \od)$ is conservative if the measure $\mu$ is non-atomic. Since the system is nonsingular and $\mu$ is non-atomic, it is enough to show that the transformation $\od$ is \emph{ergodic} with respect to $\mu$, \textit{i.e.} that any $\od\,$-$\,$invariant measurable function is  $\mu\,$-$\,$almost everywhere constant; see \cite[Proposition 1.2.1]{Aar}.

For $n\geq 1$ and $x\in\Omega,$ let us set
\[ K_{n,x}=\sum_{i=1}^n M_ix_i,\]
where $M_1=1$ and $M_i=m_1\cdots m_{i-1}$ if $i\geq 2$. 
Then   $\od^{-K_{n,x}}(x)=(0,\dots,0,x_{n+1},\dots)$ by Fact \ref{addition}. It follows that if $f:\Omega\to\mathbb R$ is a measurable $\od\,$-$\,$invariant function, then $f(x)=f_n(\sigma^n(x))$ for all $x\in\Omega$, where $f_n(y)=f(0,\dots ,0, y_1,y_2,\dots)$ and $\sigma:\Omega\to\Omega$ is the backward shift. Hence, any $\od\,$-$\,$invariant measurable function is $\sigma^{-n}\mathcal (\mathcal B)$-measurable for all $n\geq 1$.  Now, since $\mu$ is a product measure, we have 
$\bigcap_{n\geq 1}\sigma^{-n}(\mathcal B)=\{\varnothing,\mathcal B\}$ mod $\mu$ 
by Kolmogorov's $0\,$-$\,1$ law.   
So every $\od\,$-$\,$invariant measurable function is indeed $\mu\,$-$\,$a.e. constant.
\end{proof}

\begin{remark}
In \cite{BDDD22}, the measure $\mu$ is not required to be non-atomic. We won't use this assumption either, but we prefer to keep it in order to emphasize that we are working with a conservative system. Note that in some of the examples given in  \cite{BDDD22}, the measure $\mu$ does have atoms; see  e.g. \cite[Theorem 3.2]{BDDD22}. We'll have to find similar examples with a non-atomic measure.
\end{remark}

\subsection{Diagonal translation operators}\label{subsec:sumoperators}
We shall also investigate the following variant of composition operators induced by odometers. We keep the same notation for $\Omega=\prod_{i\geq 1}\Omega_i$ and the product measure $\mu=\otimes_{i\geq 1}\mu_i$, but this time we consider the usual, coordinatewise addition on $\Omega$, which we denote by $+$. We take $a:=(1,1,\dots )$ and consider the ``diagonal'' translation $\trans:\Omega\to\Omega$ defined by 
\[ \trans(x)= a+x=(x_i+1)_{i\geq 1}.\]

The main difference with the odometer map $\od$ is that the measure $\mu\trans^{-1}$ is again a product measure, namely $\mu\trans^{-1}=\otimes_{i\geq 1} \widetilde\mu_i$ where $\widetilde\mu_i(j)=\mu_i(j-1)$, $j\in\Omega_i$. By Kakutani's theorem \cite{Kak}, it follows that $\trans:(\Omega,\mathcal B,\mu)\to (\Omega,\mathcal B,\mu)$ is nonsingular if and only if 
\[\prod_{i=1}^\infty \left( \sum_{j\in\Omega_i} \sqrt{\mu_i(j)\mu_i(j-1)}\right)>0.\]

As for boundedness of $C_\trans$, note that if $[x_1,\dots,x_n]$ is a basic cylinder set, then 
\[\trans^{-1}([x_1,\dots,x_n])=[x_1-1,\dots,x_n-1].\]
Hence, there exists some constant $K$ such that $\mu(\trans^{-1}(B))\leq K\mu(B)$ for all basic cylinder sets $B$ if and only if 
\begin{equation}\label{eq:continuity}
\prod_{i=1}^{\infty}\sup_{j\in\Omega_i}\frac{\mu_i(j-1)}{\mu_i(j)}<\infty;
\end{equation}
and by the same argument as for odometers, this condition characterizes the boundedness of $C_{\trans}$ on $\Lpmu$, $1\leq p<\infty$. Composition operators $C_{\trans}$ will be called \textbf{diagonal translation operators}.

The dynamics of $C_{\trans}$ will be sometimes simpler to analyze than that of an odometer, because we do not have to handle carries (compare for instance the boundedness conditions \eqref{eq:continuityodometer} and \eqref{eq:continuity}).
Observe nevertheless that if the sequence $(m_i)$ is bounded, then there exists an integer $N$ such that $\trans^N=\textrm{Id}$ and thus $C_{\trans}$ cannot be hypercyclic. Hence we need the sequence $(m_i)$ to be unbounded if we want to get ``interesting''  examples.

\subsection{Periodic points} It is easy to check directly (by looking at cylinder sets) that odometers $C_\od$ and translation operators $C_\trans$ have dense sets of periodic points. This can also be deduced from the following lemma. Recall that a topological space is said to be zero-dimensional if it has a basis consisting of clopen sets.

\begin{lemma}\label{periodicgeneral} Let $\Omega$ be a compact, zero-dimensional topological group, and let $a\in \Omega$. Let also $\mu$ be a Borel probability measure on $\Omega$, and assume that the left translation $\tau_a(x)=ax$ induces a bounded composition operator on $\Lpmu$, $1\leq p<\infty$. Then, the periodic points of $C_{\tau_a}$ are dense in $\Lpmu$.
\end{lemma}
\begin{proof} Let us denote by $\mathcal C$ the family of all clopen subsets of $\Omega$. Since $\Omega$ is zero-dimensional, $\mathbf E:={\rm span}\,\bigl( \mathbf 1_C:\; C\in\mathcal C\bigr)$ is dense in $\mathcal C(\Omega)$, the space of continuous functions on $\Omega$, by the Stone-Weierstrass Theorem. Hence, $\mathbf E$ is also dense in $\Lpmu$. So we only need to show that any $f\in\mathbf E$ is a periodic point of $C_{\tau_a}$; and for that, it is enough to prove the following: given $C_1,\dots ,C_r\in\mathcal C$, one can find an integer $d\geq 1$ such that $a^{-d} C_j=C_j$ for $j=1,\dots ,r$. Moreover, replacing $\Omega$ by $\Omega^r$ and $a$ by $(a,\dots ,a)\in\Omega^r$ and considering $C_1\times \cdots \times C_r$, we may in fact assume that $r=1$, \textit{i.e.} we have only one clopen set $C$.

Let us first observe that for any neighbourhood $V$ of $e$, the unit element of $\Omega$, one can find an integer $d\geq 1$ such that $a^{-d}\in V$. Indeed, the sequence $(a^{n})_{n\geq 0}$ has an accumulation point by compactness of $\Omega$, so one can find $n\neq n'$ such that $a^{n-n'}\in V\cap V^{-1}$.

 Now, by a classical result of van Dantzig \cite{vD} (see \cite{Wes} for a proof in English), $\Omega$ has a neighbourhood basis at $e$ consisting of clopen subgroups. Hence, by compactness, the clopen set $C$ is a finite union of right translates of clopen subgroups, say $C=\bigcup_{k=1}^r H_k b_k$. 
If we choose $d\geq 1$ such that  $a^{-d}\in H:=\bigcap_{k=1}^rH_k$, then $a^{-d}C=\bigcup_{k=1}^ra^{-d}H_kb_k=\bigcup_{k=1}^rH_kb_k=C.$  

\end{proof}

\subsection{Organization of the paper}
In Section \ref{sec:supercyclic}, we characterize, in a general setting, the 
composition operators which are $\Gamma_\alpha$-$\,$supercyclic for some $\alpha<\pi$, where $\Gamma_\alpha=\{ z\in\CC^*\,:\, \vert \arg(z)\vert\leq \alpha\}$. Moreover, we show that on $L_p\,$-$\,$spaces of non-atomic measures, supercyclic composition operators are always $\RR_+$-$\,$supercyclic. In Sections \ref{sec:fhc} and \ref{sec:ufhc}, we obtain sufficient conditions for a composition operator to be frequently hypercyclic or $\mathcal U$-frequently hypercyclic. These conditions are based on the existence of many periodic vectors, in the spirit of 
 \cite{GMM21}. Sections \ref{sec:odometers} and \ref{sec:sumoperator} are devoted to a  thorough study of  composition 
operators induced by odometers and of diagonal translation operators.


\section{Supercyclic composition operators} \label{sec:supercyclic}
\subsection{Framework} In this section, 
we consider a $\sigma$-finite measure space $(\Omega,\mathcal B,\mu)$ and  
 a   Banach space  $(\mathcal X,\|\cdot\|)$ with $\X\subset\Lomu$. 
  We fix  once and for all a nonsingular transformation $\varphi:(\Omega,\mathcal B,\mu)\to (\Omega,\mathcal B,\mu)$, and we assume that $C_\varphi$ defines a bounded operator on $\mathcal X.$ 
 
 \smallskip In \cite{BDP} and  \cite{GG25},  a set of four natural conditions on $\mathcal X$ is introduced to study the dynamical properties of $C_\varphi$ acting on $\X$. The framework
of \cite{GG25}  is \textit{a priori} more general than that of \cite{BDP}, so this is the one we will use. Note however that, as observed in \cite{GG25}, condition (C1) below is equivalent to condition (H3) in \cite{BDP}; and that if $\mathcal X$ is a  \textbf{lattice}, \textit{i.e} ($\vert f\vert \leq \vert g\vert$ and $g\in \mathcal X$) implies ($f\in\mathcal X$ and $\Vert f\Vert \leq \Vert g\Vert$), then condition (C2) below is easily seen to be equivalent to condition (H4) in \cite{BDP}. The four conditions given in \cite{GG25} read as follows.

\smallskip
\begin{itemize}
\item[(H1)] For any $E\in\mathcal B$ with finite measure, the indicator function ${\mathbf 1}_E\in\mathcal X.$
\item[(H2)] ${\rm span}\bigl\{ \mathbf 1_E\,:\, E\in\mathcal B,\ \mu(E)<\infty\bigr\}$ is dense in $\mathcal X$.
\item[(C1)] $\mu\bigl( \vert f\vert \geq 1\bigr)\to 0$ as  $\Vert f\Vert\to 0$, $f\in\mathcal X$.
\item[(C2)] $\Vert \mathbf 1_E\Vert\to 0$ as $\mu(E)\to 0$, $E\in\mathcal B.$   
\end{itemize}

\smallskip
In more abstract terms, (C1) means that the inclusion map $\mathcal X\xhookrightarrow{}\Lomu$ is continuous when $\Lomu$ is endowed with the topology of (global) convergence in measure. It can be formulated as follows: there is a function $\kappa_1:\RR_+\to [0,\infty]$ such that $\mu\bigl( \vert f\vert \geq 1\bigr)\leq \kappa_1(\Vert f\Vert)$ for all $f\in\X$ and $\kappa_1(t)\to 0$ as $t\to 0$. 

\smallskip
As for (C2), upon assuming (H1), it means that the map $E\mapsto \mathbf 1_E$ from $\mathcal B_f:=\{ E\in\mathcal B\,:\, \mu(E)<\infty\}$ into $\X$ is continuous at $\varnothing$, when $\mathcal B_f$ is  endowed with the (pseudo$\,$-)metric $d(A,B)=\mu(A\Delta B)$. Since $\Vert \mathbf 1_A-\mathbf 1_B\Vert  \leq  \Vert \mathbf 1_{A\setminus B}\Vert +\Vert\mathbf 1_{B\setminus A}\Vert$ for all $A,B\in\mathcal B_f$, this is equivalent to continuity on the whole of $\mathcal B_f.$ Observe also that (C2) can be formulated as follows: there is a function $\kappa_2:\RR_+\to [0,\infty]$ such that $\Vert \mathbf 1_E\Vert \leq \kappa_2\bigl( \mu(E)\bigr)$ for all $E\in\mathcal B_f$  and $\kappa_2(t)\to 0$ as $t\to 0$.
  
  \smallskip
The following remark  will not be needed, but it seems worth mentioning. Recall that a normed space $\X\subset\Lomu$ is said to have an  \emph{order-continuous norm} if, whenever $(f_n)_{n\in\NN}$ is a pointwise decreasing sequence in $\X$ such that $f_n\to 0$ almost everywhere, it follows that $\Vert f_n\Vert\to 0$. For example, $\Lpmu$ has an order continuous norm if $1\leq p<\infty$, but $L_\infty(\Omega,\mu)$ does not.
\begin{remark} Let  $\X\subset \Lomu$ be a normed space, and assume that $\X$ is a lattice.
\begin{enumerate}
\item If $\X$ is complete and $\mu\bigl( \vert f\vert\geq 1\bigr)<\infty$ for every $f\in\X$, then $\X$ satisfies (C1).
\item If $\X$ satisfies (H1) and has an order-continuous norm, then $\X$ satisfies (C2).
\end{enumerate} 
\end{remark}
\begin{proof} (1) Towards a contradiction, assume that $\X$ does not satisfy (C1). Then, one can find a sequence $(f_k)\subset\X$ and $\delta >0$ such that $\sum_{k=0}^\infty \Vert f_k\Vert <\infty$ and yet  $\mu\bigl( \vert f_k\vert \geq 1\bigr)\geq \delta$ for all $k\geq 0$. 
Let us set $B_k=\{ \vert f_k\vert \geq 1\}$, then $0\leq \mathbf 1_{B_k}\leq \vert f_k\vert$, so $\Vert \mathbf 1_{B_k}\Vert\leq \Vert f_k\Vert$ by the lattice property, and hence the series $\sum \mathbf 1_{B_k}$ converges to some $g\in\X$ since $\mathcal X$ is complete. By the lattice property again, we have $\sum_{k=0}^n \mathbf 1_{B_k}\leq g$ for all $n\geq 0$: indeed, setting $g_p=\sum_{k=0}^p \mathbf 1_{B_k}$, we have $g_p\geq g_n$ for all $p\geq n$, so $\vert g_p-g_n\vert=g_p-g_n$, hence $\vert g-g_n\vert =g-g_n$ since $g_p-g_n\to g-g_n$ and  $\vert g_p-g_n\vert \to \vert g-g_n\vert$ in $\X$. Let $A=\bigcup_{k\geq 0} B_k$. Since $\mathbf1_{\bigcup_{k=0}^n B_k}\leq \sum_{k=0}^n \mathbf 1_{B_k}\leq g$ for all $n\geq 0$, we have $\mathbf 1_A\leq g$. Hence $\mathbf 1_A\in \X$, so that $\mu(A)=\mu\bigl( \mathbf 1_A\geq 1\bigr)<\infty$. Moreover, $\Vert \mathbf 1_A\Vert \leq \Vert g\Vert \leq \sum_{k=0}^\infty \Vert \mathbf 1_{B_k}\Vert \leq \sum_{k=0}^\infty \Vert f_k\Vert$. Similarly, if we set $A_n=\bigcup_{k\geq n} B_k$, then $\Vert \mathbf 1_{A_n}\Vert \leq \sum_{k=n}^\infty \Vert f_k\Vert$ for all $n\geq 0$. By the lattice property (again), it follows that $\Vert \mathbf 1_{\bigcap_{n\geq 0} A_n}\Vert =0$, so that $\mu\big( \bigcap_{n\geq 0} A_n\bigr)=0$. But $A_n\supset B_n=\{ \vert f_n\vert \geq 1\}$, so $\mu(A_n)\geq \delta$ for all $n$. Since the sequence $(A_n)$ is decreasing and $\mu(A_0)=\mu(A)<\infty$, this implies that $\mu\big( \bigcap_{n\geq 0} A_n\bigr)\geq \delta$, a contradiction. 

(2) Assume that $\X$ does not satisfy (C2). Then, one can find $\delta >0$ and a sequence $(E_k)\subset\mathcal B$ such that $\sum_{k=0}^\infty \mu(E_k)<\infty$ and $\Vert \mathbf 1_{E_k}\Vert \geq\delta$ for all $k\geq 0$. For $n\geq 0$, let $F_n=\bigcup_{k\geq n} E_k$. Then $\mu(F_n)<\infty$, so $\mathbf 1_{F_n}\in\X$ by (H1). Morevover, $F=\bigcap_{n\geq 0} F_n$ is such that $\mu(F)=0$. By order-continuity, it follows that  $\mathbf 1_{F_n}\to 0$ in $\X$. But $\mathbf 1_{F_n}\geq \mathbf 1_{E_n}$, so $\Vert \mathbf 1_{F_n}\Vert \geq \delta$ for all $n$ by the lattice property, a contradiction\end{proof}

\smallskip In order to study supercyclicity of composition operators, we will need stronger versions of (C1) and (C2), that we call (C1') and (C2').

\begin{itemize}
\item[(C1')] $\mu\bigl( \vert f\vert \geq 1\bigr)\, \mu\bigl( \vert g\vert \geq 1\bigr)\to 0$ as $\Vert f\Vert\Vert g\Vert\to 0$. Equivalently, there exists a function $\kappa_1:\RR_+\to [0,\infty]$ such that 
$\mu\bigl( \vert f\vert\geq 1\bigl)\leq \kappa_1(\Vert f\Vert)$ for all $f\in\X$, and $\kappa_1(u)\kappa_1(v)\to 0$ whenever $uv\to 0$. 

\smallskip\item[(C2')]   $\Vert \mathbf 1_E\Vert\Vert \mathbf 1_F\Vert\to 0$ as $\mu(E)\mu(F)\to 0$. Equivalently, there exists a function $\kappa_2:\RR_+\to [0,\infty]$ such that $\Vert \mathbf 1_E\Vert \leq \kappa_2\bigl( \mu(E)\bigr)$ for all $E\in\mathcal B_f$ and $\kappa_2(u)\kappa_2(v)\to 0$ as $uv\to 0$.
\end{itemize}

For example, $\Lpmu$, $1\leq p<\infty$ satisfies (C1') with $\kappa_1(t)=t^p$ 
  by Markov's inequality; and it satisfies (C2')  with $\kappa_2(t)=t^{1/p}$. This implies in particular that if 
$\mathcal X$ embeds continuously in $\Lpmu$ for some $1\leq p<\infty$, then $\mathcal X$ satisfies (C1') with $\kappa_1(t)= \bigl( \Vert \imath\Vert t\bigr)^p$
, 
where $\imath$ is the inclusion map $\mathcal X\hookrightarrow{}\Lpmu$; and that  if $\Lpmu$ embeds continuously in $\mathcal X$ for some $1\leq p<\infty$, then $\X$ satisfies (C2'). 

\smallskip
\begin{remark}\label{CC'} Assume that $\mu(\Omega)<\infty$. Then (C1) is equivalent to (C1'), and (C2) is equivalent to (C2') if $\X$ is a lattice satisfying (H1).
\end{remark}
\begin{proof} We have $\mu\bigl( \vert f\vert \geq 1\bigr)\mu\bigl( \vert g\vert \geq 1\bigr)\leq \mu(\Omega)\times \min \Bigl[ \mu\bigl( \vert f\vert \geq 1\bigr), \mu\bigl( \vert g\vert \geq 1\bigr)\Bigr]$ for any $f,g\in\X$, so (C1) implies (C1'). And if $\X$ is a lattice, then $\Vert \mathbf 1_E\Vert\,\Vert \mathbf 1_F\Vert\leq \Vert \mathbf 1\Vert \times \min\bigl( \Vert \mathbf 1_E\Vert,\Vert \mathbf 1_F\Vert\bigr)$ for any $E,F\in\mathcal B$, so (C2) implies (C2').
\end{proof}

\smallskip
\begin{example}\label{Orlicz} Let $\Psi:\RR_+\to \RR_+$ be a nontrivial Orlicz function, and let $\X=L_\Psi(\Omega,\mathcal B,\mu)$ be the associated Orlicz space, endowed with the Luxemburg norm
\[ \Vert f\Vert=\inf\left\{ a>0\,:\, \int_\Omega \Psi\left(\frac{\vert f\vert}{a}\right) d\mu\leq 1\right\}.\]
\begin{enumerate}
\item[\rm (i)] The space $\X$ satisfies (C1) and (C2), hence (C1') and (C2') if $\mu(\Omega)<\infty$.
\item[\rm (ii)] The space $\X$ satisfies  {\rm (C1')} if $\Psi(x)\Psi(y)\to \infty$ as $xy\to \infty$, and it satisfies {\rm (C2')}  if $\Psi^{-1}(x)\Psi^{-1}(y)\to\infty$ as $xy\to\infty$.
\end{enumerate}
\end{example}
\begin{proof} The nontriviality assumption means that there exists $\alpha\geq 0$ such that $\Psi$ is strictly increasing on $[\alpha,\infty)$, so that $\Psi^{-1}:[0,\infty)\to [\alpha, \infty)$ is well defined. Note also that the second part of (i) follows from Remark \ref{CC'}.

 If $f\in L_\Psi$ then, by Markov's inequality, 
\[ \mu\bigl( \vert f\vert \geq 1\bigr)\leq \mu\left[ \Psi\left(\frac{\vert f\vert}{\Vert f\Vert}\right)\geq \Psi\left(\frac1{\Vert f\Vert}\right)\right]\leq \frac1{\Psi(1/\Vert f\Vert)}\cdot\]
This shows that $L_\Psi$ satisfies (C1); and that it satisfies  (C1') if $\kappa_1(t)=1/\Psi(1/t)$ 
 is such that $\kappa_1(u)\kappa_1(v)\to 0$ as $uv\to 0$, \textit{i.e.} $\Psi(x)\Psi(y)\to\infty$ as $xy\to\infty$.

If $E\in\mathcal B$ has finite measure then, fo any $a>0$, 
\[ \int_\Omega \Psi\left(\frac{\mathbf 1_E}{a}\right) d\mu =\Psi\left(\frac1a\right) \mu(E),\]
so that 
\[ \Vert \mathbf 1_E\Vert=\frac1{\Psi^{-1}\bigl(1/\mu(E)\bigr)}\cdot\]
This shows that $L_\Psi$ satisfies (C2), and that it satisfies  (C2') if $\Psi^{-1}(x)\Psi^{-1}(y)\to\infty$ as $xy\to \infty$.
\end{proof}

\subsection{Transitivity and mixing} The characterizations of topological transitivity and mixing for $C_\varphi$ obtained in \cite{BDP} rely on ``runaway like'' properties of the transformation $\varphi:(\Omega,\mathcal B,\mu)\to (\Omega,\mathcal B,\mu)$. Following \cite{GG25}, let us say that $\varphi$ satisfies the {\bf Hypercyclic Runaway Condition} (HRC) if for any $A\in\mathcal B$ with finite measure and for every $\veps>0,$ there exist  $n\geq 1$ and $B\in\mathcal B$ such that $B\subset A,$ $\mu(A\backslash B)<\veps$,  $\mu^*(\varphi^n(B))<\veps$ and $\mu(\varphi^{-n}(B))<\veps.$ Here, $\mu^*(E)$ denotes the ``outer measure'' of a set $E\subset\Omega$, \textit{i.e.} $\mu^*(E)=\inf\,\bigl\{ \mu(F)\,:\, F\supset E\,,\, F\in\mathcal B\bigr\}$. 
Similarly, we say that $\varphi$ satisfies the {\bf Mixing Runaway Condition} (MRC) if for any $A\in\mathcal B$ with finite measure and for every $\veps>0,$ there exists  $n_0\geq 1$ such that for all $n\geq n_0,$ one can find $B\in\mathcal B$ such that $B\subset A,$ $\mu(A\backslash B)<\veps$,  $\mu^*(\varphi^n(B))<\veps$ and $\mu(\varphi^{-n}(B))<\veps.$
Finally, we say that $\varphi^{-1}(\mathcal B)$ is {\bf essentially equal} to $\mathcal B$, and we write $\essequal$, if for any $A\in\mathcal B,$ there exists $B\in\mathcal B$ such that $\mu\bigl(A\Delta \varphi^{-1}(B)\bigr)=0.$ The following results are proved in \cite{BDP,GG25}.

\begin{theorem} Assume that the Banach space $\X\subset\Lomu$ satisfies conditions {\rm (H1), (H2), (C1), (C2)}. Then, 
 $C_\varphi:\X\to\X$ is topologically transitive if and only if $\essequal$ and $\varphi$ satisfies {\rm (HRC)}; and 
 $C_\varphi$ is topologically mixing if and only if $\essequal$ and $\varphi$ satisfies {\rm (MRC)}.
\end{theorem}

\smallskip
When the measure $\mu$ is finite and the transformation $\varphi$ is one-to-one and bimeasurable (\textit{i.e.} $\varphi(B)\in\mathcal B$ for every $B\in\mathcal B$), these characterizations have the following simpler formulations.

\begin{corollary}\label{lem:topologicallytransitive} Assume that $\X$ satisfies  {\rm (H1), (H2), (C1), (C2)}, that $\mu(\Omega)<\infty$, and that the transformation $\varphi$ is  one-to-one and bimeasurable.
\begin{enumerate}
\item[\rm (1)] $C_\varphi$ is topologically transitive if and only if $\varphi$ satisfies any of the following three equivalent conditions:

\smallskip\begin{enumerate}
\item[\rm (i)] for every $\veps>0,$ there exist $n\in\NN$ and $B\in\mathcal B$ such that 
$\mu(\Omega\backslash B)< \veps$ and $B\cap \varphi^n(B)=\varnothing;$
\item[\rm (ii)] for every $\veps>0$, there exist $n\in\NN$ and $B\in\mathcal B$ such that 
$\mu(B)< \veps$ and $ \mu(\Omega\backslash \varphi^{-n}(B))< \veps;$
\item[\rm (ii')] for every $\veps>0$, there exist $n\in\NN$ and $B\in\mathcal B$ such that 
$\mu(\Omega\backslash B)< \veps$ and $ \mu( \varphi^{n}(B))< \veps.$
\end{enumerate}

\smallskip \item[\rm (2)] $C_\varphi$ is topologically mixing if and only if $\varphi$ satisfies any of the following three equivalent conditions:

\smallskip
\begin{enumerate}
\item[\rm (i)] for every $\veps>0,$ there exists $n_0\in\NN$ such that, for all $n\geq n_0,$ there exists $B\in\mathcal B$ with
$\mu(\Omega\backslash B)<\veps$ and  $B\cap \varphi^n(B)=\varnothing;$
\item[\rm (ii)] for every $\veps>0$, there exists $n_0\in\NN$ such that, for all $n\geq n_0$, there exists $B\in\mathcal B$ such that 
$\mu(B)<\veps$ and $ \mu(\Omega\backslash \varphi^{-n}(B))< \veps;$
\item[\rm (ii')] for every $\veps>0$, there exists $n_0\in\NN$ such that, for all $n\geq n_0$, there exists $B\in\mathcal B$ such that 
$\mu(\Omega\backslash B)< \veps$ and $ \mu(\varphi^{n}(B))< \veps.$
\end{enumerate}
\end{enumerate}
\end{corollary} 
\begin{proof}
(1) The equivalence of (i) and topological transitivity is \cite[Corollary 1.3]{BDP}. The equivalence of (ii) and (ii') follows by taking $\widetilde B=\varphi^n (B)$.


The implication (i)$\implies$(ii') is obvious. Conversely if (ii) is satisfied, let $\veps>0$, and choose $n\in\NN$ and $B\in\mathcal B$ such that 
$\mu(\Omega\backslash \varphi^{-n}(B))< \veps/2$ and $\mu(B)< \veps/2.$ Set $\widetilde B:=\varphi^{-n}(B)\backslash B.$ 
Then $\mu\bigl(\Omega\backslash\widetilde B\bigr)< \veps$ and $\varphi^n\bigl(\widetilde B\bigr)\cap\widetilde B\subset B\backslash B=\varnothing.$ 
This shows that (ii)$\implies$(i).

\smallskip
(2) The equivalence of (i) and topological mixing is \cite[Corollary 2.2]{BDP}; and the equivalence of (i), (ii) and (ii') is proved as above.
\end{proof}

\smallskip To conclude this section, we quote the following result from  \cite{Whi}. 
\begin{lemma}\label{ess} Assume that $\X\subset\Lomu$ satisfies {\rm (H1)} and {\rm (C1)}. If $C_\varphi:\X\to\X$ has dense range, then $\essequal$. 
\end{lemma}

In  \cite[Lemma 1]{Whi}, this is stated for $\X=L_2$, but the proof works as soon as $\X$ satisfies (H1) and any convergent sequence in $\X$ has a subsequence converging almost everywhere, which is guaranteed by (C1).

\subsection{{Supercyclicity}}
From now on, we assume that the Banach space $\X\subset\Lomu$ is separable.

\begin{definition}
We say that the transformation $\varphi:(\Omega,\mathcal B,\mu)\to (\Omega,\mathcal B,\mu)$ satisfies the {\bf Supercyclic Runaway Condition} (SRC) if for every $A\in\mathcal B$ with finite measure and 
for every $\veps>0,$ there exist $B\in\mathcal B$ and $n\geq 1$ such that $B\subset A,$ $\mu(A\backslash B)<\veps$ and $\mu^*(\varphi^n(B))\mu(\varphi^{-n}(B))<\veps.$
\end{definition}

\smallskip This condition has been introduced in \cite{DM24} (without giving it a name), and it is shown there that it characterizes $\RR_+$-$\,$supercyclicity of $C_\varphi$ on $L_p$. We are going to show that under suitable assumptions on $\X$, (SRC) characterizes $\Gamma_\alpha$-$\,$supercyclicity for any $\alpha<\pi$. 

\begin{theorem}\label{thm:sc1}
Suppose that $\mathcal X$ satisfies conditions {\rm (H1)} and {\rm (C1')}. If $C_\varphi$ is $\Gamma_\alpha$-$\,$supercyclic for some $\alpha<\pi$, then $\essequal$ and $\varphi$ satisfies {\rm (SRC)}.
\end{theorem}

\smallskip To prove this result we need the following simple geometric lemma.

\begin{lemma}\label{lem:geometric}
Given $0\leq \alpha<\pi$, there exists $b>0$ such that the following holds true for any $z\in\CC$ and any $\lambda\in \Gamma_\alpha$.
\begin{itemize}
\item[(a)] If $ |z-b|\leq 1$, then $\vert \lambda z\vert \geq \vert\lambda\vert$ and $\vert \lambda z+b\vert \geq \vert\lambda\vert$.
\item[(b)] If $ |\lambda z+b|\leq 1$, then $\vert z\vert\geq 1/\vert\lambda\vert$ and $\vert z-b\vert \geq 1/\vert \lambda\vert$.
\end{itemize}
\end{lemma}
\begin{proof} 
(a) We first claim that there exists some constant $c_\alpha >0$ such that $\vert \lambda +1\vert \geq c_\alpha \vert \lambda\vert$ for every $\lambda\in\Gamma_\alpha$. Indeed, if we write $\lambda =re^{i\theta}$ where $r\geq 0$ and $\theta\in [-\alpha,\alpha]$, then $\vert \lambda+1\vert^2=1+2r\cos(\theta)+r^2\geq 1+2cr+r^2$, where $c:=\cos(\alpha)$. Since $-1<c\leq 1$, the polynomial $P(r)=1+2cr+r^2$ doesn't vanish on $\RR_+$, so it is bounded from below by some constant $c_\alpha^2 >0$.  Since $P(r)/r^2=P(1/r)$, the claim follows. Now, let us look for conditions on $b$ ensuring (a). Let $z\in \CC$ be such that $\vert z-b\vert\leq 1$, and let $\lambda\in \Gamma_\alpha$. If $b\geq 2$ then $\vert z\vert\geq 1$, so $\vert \lambda z\vert\geq \vert \lambda\vert$. To estimate $\vert \lambda z +b\vert$, write $z=b+u$ where $\vert u\vert \leq1$. Then $\lambda z+b=(1+\lambda)b +\lambda u$, so $\vert \lambda z +b\vert \geq \vert \lambda +1\vert\, b -\vert\lambda\vert\geq (c_\alpha b-1)\, \vert\lambda\vert$. Therefore, (a) is satisfied if we choose $b\geq 2$ such that $ c_\alpha b\geq 2$.


\smallskip The proof of (b) is similar.  If $\vert \lambda z+b\vert \leq 1$, then $\vert \lambda z\vert\geq 1$ provided $b\geq 2$; and writing $\lambda z =-b+u$ where $\vert u\vert\leq 1$, we get 
$ \vert \lambda (z-b)\vert =\vert -b(\lambda +1) +u\vert\geq \vert \lambda +1\vert b-1\geq d_\alpha b-1$, 
for some constant $d_\alpha >0$. So (b) is satisfied if $b\geq 2$ and $d_\alpha b\geq 2$. 
\end{proof}

\smallskip
\begin{proof}[Proof of Theorem \ref{thm:sc1}]
Since $C_\varphi$ is supercyclic, it has dense range, so $\essequal$ by Lemma \ref{ess}. 

Let us fix $A\in\mathcal B$ with finite measure and $\varepsilon>0$.  We need to find $B\subset A$ and $n\in\NN$ in accordance with (SRC).

Let us fix $b$ as in Lemma \ref{lem:geometric}, and let $\delta\in(0,1)$ to be chosen later. 
Since $C_\varphi$ is $\Gamma_\alpha$-$\,$supercyclic, one can find $f\in\mathcal X,$ $n\in\NN$ and $\lambda\in \Gamma_\alpha$ such that 
$$\|f-b\cdot {\mathbf 1}_A\|<\delta\quad\textrm{ and }\quad\|\lambda f\circ\varphi^n +b\cdot {\mathbf 1}_A\|<\delta.$$
We set 
$$C:=\{x\in\Omega:\ |f(x)-b|<1\}\;,\; \ D:=\{x\in\Omega:\ |\lambda f\circ\varphi^n(x)+b|<1\},$$
and 
\[ B:=A\cap C\cap D.\]

Observe that 
\begin{align*}
A\backslash B&\subset (A\backslash C)\cup (A\backslash D)\\
& \subset \bigl\{ |f -b\cdot {\mathbf 1_A}|\geq 1\bigr \}\cup \bigl\{|\lambda f\circ\varphi^n +b\cdot {\mathbf 1_A} |\geq 1\bigr\}.
\end{align*}
Hence, if $\kappa_1$ is the function given by (C1'), which may be assumed to be nondecreasing, we get
\[ \mu(A\setminus B)\leq 2\kappa_1(\delta)
.\]

We now observe that 
\[ \varphi^{-n}(B)\subset\bigl \{ \ |\lambda f\circ\varphi^n +b\cdot{\mathbf 1}_A |\geq|\lambda|\bigr\}\quad{\rm and}\quad \varphi^n(B)\subset \bigl\{  |f -b\cdot {\mathbf 1}_A |\geq 1/{|\lambda|}\bigr\}.\]

\noindent 
Indeed, if $x\in\varphi^{-n}(B)$, then  $\varphi^n(x)\in C,$ \textit{i.e.} $\vert f\circ\varphi^n(x)-b\vert <1$, and hence $\vert \lambda f\circ\varphi^n(x)+b\mathbf 1_A(x)\vert \geq \vert \lambda\vert$ by Lemma \ref{lem:geometric}. 
Similarly, 
if $y\in\varphi^n(B)$, then $y=\varphi^n(x)$ for some $x\in D$, so  $\vert \lambda f(y)+b\vert<1$, hence $\vert f(y)-b\mathbf 1_A(y)\vert \geq 1/\vert\lambda\vert$ by Lemma \ref{lem:geometric} again.

By (C1'), it follows that 
\[ \mu\bigl( \varphi^{-n}(B)\bigr)\, \mu^*\bigl( \varphi^n(B)\bigr)\leq \kappa_1\bigl( \delta/ \vert\lambda\vert\bigr)\,\kappa_1\bigl( \vert\lambda\vert\,\delta\bigr);\]
and hence we get (SRC) for $\varepsilon$ if $\delta$ is small enough. 
\end{proof}

\begin{remark}\label{Daniel} In an earlier version of the paper, we pretended to prove that (SRC) is necessary for \emph{supercyclicity} of $C_\varphi$. However, a serious mistake in the ``proof'' was pointed out to us by Daniel Aguilar Gomes, who actually explained to us that (SRC) is \emph{not} necessary for supercyclicity. For example, the operator $I_\CC$ can be wiewed as a supercyclic composition operator $C_\varphi$ on $L_p(\Omega, \mu)$ for a one-point space $\Omega=\{ a\}$, $\mu=\delta$ and $\varphi(a)=a$.  Obviously, the map $\varphi$ does not satisfy (SRC). For an infinite-dimensional example, take any hypercyclic composition operator $C_{\widetilde\varphi}$ on some $L_p(\widetilde\Omega, \widetilde\mu)$ and consider $T:=I_\CC\oplus C_{\widetilde\varphi}$, which can be viewed as the (supercyclic) composition operator $C_\varphi$ on $L_p(\widetilde\Omega\sqcup\{ a\}, \widetilde\mu+\delta_a)$, where $\varphi$ is the extension of $\widetilde\varphi$ such that $\varphi(a)=a$. We are very grateful to Daniel for pointing out our mistake.
\end{remark}

We now turn to the converse of Theorem \ref{thm:sc1}, which reads as follows.
\begin{theorem}\label{sc2}
Suppose that $\mathcal X$ satisfies {\rm (H2)} and {\rm (C2')}. If $\essequal$ and $\varphi$ satisfies {\rm (SRC)}, then $C_\varphi\oplus C_\varphi$ is $\mathbb R_+$-$\,$supercyclic.
\end{theorem}

\smallskip To prove this result, we will use the following version of the Supercyclicity Criterion.
Recall that a subset $\mathcal D$ of a Banach space $X$ is said to be \emph{total} if ${\rm span}(\mathcal D)$ is dense in $X$.
\begin{lemma}\label{SCC} Let $X$ be a separable Banach space, and let $T\in\mathfrak L(X)$. Assume that there exist two total sets $\mathcal D_1 , \mathcal D_2\subset X$, a sequence of integers $(n_k)_{k\geq 0}$ and, for each $k\geq 0$, maps $\alpha_k:\mathcal D_1\to X$ and $\beta_k:\mathcal D_2\to X$ such that \begin{enumerate}
\item[\sbt] $\alpha_k(x)\to x$  for all $x\in \mathcal D_1$ and  $T^{n_k} \beta_k(y)\to  y$ for all $y\in \mathcal D_2$;
\item[\sbt]  $\Vert T^{n_k}\alpha_k(x)\Vert\, \Vert \beta_k(y)\Vert\to 0$ for all $(x,y)\in\mathcal D_1\times \mathcal D_2$.
\end{enumerate}

Then $T\oplus T$ is $\RR_+$-$\,$supercyclic.
\end{lemma} 
\begin{proof} Since $T\oplus T$ satisfies the hypotheses of the lemma if $T$ does, it is enough to show that under the above assumptions, $T$ is $\RR_+$-$\,$supercyclic. Assuming, as we may, that the sets $\mathcal D_1$ and $\mathcal D_2$ are linearly independent, we can extend the maps $\alpha_k$ and $\beta_k$ by linearity to ${\rm span}(\mathcal D_1)$ and ${\rm span}(\mathcal D_2)$, and the extended maps still satisfy the above conditions. So, we may assume that $\mathcal D_1$ and $\mathcal D_2$ are in fact dense in $X$. Now, let $U$ and $V$ be two nonempty open sets in $X$. Choose $x\in\mathcal D_1\cap U$ and $y\in \mathcal D_2\cap V$. Since $\Vert T^{n_k}\alpha_k(x)\Vert\, \Vert \beta_k(y)\Vert\to 0$, we may find a sequence of positive real numbers $(\lambda_k)$ such that $\lambda_k T^{n_k}\alpha_k(x)\to 0$ and $\lambda_k^{-1} \beta_k(y)\to 0$. Then $z_k:= \alpha_k(x)+\lambda_k^{-1}\beta_k(y)\in U$ and $\lambda_kT^{n_k} z_k\in V$ if $k$ is large enough. By Birkhoff's transitivity theorem, it follows that $T$ is $\RR_+$-$\,$supercyclic.
\end{proof}

\begin{remark} 
As shown in \cite{BBP2}, $T\oplus T$ is supercyclic if and only if $T$ satisfies any of the known versions of the Supercyclicity Criterion.
\end{remark}

\smallskip
\begin{proof}[\it Proof of Theorem \ref{sc2}] 
Let $\mathcal D$ be the (total) subset of $\X$ consisting of all functions $\mathbf 1_D$ where $D\in\mathcal B$ has finite measure. We show that the assumptions of Lemma \ref{SCC} are satisfied with $\mathcal D_1=\mathcal D=\mathcal D_2$. 

\smallskip Let $(E_k)_{k\geq 0}$ be an increasing sequence of measurable sets with finite measure such that $\bigcup_{k\geq 0} E_k=\Omega$. By (SRC), one can find measurable sets $F_k\subset E_k$  and integers $n_k$ such that $\mu( E_k\setminus F_k)<2^{-k}$ and $\mu\bigl( \varphi^{-n_k}(F_k)\bigr) \mu^*\bigl(\varphi^{n_k}(F_k)\bigr)<2^{-k}$. We also choose measurable sets $\widetilde F_k$ such that $\varphi^{n_k}(F_k)\subset \widetilde F_k$ and $\mu(\widetilde F_k)=\mu^*\bigl( \varphi^{n_k}(F_k)\bigr)$.  

\smallskip We observe that 
\begin{equation}\label{debile} \mu(S\setminus F_k)\to 0\quad\hbox{for any $S\in\mathcal B$ with finite measure}.\end{equation}
Indeed, this is clear since $\mu(S\setminus F_k)=\mu\bigl(S\cap (E_k\setminus F_k)\bigr)+\mu(S\setminus E_k)\leq 2^{-k}+\mu(S\setminus E_k)$.

\smallskip Now, we define maps $\alpha_k, \beta_k:\mathcal D\to \X$ as follows. 

\begin{enumerate}
\item[-] For any $\mathbf 1_A\in\mathcal D$, we set
\[ \alpha_k(\mathbf 1_A)=\mathbf 1_{A\cap F_k}.\]

\item[-] For any $\mathbf 1_B\in \mathcal D$, we choose a sequence $(B_k)\subset \mathcal B$ such that $\mathbf 1_{B}=\mathbf 1_{\varphi^{-n_k}(B_k)}$ almost everywhere for every $k\geq 0$, which is possible since $\essequal$, and we set 
\[ \beta_k(\mathbf 1_B)=\mathbf 1_{B_k\cap \widetilde F_k}.\]
\end{enumerate}

Let $\kappa_2:\RR_+\to[0,\infty]$ be the function given by (C2'). We may assume that $\kappa_2$ is nondecreasing. 

For any $\mathbf 1_A\in \mathcal D$, we have 
\[ \Vert \mathbf 1_A-\alpha_k(\mathbf 1_A)\Vert =\Vert \mathbf 1_{A\setminus F_k}\Vert \leq \kappa_2\bigl( \mu(A\setminus F_k)\bigr),\]
so that $\alpha_k(\mathbf 1_A)\to \mathbf 1_A$ as $k\to\infty$ by (\ref{debile}).

For any $\mathbf 1_B\in \mathcal D$, we have
\[ C_\varphi^{n_k}\beta_k(\mathbf 1_B)= \mathbf 1_{\varphi^{-n_k}(B_k)\cap \varphi^{-n_k}(\widetilde F_k)}=\mathbf 1_{B\cap \varphi^{-n_k}(\widetilde F_k)}.\]
Since $\varphi^{-n_k}(\widetilde F_k)\supset F_k$ and $\kappa_2$ is nondecreasing, it follows that 
\[ \Vert \mathbf 1_B-C_\varphi^{n_k}\beta_k(\mathbf 1_B)\Vert \leq \kappa_2\bigl( \mu(B\setminus F_k)\bigr);\]
and hence $C_\varphi^{n_k}\beta_k(\mathbf 1_B)\to \mathbf 1_B$ by (\ref{debile}).

Finally, for any $\mathbf 1_A, \mathbf 1_B\in\mathcal D$, we have 
\begin{align*} \Vert C_\varphi^{n_k}\alpha_k(\mathbf 1_A)\Vert\, \Vert \beta_k(\mathbf 1_B)\Vert &= \Vert \mathbf 1_{\varphi^{-n_k}(A\cap F_k)}\Vert\,\Vert \mathbf 1_{B_k\cap \widetilde F_k}\Vert \\
&\leq \kappa_2\bigl( \mu(\varphi^{-n_k}(F_k)\bigr)\kappa_2\bigl( \mu(\widetilde F_k)\bigr) \\
&= \kappa_2\bigl( \mu(\varphi^{-n_k}(F_k)\bigr) \kappa_2\bigl( \mu^*(\varphi^{n_k}(F_k)\bigr),
\end{align*}
so that $\Vert C_\varphi^{n_k}\alpha_k(\mathbf 1_A)\Vert\, \Vert \beta_k(\mathbf 1_B)\Vert\to 0$. 

By Lemma \ref{SCC},  the proof of Theorem \ref{sc2} is now complete. 
\end{proof}

\smallskip 
\begin{remark}\label{HRC} 
Assuming that $\X$ satisfies (H2) and (C2), it is shown in \cite{GG25} that if $\essequal$ and $\varphi$ satisfies the Hypercyclic Runaway Condition (HRC), then $C_\varphi$ is weakly mixing. This result can be proved in exactly the same way as Theorem \ref{sc2}, by showing that $C_\varphi$ satisfies the appropriate version of the Hypercyclicity Criterion, \textit{i.e.}  Lemma \ref{SCC} where the condition ``$\Vert T^{n_k}\alpha_k(x)\Vert\, \Vert \beta_k(y)\Vert\to 0$'' is replaced by ``$\Vert T^{n_k}\alpha_k(x)\Vert\to 0$ and  $\Vert \beta_k(y)\Vert\to 0$''.
\end{remark}

Putting the two previous theorems together, we have obtained the following result.
\begin{corollary}
Suppose that $\mathcal X$ satisfies {\rm (H1)}, {\rm (H2)}, {\rm (C1')}, {\rm (C2')}. Then, the following are equivalent:
\begin{itemize}
\item[(i)] $C_\varphi$ is $\Gamma_\alpha$-$\,$supercyclic for some $\alpha<\pi$.
\item[(ii)] $C_\varphi\oplus C_\varphi$ is $\RR_+$-$\,$supercyclic.
\item[(iii)] $\essequal$ and $\varphi$ satisfies {\rm (SRC)}.
\end{itemize}
\end{corollary}

\smallskip Regarding this result, it is perhaps worth mentioning for any $\alpha\in (0,\pi)$, there exists a $\Gamma_\alpha\,$-$\,$supercyclic operator $T$ which is not $\RR_+$-$\,$supercyclic. For a $1$-dimensional example, take $T:=\omega I_\CC\in\mathcal L(\CC)$, where $\omega\in\TT$ is a root of unity of sufficiently large order. For an infinite-dimensional example, take $T:=\omega I_\CC\oplus R$ where $R$ is any hypercyclic operator. On the other hand, if $\omega\in\TT$ is not a root of unity and $R$ is hypercyclic, then $\omega I_\CC\otimes R$ is $\RR_+$-$\,$supercyclic by \cite{Shk}. However, as the next result shows, we cannot do better than working at a fix $\alpha.$ 

\begin{proposition} Let $T\in\mathfrak L(X)$ be $\Gamma_\alpha$-supercyclic for all $\alpha>0$. Then $T$ is $\mathbb R_+$-supercyclic.
\end{proposition}  
\begin{proof}
Assume that $T$ is not $\mathbb R_+$-supercyclic. Then, by a result of Shkarin \cite[Theorem 1.3]{Shk}, $\sigma_p(T^*)=\{re^{i2\pi p/q}\}$
 for some $r>0$ and $(p,q)\in\ZZ\times \NN$. Choose $x^*\neq 0$ with $T^* x^*=r e^{i2\pi p/q} x^*$ and fix $\alpha\in(0,1/q)$. Finally, let $x$ be a $\Gamma_\alpha$-supercyclic vector for $T$. Then, the set 
$\Lambda:=\bigl\{x^* (\mu T^k x) :\ k\geq 0, |\arg(\mu)|\leq \alpha\bigr\}$ has to be dense in $\CC$. However, 
$x^*(\mu T^k x)=\mu r^k e^{i2\pi kp/q}x^*(x),$ so that $\Lambda$ is contained in the union of $q$ angular sectors of aperture $2\alpha,$ which prevents it from being dense. 
\end{proof}

The moral of the previous results and remarks is that a supercyclic composition operator is not necessarily $\RR_+$-$\,$supercyclic, and that $\RR_+$-$\,$supercyclicity is essentially characterized by (SRC). Now, recall that for an operator $T$ such that $T^*$ has no eigenvalue, supercyclicity is equivalent to $\RR_+$-$\,$supercyclicity \cite{LeMu04}. So it is natural to look at the possible eigenvectors of $C_\varphi^*$. Our next theorem will do this on $L_p$.  

\smallskip Recall that a measurable set $E\subset \Omega$ is said to be an \emph{atom} for $\mu$ if, for any measurable set $B\subset E$, we have either $\mu(B)=0$ or $\mu(E\setminus B)=0$. Since $\mu$ is $\sigma$-finite, every atom has finite measure. Also, a measurable set $E\subset\Omega$ is said to be \emph{$\varphi^{-1}$-$\,$invariant} if $\varphi^{-1}(E)=E$ up to a set of measure $0$. 

\begin{theorem}\label{eigenv} Let $\X=\Lpmu$, $1\leq p<\infty$.
\begin{itemize}
\item[\rm (1)] If $E\subset\Omega$ is an atom for $\mu$ such that $E\subset\varphi^{-1}(E)$, then $g:=\mathbf 1_E\in L_q$ satisfies $C_\varphi^*g=g$.
\item[\rm (2)] Assume that $C_\varphi$ is supercyclic. If $C_\varphi^*$ has an eigenvalue, then this eigenvalue is equal to $1$ and $\ker(C_\varphi^*- 1)=\CC\mathbf 1_E$ for some $\varphi^{-1}$-$\,$invariant atom $E$. 
\end{itemize}
\end{theorem}

\smallskip For the proof, we will need the following lemma.
\begin{lemma}\label{circles}
For any $u, v\in\mathbb C\backslash\{0\}$ and $c>0$, one can find $\xi\in\TT$ such that 
$$\vert u+  v\vert\neq c\, \vert  u+\xi v\vert.$$
\end{lemma}
\begin{proof} Replacing $v$ by $-v$ and dividing by $\vert v\vert$, we have to show that for any $z\in\CC\setminus\{ 0\}$ and $c>0$, one can find $\xi\in\TT$ such that $\vert z-1\vert \neq c\, \vert z-\xi\vert$; which is obvious since the circle $\TT$, centered at $0$, is not contained in any circle centered at $z$.          
\end{proof}

\begin{proof}[Proof of Theorem \ref{eigenv}] In what follows, we denote by $\langle\,\cdot\, ,\cdot\,\rangle$ the $L_p\,$-$\,L_q$ duality pairing, where $q$ is the conjugate exponent of $p$; more precisely, for $(f,g)\in L_p\times L_q$, we set
\[ \langle g,f\rangle:=\int_\Omega fg\, d\mu.\]

(1) It is enough to show that $\langle C_\varphi^*\mathbf 1_E, \mathbf 1_B\rangle=\langle \mathbf 1_E, \mathbf 1_B\rangle$ for every measurable set $B$ such that $\mu(B)<\infty$, \textit{i.e.} that
\[\mu(E\cap \varphi^{-1}(B))= \mu(E\cap B).\]
Since $E$ is an atom, we have $\mu(E\cap B)=0$ or $\mu(E\cap B)=\mu(E)$. 
If $\mu(E\cap B)=0$, then $\mu(\varphi^{-1}(E)\cap \varphi^{-1}(B))=\mu(\varphi^{-1}(E\cap B))=0$, so $\mu(E\cap\varphi^{-1}(B))=0$ since $E\subset\varphi^{-1}(E)$. If $\mu(E\cap B)=\mu(E)$, \textit{i.e.} $\mu(E\setminus B)=0$, then $\mu(\varphi^{-1}(E)\setminus\varphi^{-1}(B))=0$, so $\mu(E\setminus\varphi^{-1}(B)$, \textit{i.e.} $\mu(E\cap\varphi^{-1}(B))=\mu(E)$.

\smallskip (2) Let us fix an eigenvector $g\in L_q$ for $C_\varphi^*$, with eigenvalue $\lambda$. We have to show that $\lambda=1$ and that $g=a\mathbf 1_E$ for some $a\in\CC$ and some $\varphi^{-1}$-$\,$invariant atom $E$. The key point is the following fact.
\begin{fact} There exists a constant $c>0$ such that 
\[ \forall h\in L_p\,:\,  \bigl\vert \langle g,   \vert h\vert\rangle \bigr\vert=c\,  \vert\langle g,     h \rangle \vert.\]
\end{fact}
\begin{proof} Let $f\in L_p$ be a supercyclic vector for $C_\varphi$. We note that $\langle g,f\rangle\neq 0$, since otherwise we would have $\langle g, C_\varphi^{n}f\rangle=\lambda^{n}\langle g, f\rangle=0$ for all $n\in\NN$ and $f$ would not be supercyclic. Moreover, we also have $\langle g, \vert f\vert\rangle\neq 0$. Indeed, assume that $\langle g, \vert f\vert\rangle=0$. Then $\langle g, C_\varphi^n\vert f\vert\rangle=\lambda^n \langle g,\vert f\vert\rangle=0$ for all $n\in\NN$. But $C_\varphi^n\vert f\vert=\vert f\circ\varphi^n\vert=\bigl\vert C_\varphi^n f\bigr\vert$, so we have $\langle g,  \vert C_\varphi^n f \vert\rangle=0$ for all $n\in\NN$. Since $f$ is a supercyclic vector for $C_\varphi$, it follows that $\langle g, \vert h\vert\rangle=0$ for every $h\in L_p$. In particular, $\langle g, \mathbf 1_B\rangle=0$ for every measurable set $B$ such that $\mu(B)<\infty$, so $g=0$, a contradiction. 

Now, let $h\in L_p$ be arbitrary. Choose a sequence of integers $(n_k)$ and a sequence of scalars $(\alpha_k)$ such that $\alpha_k C_\varphi^{n_k} f\to h$. Then, since  $\langle g, \alpha_k C_\varphi^{n_k} f\rangle=\alpha_k\lambda^{n_k}\,\langle g, f\rangle$, we see that $\alpha_k\lambda^{n_k}\to \langle g,h\rangle/\langle g,f\rangle$. Similarly, $ \vert \alpha_k\vert \lambda^{n_k}\, \langle g,\vert f\vert\rangle=\langle g, \vert\alpha_k\vert\, C_\varphi^{n_k} \vert f\vert \rangle=\langle g,  \vert\alpha_k \vert \, \vert C_\varphi^{n_k} f\vert\rangle\to \langle g, \vert h\vert\rangle/\langle g,\vert f\vert\rangle$. It follows that $\bigl\vert \langle g,h\rangle/\langle g,f\rangle\bigr\vert=\bigl\vert \langle g,\vert h\vert\rangle/\langle g,\vert f\vert\rangle\bigr\vert$, so $ \vert \langle g,     \vert h\vert \rangle \vert=c\,  \vert\langle g,   h \rangle \vert$ with $c:=\bigl\vert \langle g, \vert f\vert\rangle/\langle g,f\rangle\bigr\vert $. 
\end{proof}

Now, let us show that $E:=\{ g\neq 0\}$ is an atom for $\varphi$. Towards a contradiction, assume that one can find $A,B\subset E$ such that $A\cap B=\varnothing$ and $\mu(A), \mu(B)>0$.  Then, one can find $A'\subset A$ and $B'\subset B$ such that $\mu(A')<\infty$, $\mu(B')<\infty$, $u:=\int_{A'} g\, d\mu\neq 0$ and $v:=\int_{B'} g\, d\mu\neq 0$. By Lemma \ref{circles}, one can find $\xi\in\TT$ such that $\vert u+v\vert \neq c\, \vert u+\xi v\vert$. Let us consider 
\[ h:= \mathbf 1_{A'}+ \xi\mathbf 1_{B'},\]
so that 
\[ \vert h\vert =\mathbf 1_{A'}+\mathbf 1_{B'}.\]
Then
\[ \vert \langle g,\vert h\vert \rangle\vert =\vert u+v\vert \neq c\,\vert u+\xi v\vert =\bigl\vert\langle g,h\rangle\vert, \] 
which is the required contradiction.

\smallskip Since $g=g\mathbf1_E$ and $E$ is an atom, we see that $g=a\mathbf 1_E$ for some $a\in\CC$.  To conclude the proof, we have to show that $\lambda=1$ and that $E$ is $\varphi^{-1}$-$\,$invariant.

\smallskip Since $C_\varphi^*\mathbf 1_E=\lambda\mathbf1_E$, we have $\langle \mathbf 1_E, C_\varphi\mathbf 1_B\rangle=\lambda \langle \mathbf 1_E, \mathbf 1_B\rangle$ for every measurable set $B$ such that $\mu(B)<\infty$, \textit{i.e}
\[ \mu\bigl(E\cap \varphi^{-1}(B)\bigr)=\lambda\, \mu(E\cap B);\]
and since $\mu$ is $\sigma$-finite, this holds in fact for every measurable set $B$. Taking $B:=\Omega$, we get $\mu(E)=\lambda\,\mu(E)$, so $\lambda=1$. Hence, we now know that 
\[  \mu\bigl(E\cap \varphi^{-1}(B)\bigr)= \mu(E\cap B)\]
for every measurable set $B$.

\smallskip 
Taking $B:=E$, we get $\mu(E\cap \varphi^{-1}(E))=\mu(E)$, so $E\subset\varphi^{-1}(E)$ up to a measure $0$ set. It remains to prove the other ``inclusion'', \textit{i.e.} that $\mu(\varphi^{-1}(E)\setminus E)=0$. Before doing that, we observe that $\mu(E\setminus\varphi^{-n}(E))=0$ for every $n\geq 1$: this is easily proved by induction using the fact that $E\subset \varphi^{-1}(E)$ up to a measure $0$ set. Since $E$ is an atom, it follows that if $f:\Omega\to\CC$ is any measurable function, then there is some constant $c$ such that 
$\forall n\geq 0\,:\, f\circ\varphi^n=c$ a.e. on $E$; and this implies that $\forall n\geq 1\,:\, f\circ\varphi^n=c$ a.e. on $\varphi^{-1}(E)$.

 Now, towards a contradiction, assume that $\mu(\varphi^{-1}(E)\setminus E)>0$. Then, one can find a measurable set $A\subset\varphi^{-1}(E)$ such that $A\cap E=\varnothing$ and $0<\mu(A)<\infty$. Let $f$ be a supercyclic vector for $C_\varphi$, and choose $\alpha\in\CC$ and $n\in\NN$ such that $\alpha C_\varphi^nf$ is close to $\mathbf 1_E$. Let also $c$ be the constant given above. Then 
 \[ \langle \mathbf 1_E, \alpha C_\varphi^n f\rangle=\alpha\int_E f\circ\varphi^n d\mu=\alpha c\, \mu(E),\]
 and similarly, since $A\subset \varphi^{-1}(E)$,
 \[ \langle \mathbf 1_A, \alpha C_\varphi^n f\rangle=\alpha c\, \mu(A).\]
 Since $\alpha C_\varphi^nf$ is close to $\mathbf 1_E$ and $E\cap A=\varnothing$, it follows that $\alpha c$ is both close to $1$ and close to $0$, a contradction.
\end{proof}

\begin{corollary}\label{R+} Assume that $C_\varphi$ is supercyclic on $\Lpmu$. Then, the following are equivalent.
\begin{itemize}
\item[\rm (i)] $C_\varphi$ is not $\RR_+$-supercyclic;
\item[\rm (ii)] $C_\varphi^*$ has an eigenvalue;
\item[\rm (iii)] $\mu$ has a $\varphi^{-1}$-$\,$invariant atom.
\item[\rm (iv)] $C_\varphi$ is similar to an operator of the form $I_\CC\oplus T$.
\end{itemize} 
\end{corollary}
\begin{proof} The implication (i)$\implies$(ii) follows from \cite{LeMu04},   (ii)$\implies$(iii) is immediate by Theorem \ref{eigenv}, and (iv)$\implies$(i) is clear. To prove that (iii)$\implies$(iv), assume that $\mu$ has a $\varphi^{-1}$-$\,$invariant atom $E$. Then $C_\varphi \mathbf 1_E=\mathbf 1_E$, and $\varphi(\Omega\setminus E)\subset\Omega\setminus E$ up to a measure $0$ set. So, $C_\varphi$ is similar to $I_\CC\oplus T$, where $T=C_{\varphi_{| \Omega\setminus E}}$ acting on $L_p(\Omega\setminus E)$.
\end{proof}

\subsection{{Examples}}
In this section, we give two illustrations of the above results.

\smallskip
We first show that on a finite measure space, $\RR_+$-$\,$supercyclic and hypercyclic composition operators are the same, and  that if $\mu$ has no $\varphi^{-1}$-$\,$invariant atom, then every supercyclic $C_\varphi$ on $\Lpmu$ is in fact hypercyclic. The latter result is due to  U. Darji, D. Gomes and R. Var$\tilde{\rm a}$o \cite{DGV}, with a rather different proof. The idea of considering the invariant atoms comes from \cite{DGV}, and we are grateful to the authors for allowing us to use this idea, which led to Theorem \ref{eigenv}. 
\begin{example}
Assume that $\mu(\Omega)<\infty$.

\begin{itemize}  
\item[\rm (1)] If $\mathcal X$ satisfies {\rm (H1)}, {\rm (H2)}, {\rm (C1)}, {\rm (C2)}, the following are equivalent:

\smallskip
\begin{itemize}
\item[\rm (i)] $C_\varphi$ is $\Gamma_\alpha$-$\,$supercyclic for some $\alpha<\pi$.
\item[\rm (ii)] $C_\varphi$ is weakly mixing.
\item[\rm (iii)] $\essequal$ and for every $\veps>0,$ there exists $B\in\mathcal B$ and $n\in\NN$ such that $\mu(\Omega\backslash B)<\veps$ and $B\cap \varphi^n(B)=\varnothing.$
\end{itemize}

\smallskip
\item[\rm (2)] If $\X=\Lpmu$ and $\mu$ has no $\varphi^{-1}$-$\,$invariant atom, then $C_\varphi$ is supercyclic if and only if it is weakly mixing.
\end{itemize} 
\end{example}

\begin{proof} (1) Note that (iii) clearly implies the Hypercyclic Runaway Condition (HRC). 
So one just has to show that (i)$\implies $(iii), since $\X$ satisfies (H2) and (C2) -- see Remark \ref{HRC}.

Assume that $C_\varphi$ is $\Gamma_\alpha$-$\,$supercyclic for some $\alpha<\pi$, and let $\veps>0.$ Since $\X$ satisfies (H1) and (C1), which is equivalent to (C1') because $\mu(\Omega)<\infty$, the transformation $\varphi$ must satisfy (SRC). So there exist $B'\in\mathcal B$ and $n\in\NN$ such that $\mu(\Omega\backslash B')<\veps/2$ and 
$\mu^*(\varphi^n(B'))\mu(\varphi^{-n}(B'))<\veps^2/4.$ Choose a measurable set $C$  such that $\varphi^n(B')\subset C$ and 
$\mu(C)\mu(\varphi^{-n}(B'))<\veps^2/4$. Then either $\mu(C)<\veps/2$ or $\mu(\varphi^{-n}(B'))<\veps/2.$ Let $B=B'\backslash C$ in the first case, and $B=B'\backslash \varphi^{-n}(B')$ in the second case. 
Then $\mu(\Omega\backslash B)<\veps$ and $B\cap \varphi^n(B)=\varnothing$.

\smallskip
(2) This follows from (1) and Corollary \ref{R+}. 
\end{proof}

\smallskip
Next, we revisit the classical results of Salas \cite{Sal99} on supercyclicity of weighted shifts. Let $I=\ZZ$ or $\ZZ_+$,  let $(\nu_i)_{i\in I}$ be a sequence of positive numbers such that $\sup_{i\in I} \nu_i/\nu_{i+1}<\infty$, and let $\pmb{\rm B}:\ell_p(I,\nu)\to \ell_p(I,\nu)$ be the backward shift acting on the weighted space $\ell_p(I,\nu)=\{x\in\CC^I:\ \|x\|^p:=\sum_{i\in I} |x_i|^p \nu_i<\infty\}$. When $I=\ZZ$, it is shown in \cite{Sal99} that $\pmb{\rm B}$ is supercyclic if and only if for all $i\in \ZZ$, there exists an increasing sequence of integers $(n_k)$ such that $\nu_{i+n_k}\nu_{i-n_k}\to 0$. On the other hand, when $I=\ZZ_+$, $\pmb{\rm B}$ is always supercyclic, regardless of the sequence $(\nu_i)$. Since $\ell_p(I,\nu)=\Lpmu$ where $\Omega=I$, $\mathcal B=\mathcal P(I)$, $\mu=\sum_{i\in I} \nu_i\,\delta_{\{i\}}$ and $\pmb{\rm B}=C_\sigma$ where $\sigma:\Omega\to \Omega$ is the shift map, $\sigma(i)=i+1$, these results can be put in a slightly more general framework.



\begin{example} Let $\Omega$ be a countable {\rm (}infinite{\rm )} set, $\mathcal B=\mathcal P(\Omega)$, and let $\mu$ be a  positive measure on $(\Omega,\mathcal B)$ such that $0<\mu(i)<\infty$ for all $i\in\Omega$. Let also $\varphi :\Omega\to \Omega$, and assume that 
$\sup_{i\in\Omega} \mu\bigl(\varphi^{-1}(\{ i\})\bigr)/\mu(i)<\infty$. 

\begin{itemize}
\item[\rm (1)] $C_\varphi$ is $\RR_+$-$\,$supercyclic on $\Lpmu$, $1\leq p<\infty$ if and only if $\varphi$ satisfies the following condition $(*)$: for any finite set $B\subset\Omega$, one can find an increasing sequence of integers $(n_k)$ such that $\mu\bigl(\varphi^{n_k}(B)\bigr)\, \mu\bigl(\varphi^{-n_k}(B)\bigr)\to 0$.
\item[\rm (2)] If $\varphi$ has no fixed point
, this is equivalent to supercyclicity of $C_\varphi$.
\item[\rm (3)] Assume that the $\varphi\,$-$\,$orbits are pairwise intersecting. Then $(*)$ is equivalent to Salas' condition: for any $i\in \Omega$, one can find an increasing sequence of integers $(n_k)$ such that $\mu\varphi^{n_k}(i)\, \mu\varphi^{-n_k}(\{i\})\to 0$.
\end{itemize}


\end{example}
\begin{proof} (1) The condition $\sup_{i\in\Omega} \mu\bigl(\varphi^{-1}(\{ i\})\bigr)/\mu(i)<\infty$ ensures that $C_\varphi$ is bounded on $\Lpmu$; and (ii) is a restatement of (SRC) in the present setting.




\smallskip 
(2)  In the present setting a $\varphi^{-1}$-$\,$invariant atom is a fixed point of $\varphi$, so the result follows from Corollary \ref{R+}.

\smallskip
(3) Assume that $\varphi$ satisfies Salas' condition, and let us prove $(*)$. In what follows, we put $C:=\sup_{i\in\Omega} \mu\bigl(\varphi^{-1}(\{ i\})\bigr)/\mu(i)$.

We first note that any finite family of $\varphi\,$-$\,$orbits has a non-empty intersection. Indeed, if $F\subset\Omega$ is a finite set and if we take any $i\in\Omega$, then $\varphi^p(i)\in\bigcap_{j\in F} O(\varphi, j)$ if $p$ is large enough.


Now, let $B$ be a finite subset of $\Omega$, and let $i_0\in\bigcap_{j\in B} O(\varphi, j)$. We claim that for every $j\in B$, there exists some constant $K_j$ such that $\forall n\in\NN\,:\, \mu\varphi^n(j)\leq K_j \mu\varphi^n(\{i_0\})$ and $\mu \varphi^{-n}(\{j\})\leq K_j \mu \varphi^{-n}( \{i_0\})$. Indeed, we can find $r_j\in\NN$ such that $i_0=\varphi^{r_j}(j)$. Then $\varphi^n(j)\in \varphi^{-r_j}(\{ \varphi^n(i_0)\})$, so $\mu\varphi^n(j)\leq C^{r_j}\mu\varphi^n(i_0)$; and $\varphi^{-n}(\{ j\})\subset \varphi^{-r_j}\bigl(\varphi^{-n}(\{ i_0\})\bigr)$, so $\mu\varphi^{-n}(\{j\})\leq C^{r_j} \mu \varphi^{-n}(\{i_0\})$. If we now put $K:=\max_{j\in B} K_j$, we get 
$\mu\varphi^n(j)\, \mu\varphi^{-n}(\{ j'\})\leq K^2 \mu\varphi^n(i_0)\, \mu \varphi^{-n}( \{i_0\})$ for all $n\in\NN$ and $j,j'\in B$; and using Salas' condition for $i_0$, we deduce $(*)$ for $B$.
\end{proof}

\begin{remark} As the examples $I_\CC$ and $I_\CC\oplus C_{\widetilde\varphi}$ mentioned in Remark \ref{Daniel} show, there are supercyclic $C_\varphi$'s for which $\varphi$ has a fixed point.
\end{remark}



\section{Frequently hypercyclic composition operators} \label{sec:fhc}

The easiest way to prove that an operator is frequently hypercyclic is to use the so-called Frequent Hypercyclicity Criterion (see e.g. \cite[Theorem 6.18]{BM09}). This criterion gives an extremely strong conclusion since any operator satisfying it is frequently hypercyclic,  chaotic and topologically mixing.  
In the context of nonsingular measurable systems, a detailed study of which composition operators satisfy the Frequent Hypercyclicity Criterion has been made in \cite{DP21} (see e.g. Theorem 3.2). In particular it is shown that if $(\Omega,\mathcal B,\mu,\varphi)$ is a nonsingular  system such that $C_\varphi$ satisfies the Frequent Hypercyclicity Criterion on $L_p$, $p\geq 2$, then $\varphi$ has to be dissipative. Since, as observed in the introduction, odometers are never dissipative, it follows that there is no hope to apply the Frequent Hypercyclicity Criterion in the odometer setting. However, we have also observed that odometers have a wealth of periodic points; and this will allow us to apply another criterion for frequent hypercyclicity,  which appears in \cite[Theorem 5.35]{GMM21}. Here and afterwards, we denote by $\textrm{Per}(T)$ the set of all periodic points of an operator $T$, and by ${\rm per}(u)$ the period of a periodic point $u$, \textit{i.e.} the least $d\geq 1$ such that $T^du=u$.
\begin{lemma}\label{lem:fhc}
Let $X$ be a separable Banach space and let $T\in\mathfrak L(X)$. Assume that there exist a dense linear subspace $X_0$ of $X$ contained in ${\rm Per}(T)$ with $T(X_0)\subset X_0$ and a constant $\kappa\in(0,1)$ such that the following  holds true: for every $u\in X_0,$ every $\veps>0$ and every integer $d_0\geq 1$ which is the period of some vector in $X_0,$ there exist $v\in X_0$ 
and integers $n,d\geq 1$ such that 
\begin{enumerate}
\item[\rm (a)] $d$ is a multiple of $d_0$ and of ${\rm per}(v)$;
\item[\rm (b)] $\|T^k v\|\leq \veps$ for every $0\leq k\leq \kappa d$;
\item[\rm (c)] $\|T^{n+k}v-T^k u\|\leq \veps$ for every $0\leq k\leq\kappa d$.
\end{enumerate}
Then $T$ is frequently hypercyclic.
\end{lemma}

Using this criterion, we now give a ``runaway like'' sufficient condition for a composition operator to be frequently hypercyclic.  To put the result into a rather general framework, we need a new assumption. 
We still consider a $\sigma$-finite measure space $(\Omega,\mathcal B,\mu)$ and  
 a   Banach space  $(\mathcal X,\|\cdot\|)$ with $\X\subset\Lomu$. We consider the following condition on  $\X$: 
 
 \smallskip
\begin{itemize}
\item[(C3)] $L_\infty(X,\mu)\cdot \X\subset \X$ and the map $L_\infty\times \X\to \X,\ (f,g)\mapsto fg$ is continuous.
\end{itemize}

\smallskip
Again, the spaces $L_p(X,\mu),$ $1\leq p<\infty,$ satisfy (C3); more generally, if $\X$ is a lattice then it satisfies (C3). In fact, (C3) means that there is an equivalent norm $\vertiii{\,\cdot\,}$ on $\X$  such that $(\X,  \vertiii{\,\cdot\,})$ is a lattice. Note also that if every convergent sequence in $\X$ has a subsequence converging almost everywhere, for example if $\X$ satisfies (C1), then the continuity of the map $(f,g)\mapsto fg$ follows from the inclusion $L_\infty(X,\mu)\cdot \X\subset \X$ by the closed graph theorem.

\begin{theorem}\label{thm:fhcco}
Let $(\Omega,\mathcal B,\mu)$ be a separable measure space with $\mu(\Omega)<\infty$, let $\X\subset \Lomu$ be a Banach space satisfying {\rm (H1)}, {\rm (C2)} and {\rm (C3)}, and let $\varphi:(\Omega,\mathcal B,\mu)\to(\Omega,\mathcal B,\mu)$ be
a nonsingular transformation such that $C_\varphi$ is bounded on $\X$.
Assume that there exist $\mathcal C\subset \mathcal B$ and $\kappa\in(0,1)$ such  
that ${\rm span}(\mathbf 1_B:\ B\in\mathcal C)$ is dense in $\X$ and condition $({\rm H}_\kappa)$ below is satisfied.

\begin{enumerate} 
\item[$({\rm H}_\kappa)$] For every $\veps>0$, every $m\in\NN$ and any $B_1,\dots ,B_r\in\mathcal C$, the following holds true: there exist $d\geq m$, $n\in\NN$ and $B\in\mathcal B$ such that $\varphi^{-d}(B_j)=B_j$ for $j=1\dots ,r$, $\varphi^{-d}(B)=B$ and 
\[ \forall 0\leq k\leq \kappa d\;:\; \mu(\varphi^{-k}(B))\leq\veps\quad \textrm{ and }\quad \mu\bigl(\Omega\setminus \varphi^{-(n+k)}(B)\bigr)\leq\varepsilon.\]
\end{enumerate}
Then $C_\varphi$ is frequently hypercyclic on $\X$.
\end{theorem}
\begin{proof} Note first that $\X$ is separable. We shall apply Lemma \ref{lem:fhc} with $X_0:=\textrm{span}(\mathbf 1_B:\ B\in\mathcal C)$, which is contained in ${\rm Per}(T)$ by $({\rm H}_\kappa)$.   Without loss of generality, we may assume that $\mathcal C$ is stable under finite intersections: indeed, if we denote by $\widetilde{\mathcal C}$ the family of all finite intersections of sets from $\mathcal C$, 
then $({\rm H}_\kappa)$ is satisfied with $\widetilde{\mathcal C}$ in place of $\mathcal C$ since $\varphi^{-d}(B_1\cap B_2)=B_1\cap B_2$ if  $\varphi^{-d}(B_1)=B_1$ and $\varphi^{-d}(B_2)=B_2$. So (we may assume that) $X_0$ is stable under products. Note also that $X_0\subset L_\infty(\Omega,\mu)$.

\smallskip
Let $0\neq f\in X_0,$ let $\veps>0$ and let $d_0$ be the period of some vector in $X_0$. By (C2), we may choose $\delta>0$  such that $\|1_E\|_\X\leq \veps/(M\|f\|_\infty)$ for all $E\in\mathcal B$ with $\mu(E)\leq\delta,$ where $M$ is a constant, given by (C3), such that $\|hu\|_\X\leq M\|h\|_\infty\|u\|_\X$ for all $(h,u)\in L_\infty\times \X.$ 

Let $m\in\NN$ be large. By  $({\rm H}_\kappa)$, one can find $d\geq m$ which is a multiple of $d_0$ and of ${\rm per}(f)$, $n\in\NN$ and $B\in\mathcal B$ with $\varphi^{-d}(B)=B$ such that
$$\forall 0\leq k\leq \kappa d\, :\, \mu(\varphi^{-k}(B))\leq\delta\quad \textrm{ and }\quad 
{\mu\bigl(\Omega\setminus \varphi^{-(n+k)}(B)\bigr)\leq\delta.}$$
Note that $d$ is a multiple of $d_0$ and of ${\rm per}(\mathbf 1_Bf)$. 
Moreover, replacing $\kappa$ by $\kappa/2$, taking $m$ large enough 
and replacing $n$ by the smallest multiple of the period of $f$ greater than $n$, we may assume that $n$ is a multiple of the period of $f$.

Let $g=\mathbf 1_{B}f$, so that $g\in X_0$ and $d$ is 
a multiple of the period of $g$. By definition, we have \[ C_\varphi ^kg =C_\varphi ^kf\times \mathbf 1_{\varphi^{-k}(B)}\quad\hbox{for all $k\geq 0$.}\] 
In particular, for $k\in[0,\kappa d],$ 
\[\|C_\varphi^kg\|_\X\leq M \|f\|_\infty \, \|\mathbf 1_{\varphi^{-k}(B)}\|_\X\leq \veps.\]
On the other hand, since $n$ is a multiple of the period of $f$, we have $C_\varphi^{n+k}g =C_\varphi^kf \times \mathbf 1_{\varphi^{-(n+k)}(B)}$; and  since for $k\in[0,\kappa d],$ 
\[\| \mathbf 1-\mathbf{1}_{\varphi^{-(n+k)}(B)}\|_\X=\|\mathbf 1_{\Omega\backslash \varphi^{-(n+k)}(B)}\|_\X\leq \veps/(M\|f\|_\infty),\]
it follows that 
\[\forall k\in[0,\kappa d]\,:\, \|C_\varphi^{n+k}g-C_\varphi^kf\|\leq \veps.\]

By Lemma \ref{lem:fhc}, we conclude that $C_\varphi$ is frequently hypercyclic.
\end{proof}


\begin{remark}\label{rem:fhcco}
We may exchange the roles played by $\varphi^{-k}(B)$ and $\Omega\setminus \varphi^{-k}(B)$ in $({\rm H}_\kappa)$,  by considering $g=(\mathbf 1-\mathbf 1_{B})f$ instead of $g=\mathbf 1_{B}f$ in the above proof. Hence we may assume in 
$({\rm H}_\kappa)$ that
\[ \forall 0\leq k\leq \kappa d\;:\; \mu(\Omega\setminus \varphi^{-k}(B))\leq\veps\quad \textrm{ and }\quad \mu\bigl(\varphi^{-(n+k)}(B)\bigr)\leq\varepsilon.\]
\end{remark}

\smallskip Condition $({\rm H}_\kappa)$ above is perhaps a bit cumbersome.
However, here is a consequence of Theorem \ref{thm:fhcco} whose statement may be slightly easier to grasp.
\begin{corollary}\label{bis}
Let $(\Omega,\mathcal B,\mu)$ be a separable measure space with $\mu(\Omega)<\infty$, let $\X\subset \Lomu$ be a Banach space satisfying {\rm (C2)} and {\rm (C3)}, and let $\varphi:(\Omega,\mathcal B,\mu)\to(\Omega,\mathcal B,\mu)$ be
a nonsingular transformation such that $C_\varphi$ is bounded on $\X$.
 Assume that ${\rm span}(\mathbf 1_B:\ B\in\mathcal C)$ is dense in $\X$, where $\mathcal C=\{ B\in\mathcal B\,:\, \exists d\geq 1\,,\, \varphi^{-d}(B)=B\}$. Moreover, assume that there exists $\kappa\in (0,1)$ such that, for every $\veps>0$ and  every $a\in\NN$, there exist a multiple $d$ of $a$, $n\in\NN$ and a set $B\in\mathcal B$ with $\varphi^{-d}(B)=B$ such that
\[ \forall 0\leq k\leq \kappa d\;:\; \mu(\varphi^{-k}(B))\leq\veps\quad \textrm{ and }\quad \mu\bigl(\Omega\setminus \varphi^{-(n+k)}(B)\bigr)\leq\varepsilon.\]
Then $C_\varphi$ is frequently hypercyclic on $\X$.
\end{corollary}
\begin{proof} We have to check Condition $({\rm H}_\kappa)$, which is easy: given $\veps>0$, $m\in\NN$ and $B_1,\dots ,B_r\in \mathcal C$, choose $d_1,\dots ,d_r\geq 1$ such that $\varphi^{-d_j}(B_j)=B_j$, and use the assumption of Corollary \ref{bis} with $a:=md_1\cdots d_r$.
\end{proof}

\smallskip One can check that Corollary \ref{bis} applies to $\Omega=\NN$ or $\ZZ$ and the shift map $\varphi(i)=i+1$; and this yields a new proof of a very well-known fact, namely that if $(\nu_i)_{i\in\Omega}$ is a sequence of positive numbers such that 
$\sup_{i\in\Omega}\nu_{i}/\nu_{i+1}<\infty$ and $\sum_{i\in\Omega} \nu_i<\infty$, then the backward shift $\pmb{\rm B}$ acting on  $\ell_p(\Omega,\nu)$ is frequently hypercyclic. More generally, we have the following.
\begin{corollary}\label{fhcshift} Let $\Omega$ be a countable {\rm (}infinite{\rm )} set, $\mathcal B=\mathcal P(\Omega)$, and let $\mu$ be a  positive measure on $(\Omega,\mathcal B)$ such that $\mu(\Omega)<\infty$ and $\mu(i)>0$ for all $i\in\Omega$. Let also $\varphi :\Omega\to \Omega$ be such that $\sup_{i\in\Omega} \mu\bigl(\varphi^{-1}(\{ i\})\bigr)/\mu(i)<\infty$. Moreover, assume $\varphi$ has no periodic point. Then, $C_\varphi$ is frequently hypercyclic on $\Lpmu$, $1\leq p<\infty.$
\end{corollary}
\begin{proof}  We will need the following fact.
\begin{fact}\label{rnw} For any finite sets $E,F\subset\Omega$, we have $\varphi^n(E)\cap F=\varnothing$ for  all $n\in\ZZ$ such that $\vert n\vert$ is large enough. 
\end{fact}
\begin{proof} Otherwise one could find $x\in E$ and $y\in F$ such that $\varphi^n(x)=y$ for infinitely many $n\geq 1$, or $\varphi^n(y)=x$ for infinitely many $n\geq 1$. If, for example, $\varphi^n(x)=y=\varphi^{n'}(x)$ for some $1\leq n<n'$, then $\varphi^{n'-n}(x)$ is a periodic point of $\varphi$, a contradiction. 
\end{proof}

Let us first check that $\textrm{span}(\mathbf 1_B:\ B\in\mathcal C)$ is dense in $\Lpmu$, where 
\[ \mathcal C=\bigl\{ B\subset\Omega\,:\, \exists d\geq 1\,,\, \varphi^{-d}(B)=B\bigr\}.\] It is enough to show that for any finite set $E\subset\Omega$, one can approximate $\mathbf 1_E$ as close as we wish by $\mathbf 1_B$, for some $B\in\mathcal C$. Let $\varepsilon >0$, and choose a finite set $F\subset\Omega$ such that $\mu(\Omega\setminus F)<\varepsilon$. By Fact \ref{rnw}, there exists $d\geq 1$ such that $\varphi^n(E)\cap F=\varnothing$ for all $n\in\ZZ$ such that $\vert n\vert \geq d$. Then $B=\bigcup_{l\in\ZZ} \varphi^{-ld}(E)$ belongs to $\mathcal C$ and $E\subset B\subset E\cup (\Omega\setminus F)$, so $\Vert \mathbf 1_B-\mathbf 1_E\Vert <\veps^{1/p}$. 

\smallskip
Now, we fix $\kappa$ such that $0<\kappa <1/6$, and we show that the assumption of Corollary \ref{bis} is satisfied.

Let $\varepsilon >0$, and choose a finite set $F\subset\Omega$ such that $\mu\bigl(\Omega\setminus F\bigr) <\varepsilon$. It is enough to show that if $d\in\NN$ is large enough, then one can find $B\subset\Omega$ with $\varphi^{-d}(B)=B$ and $n\in\NN$ such that 
\[ \forall 0\leq k\leq \kappa d\,:\, \varphi^{-k}(B)\cap F=\varnothing\quad{\rm and}\quad \varphi^{-n-k}(B)\supset  F.\]  

Let $d\in\NN$ be large, and let $n\in\NN$ be such that  $3\kappa d\leq n\leq 4\kappa d$. If $d$ is large enough then, by Fact \ref{rnw} and since $0<\kappa <1/6$, we have $\varphi^{i+n}(F)\cap F=\varnothing$ for all $i$ such that $\vert i\vert \leq 2\kappa d$ or $\vert i\vert \geq (1-2\kappa) d$. We fix $d$ large enough and $n$ such that $3\kappa d\leq n\leq 4\kappa d$, and we define 
\[ E=\bigcup_{0\leq k\leq \kappa d} \varphi^{k+n}(F)\]
and
\[ B=\bigcup_{l\in\ZZ} \varphi^{-ld}(E)\in\mathcal C.\]

Let us fix $0\leq k\leq \kappa d$. We  show that $\varphi^{-k}(B)\cap F=\varnothing$ and $\varphi^{-n-k}(B)\supset  F$.

First, by definition of $B$ and $E$,
\[ \varphi^{-n-k}(B)\supset \varphi^{-n-k}(E)\supset F.\]

Next,
\begin{align*} \varphi^{-k}(B)&=\varphi^{-k}(E)\cup\bigcup_{l\neq 0} \varphi^{-ld-k}(E)\\
&= \bigcup_{0\leq k'\leq \kappa d} \varphi^{n+k'-k}(F)\cup \bigcup_{0\leq k'\leq \kappa d\atop l\neq 0}\varphi^{n+k'-k-ld}(F)\\
&\subset \bigcup_{\vert i\vert \leq 2\kappa d} \varphi^{n+i}(F)\cup\bigcup_{\vert j\vert \geq (1-2\kappa)d} \varphi^{n+j}(F)\\
&\subset \Omega\setminus F.
\end{align*}

This concludes the proof.
\end{proof}

\begin{remark}{Under the assumptions of Corollary \ref{fhcshift}, it is not hard to check that in fact, $C_\varphi$ satisfies the Frequent Hypercyclicity Criterion, so that it is  also chaotic and topologically mixing.} On the other hand, if $\varphi$ has a periodic point, then $C_\varphi$ cannot be hypercyclic. Indeed, assume that $\varphi^d(i_0)=i_0$ for some $i_0\in \Omega$ and some $d\geq 1$. Then, for any $f\in\Lpmu$, we have $C_\varphi^{dn} f(i_0)=f(i_0)$ for all $n\in\NN$. It follows that $C_\varphi^d$ is not hypercyclic, and hence $C_\varphi$ is not hypercyclic either by Ansari's Theorem \cite{Ans}.
\end{remark}

\begin{corollary} Let $\Omega$ be a countable {\rm (}infinite{\rm )} set, $\mathcal B=\mathcal P(\Omega)$, and let $\mu$ be a  positive measure on $(\Omega,\mathcal B)$ such that $\mu(\Omega)<\infty$ and $\mu(i)>0$ for all $i\in\Omega$. Let also $\varphi :\Omega\to \Omega$ be such that $\sup_{i\in\Omega} \mu\bigl(\varphi^{-1}(\{ i\})\bigr)/\mu(i)<\infty$. If $\varphi$ has no fixed point and $C_\varphi$ is supercyclic on $\Lpmu$, then $C_\varphi$ is frequently hypercyclic.
\end{corollary} 
\begin{proof} By Corollary \ref{fhcshift}, it is enough to show that $\varphi$ has no periodic point. Towards a contradiction, assume that $i_0\in\Omega$ is a periodic point for $\varphi$, with period $d$.  By the supercyclic version of Ansari's theorem \cite{Ans}, the operator $C_\varphi^d$ is supercyclic. If $x$ is any supercyclic vector for $C_\varphi^d$, one can find $\lambda\in\CC$ and $n\in\NN$ such that $\lambda C_\varphi^{dn}x$ is close to $\mathbf 1_{\{ i_0\}}$.  Then $\lambda x(i_0)=\lambda C_\varphi^{dn}x(i_0)$ is close to $1$, and $\lambda x(\varphi(i_0))=\lambda C_\varphi^{dn}x(\varphi(i_0))$ is close to $0$ since $\varphi(i_0)\neq i_0$. This is possible only if $x(\varphi(i_0))=0$. However, the supercyclic vectors for $C_\varphi^d$ are dense in $L_p(\Omega,\mu)$, so one can find such a vector $x$ with $x(\varphi(i_0))\neq 0$, which is the required contradiction.
\end{proof}


\smallskip
When $\Omega$ is a zero-dimensional compact group and $\varphi$ is a  translation, Theorem \ref{thm:fhcco} implies the following statement. 

\begin{corollary}\label{cor:fhcgroup}
Let $\Omega$ be a metrizable, compact and zero-dimensional topological group, and let $\mu$ be a Borel probability measure on $\Omega$. Let also  $a\in G$ be such that the composition operator $C_{\tau_a}$ associated to the left translation $\tau_a(x)= ax$ is bounded on $\Lpmu.$ Finally, let $(d_i)_{i\in\NN}$ be an increasing sequence of integers such that $a^{d_i}\to e$, the unit element of $\Omega$. Assume that there exists $\kappa>0$ such that, 
for every $\veps>0$ the following holds true: 
{for all $i_0\in\NN,$ there exist $i\geq i_0,$ $n\in\NN$}
and a Borel set $B\subset\Omega$ 
such that $a^{-d_i}B=B$ and
$$\forall 0\leq k\leq \kappa d_i\,:\, \mu(a^{-k}B)\leq \varepsilon\quad{\rm and}\quad \mu(a^{-(n+k)}B)\geq 1-\varepsilon.$$
Then $C_{\tau_a}$ is frequently hypercyclic on $\Lpmu$, $1\leq p<\infty.$
\end{corollary}
\begin{proof} Let $\mathcal B$ be the Borel $\sigma$-algebra of $\Omega$. Since $\Omega$ is a Polish space, the  probability space $(\Omega,\mathcal B,\mu)$ is separable. 

Let us denote by $\mathcal C$ the family of all clopen subsets of $\Omega$.  Since $\Omega$ is zero-dimensional, $\textrm{span}(\mathbf 1_{C}:\, C\in\mathcal C)$ is dense in  $\Lpmu$. We want to apply Theorem \ref{thm:fhcco} with the family $\mathcal C$, so we need to check that Condition $({\rm H}_\kappa)$ is satisfied. For that, we only need to show that if $C\subset \Omega$ is a clopen set, then $a^{-d_i}C=C$ provided $i\in\NN$ is large enough; which follows from the proof of Lemma \ref{periodicgeneral}. 
\end{proof}


\section{$\mathcal U$-frequently hypercyclic operators} \label{sec:ufhc}

\subsection{A criterion for $\mathcal U$-frequent hypercyclicity} In this section, we state a general result allowing to prove that an operator with many periodic vectors is $\mathcal U$-frequently hypercyclic. We even extend it to a more restrictive property, namely \emph{hereditary} $\mathcal U$-frequent hypercyclicity. An operator $T\in \mathfrak L(X)$ is said to be hereditarily $\mathcal U$-frequently hypercyclic if for any countable family $(V_q)_{q\in\NN}$ of nonempty open subsets of $X,$ for any family $(B_q)_{q\in\NN}$ of subsets of $\NN$ with positive upper density, there exists $x\in X$ such that $\mathcal N_T(x,V_q)\cap B_q$ has positive upper density for all $q\in\NN.$ This notion was introduced and studied in \cite{hfhc}.

\begin{proposition}\label{prop:ufhcbis}
Let $X$ be a separable Banach space and let $T\in\mathfrak L(X)$ be such that ${\rm Per}(T)$ is dense in $X$.
\begin{enumerate}
\item[\rm (a)] $T$ is $\mathcal U$-frequently hypercyclic if and only if, for any neighbourhood $W$ of $0$ in $X$, there exists $\alpha>0$ such that the following holds true: for every open set $\mathcal O\neq\varnothing$ in $X$ and any $N\in\NN$, one can find $x\in\mathcal O$ and an integer $m\geq N$ such that $\# \mathcal N_T(x,W)\cap\llbracket 1,m\rrbracket\geq \alpha m$. 
   
   \smallskip
    \item[\rm (b)] The following assertions are equivalent:
    
    \smallskip
    \begin{enumerate}
    \item[\rm (i)] For any nonempty open set $V\subset X$ and any $B\subset\NN$ with positive upper density, there exists
    $\alpha_{V,B}>0$ such that 
    \[ \left\{u\in X:\ \overline{\rm dens}\bigl(\mathcal N_T(u,V)\cap B\bigr)\geq \alpha_{V,B}\right\}\]
    is dense in $X.$
    \item[\rm (ii)] For any set $A\subset \NN$ with positive upper density and any neighbourhood $W$ of $0$ in $X$,  there exists $\alpha>0$ such that: for every open set $\mathcal O\neq\varnothing$ in $X$ and any $N\in\NN$, one can find $x\in\mathcal O$ and $m\geq N$ such that $\# \mathcal N_T(x,W)\cap A\cap \llbracket 1,m\rrbracket\geq \alpha m$. 
    \end{enumerate}

\smallskip Moreover, these assertions imply that $T$ is hereditarily $\mathcal U$-frequently hypercyclic.
\end{enumerate}
\end{proposition}
\begin{proof}
(a) Assume that $T$ is $\mathcal U$-frequently hypercyclic and let $y$ be a $\mathcal U$-frequently hypercyclic vector for $T.$ Let $W$ be a neighbourhood of $0$.  Let $\beta>0$ be such that $ \overline{\rm dens}\bigl(\mathcal N_T(y,W)\bigr)>\beta$ and set $\alpha=\beta/2.$ Let $\mathcal O\neq\varnothing$ be an open subset of $X$ and let $N\in\NN.$ There exists $n\in\NN$ such that $T^n y\in\mathcal O.$ Set $x=T^n y.$ There exists $m_1>\max(N+n,2n/\beta)$ such that 
\[\# \left\{k\in \llbracket 1,m_1\rrbracket:\ T^k y\in W\right\}\geq \beta m_1.\]
Then, if we set $m=m_1-n,$ 
\begin{align*}
\# \left\{k\in \llbracket 1,m\rrbracket:\ T^k x\in W\right\}&\geq \beta m_1-n\\
&\geq \alpha m.
\end{align*}

To prove the converse, note first that, by the Baire category theorem and since any set of the form $\{ x\in X:\, \# \mathcal N_T(x,W)\cap\llbracket 1,m\rrbracket\geq \alpha m\}$ is open in $X$, the assumption concerning $W$ and $\alpha$ implies that  the set $\{ x\in X:\, \overline{\rm dens}\bigl(\mathcal N_T(x,W)\bigr)\geq \alpha\}$ is dense in $X$. 

It suffices to prove that for any nonempty open set $V$, there exists $\alpha_V>0$ such that the $G_\delta$ set
\[ G_{V}:=\left\{u\in X:\ \overline{\rm dens}\bigl(\mathcal N_T(u,V)\bigr)\geq \alpha_{V}\right\}\]
is dense in $X$. Indeed, picking $(V_p)$ a countable basis of nonempty open subsets of $X$, any vector in the dense $G_\delta$ set $\bigcap_{p} G_{V_p}$ will be a $\mathcal U$-frequently hypercyclic vector for $T.$

So let $V$ be a nonempty open subset of $X$. Let $v\in V$ be a periodic point of $T$ with period $d$, and let $W_0$ be a neighbourhood of $0$ such that $v+W_0\subset V$. Choose a neighbourhood $W$ of $0$ such that $T^i(W)\subset W_0$ for $i=0,\dots ,d$, and let $\alpha>0$  be associated with this choice of $W$. Finally, let $\alpha_V:= \alpha/d$. 

Let $U$ be a nonempty open subset of $X$: we need to find $u\in U$ such that   $\overline{\rm dens}\bigl(\mathcal N_T(u,V)\bigr)\geq \alpha_{V}$. By assumption,
one can find $x$ in the nonempty open set $U-v$ such that $\overline{\rm dens}\bigl(\mathcal N_T(x,W)\bigr)\geq \alpha$. Then, one can find $i\in\llbracket 1,d\rrbracket$ such that $\overline{\rm dens}\bigl(\mathcal N_T(x,W)\cap (d\NN-i)\bigr)\geq \alpha_V$, so that $\overline{\rm dens}\bigl((\mathcal N_T(x,W)+i)\cap d\NN\bigr)\geq \alpha_V$. Now, if $k\in (\mathcal N_T(x,W)+i)\cap d\NN$ then, 
writing $x=u-v$ with $u\in U$ and $k=l+i$ with $l\in \mathcal N_T(x,W)$, we get
\[ T^k u=T^i(T^l x)+v\in W_0+v\subset V.\]
So  $\overline{\rm dens}\bigl(\mathcal N_T(u,V)\bigr)\geq \alpha_{V}$, as required.

\smallskip
(b) That (i) implies hereditary $\mathcal U$-frequent hypercyclicity  follows from the Baire category theorem. Moreover, it is clear that (i) implies (ii). 
Conversely, assume that (ii) holds true. As above,  note that the assumption concerning $A$, $W$ and $\alpha$ in (ii) implies that the set $\{ x\in X:\, \overline{\rm dens}\bigl(\mathcal N_T(x,W)\cap A\bigr)\geq \alpha\}$ is dense in $X$.

Let $V\subset X$ be nonempty and open, let $B\subset\NN$ with positive upper density and let $v$ be a periodic vector belonging to $V$. Let $W$ be  a neighbourhood of $0$ such that $v+W\subset V$ and let $d$ be the period of $v$. For $i=0,\dots,d-1,$ let $B_i=\{n\in B:\ n\equiv i\ [d]\}.$ Then at least one $B_i$, that we will call $B_p$, has positive upper density. Let us take $A=B_p$ and apply (ii) to get some $\alpha=\alpha_{V,B}>0$. 

Let $U$ be a nonempty open subset of $X$. By assumption, one can find $x\in U-T^{d-p}v$ such that 
$\overline{\rm dens}\bigl(\mathcal N_T(x,W)\cap B_p\bigr)\geq \alpha_{V,B} $. 
Write $x=u-T^{d-p}v$ with $u\in U$. Then, for any $k\in\mathcal N_T(x,W)\cap B_p$, we get \[ T^ku=T^k x+ v\in v+W\subset V,\]
so that $\overline{\rm dens}\bigl(\mathcal N_T(u,V)\cap B\bigr)\geq \alpha_{V,B}$. This shows that (i) is satisfied.
\end{proof}

\begin{remark}
The proof of (i)$\implies$(ii) in (b) as well as the ``only if'' implication in (a) do not use the density of $\textrm{Per}(T)$.
\end{remark}

\subsection{Composition operators} From Proposition \ref{prop:ufhcbis}, we now deduce a result regarding $\mathcal U$-frequent hypercyclicity of general composition operators.
\begin{theorem}\label{thm:ufhcgeneral}
Let $(\Omega,\mathcal B,\mu)$ be a separable measure space with $\mu(\Omega)<\infty$, let $\X\subset \Lomu$ be a Banach space satisfying {\rm (H1)}, {\rm (C2)} and {\rm (C3)}, and let $\varphi:(\Omega,\mathcal B,\mu)\to(\Omega,\mathcal B,\mu)$ be
a nonsingular transformation such that $C_\varphi$ is bounded on $\X$.
Let $\mathcal C:=\{B\in\mathcal B:\ \exists d\geq 1,\ \varphi^{-d}(B)=B\}$ and let us assume
that ${\rm span}(\mathbf 1_B:\ B\in\mathcal C)$ is dense in $\X$.
\begin{enumerate}
\item[\rm (1)] Assume that there exists $\alpha>0$
such that the following holds true: for any $\veps>0$ and $N\in\NN$, there exist $B\in\mathcal B$ and  an integer $m\geq N$ such that $\mu(B)\leq \veps$ and 
$\#\bigl\{k\in\llbracket 1,m\rrbracket:\ \mu(\Omega\backslash \varphi^{-k}(B))\leq \veps\bigr\}\geq \alpha m.$ 
Then $C_\varphi$
is $\mathcal U$-frequently hypercyclic.
\item[\rm (2)] Assume that for any set $A\subset\NN$ with positive upper density,
there exists $\alpha>0$
such that: for every $\veps>0$ and $N\in\NN$, there exist $B\in\mathcal B$ and $m\geq N$ such that $\mu(B)\leq \veps$ and 
{$\#\bigl\{k\in\llbracket 1,m\rrbracket:\ k\in A\textrm{ and }\mu(\Omega\backslash \varphi^{-k}(B))\leq \veps\bigr\}\geq \alpha m.$ }
Then $C_\varphi$
is hereditarily $\mathcal U$-frequently hypercyclic.
\end{enumerate}
\end{theorem}
\begin{proof}
We only prove (1), the proof of (2) being completely similar. 

Observe first that  ${\rm Per}(C_\varphi)$ is dense in $\X$. We show that the assumption of Proposition \ref{prop:ufhcbis} (a) is satisfied with the same $\alpha >0$ working for any $W$, namely $\alpha$ given by (1) above.
Let $W$ be a neighbourhood of $0$ in $\X$ and $\delta>0$ such that  $W\supset\overline B(0,\delta)$. Let also $\mathcal O$ be a nonempty open subset of $\X$
and let $N\in\NN$. We need to find  $g\in \mathcal O$ and $m\geq N$ such that $\#\bigl\{ k\in\llbracket 1,m\rrbracket:\, \Vert C_\varphi^k g\Vert\leq\delta\bigr\}\geq \alpha m$.  

Let $0\neq f\in\mathcal O\cap L_\infty$, and choose $\eta>0$ such that $\overline B(f,\eta)\subset \mathcal O$.  Let also $M>0$ be  such that $\|hu\|_\X\leq M \|h\|_\infty\,\|u\|_\X$ for any $(h,u)\in L_\infty\times\X$.  
We consider $\veps>0$ such that $\|\mathbf 1_E\|< \min(\delta,\eta)/(M\|f\|_\infty)$ for all $E\in\mathcal B$ with $\mu(E)\leq \veps.$

By assumption one may find $B\in\mathcal B$ and $m\geq N$ such that $\mu(B)\leq \veps$
and 
$$\#\bigl\{k\in\llbracket 1,m\rrbracket:\ \mu(\Omega\backslash \varphi^{-k}(B))\leq\veps\bigr\}\geq\alpha m.$$
We set $g:=f(\mathbf 1-\mathbf 1_B)$ and first observe that 
$$\|g-f\|_\X\leq M \|f\|_\infty\cdot \|\mathbf 1_B\|_\X < M\|f\|_\infty\times \frac{\eta}{M\|f\|_{\infty}} = \eta.$$
Hence, $g\in\mathcal O.$ Moreover, for any $k\geq 0$ such that $\mu(\Omega\backslash \varphi^{-k}(B))\leq \veps$, we have 
\[ \Vert C_\varphi^k g\Vert_\X\leq M \Vert C_\varphi^k f\Vert_\infty\,\Vert \mathbf 1-\mathbf 1_{\varphi^{-k}(B)}\Vert_\X\leq M \leq \delta.\]
Hence,
$\#\left\{k\in\llbracket 1,m\rrbracket:\ \|C_\varphi^k g\|\leq \delta\right\}\geq \alpha m.$
\end{proof}
\begin{remark}\label{hum}
We could exchange the roles of  $B$ and $\varphi^{-k}(B)$ in (1) by requiring that $\mu(\Omega\backslash B)\leq \veps$ and $\#\bigl\{k\in\llbracket 1,m\rrbracket:\ \mu(\varphi^{-k}(B))\leq \veps\bigr\}\geq \alpha m.$ The proof is almost identical, taking now $g:=f\mathbf 1_B.$ 
\end{remark}

When we work with a translation in a zero-dimensional compact group, the statement admits a simplified form: it is not necessary to assume that  ${\rm span}(\mathbf 1_B:\ B\in\mathcal C)$ is dense in $\Lpmu$, since this is automatically true (see the proof of Lemma \ref{periodicgeneral}).
\begin{corollary}\label{cor:ufhcgroup}
Let $\Omega$ be a metrizable, compact zero-dimensional group, and let $\mu$ be a Borel probability measure on $\Omega$. Let also $a\in G$, and assume that the composition operator $C_{\tau_a}$ associated to the left translation $x\mapsto ax$ is bounded on $\Lpmu$, $1\leq p<\infty.$
\begin{enumerate}
    \item[\rm (1)] Suppose that there exists $\alpha>0$
such that the following holds true: for every $\veps>0$, one can find a Borel set $B\subset\Omega$ and an arbitrarily large integer $m$ such that $\mu(B)\leq \veps$ and 
$\#\left\{k\in\llbracket 1,m\rrbracket:\ \mu(a^{-k}B)\geq 1-\veps\right\}\geq \alpha m.$ 
Then $C_{\tau_a}$ is $\mathcal U$-frequently hypercyclic.
\item[\rm (2)] Suppose that for any set $A\subset \NN$ with $\overline{\rm dens}(A)>0$, 
there exists $\alpha>0$
such that:  for every $\veps>0$, one can find   $B\subset\Omega$ and an arbitrarily large $m$ such that $\mu(B)\leq\veps$ and 
$\#\left\{k\in\llbracket 1,m\rrbracket:\ k\in A\textrm{ and }\mu(a^{-k}B)\geq 1-\veps\right\}\geq \alpha m.$
Then $C_{\tau_a}$ is hereditarily $\mathcal U$-frequently hypercyclic.
\end{enumerate}
\end{corollary}

\smallskip One may ask if there is a converse to Theorem \ref{thm:ufhcgeneral} under some assumptions of $\mathcal X.$ 
We do not know if this is true, but one can observe the following.

\begin{remark} Assume that $\mu(\Omega)<\infty$ and that $\X$ satisfies {\rm (H1)} and {\rm (C1)}. 
	If $C_\varphi$ is $\mathcal U$-frequently hypercyclic then, for every $\veps>0$, there exists $\alpha>0$ such that: for any $\delta>0$, there exists $B\in\mathcal B$ such that \[ \mu(B)\leq \delta\qquad{\rm and}\qquad \overline{\rm dens} \bigl( \bigl\{k\in\NN:\ \mu(\Omega\backslash \varphi^{-k}(B))\leq\veps\bigr\}\bigr)\geq\alpha.\]
\end{remark}
\begin{proof}
Let $\veps>0$, and choose $\eta>0$ such that $\|\psi\|\leq\eta\implies \mu\bigl( |\psi|\geq 1\bigr)<\veps.$
Let $f$ be a $\mathcal U$-frequently hypercyclic vector for $C_\varphi,$ let $V:=\{g\in\X:\ \Vert g-\pmb 4\Vert<\eta\}$ and let $\alpha:=\overline{\rm dens}\bigl(\mathcal N_{C_\varphi}(f,V)\bigr).$
Let $\delta>0$ and let $\eta'>0$ be such that 
$\|\psi\|\leq\eta'\implies \mu\bigl( |\psi|\geq 1\bigr)<\delta.$
Replacing $f$ by some $f\circ\varphi^k,$ which does not change the value of $\alpha$, we may assume $\|f-\pmb 2\|<\eta'.$ We then set $B:=\{x\in\Omega:\ |f(x)-4|<1\}.$ Since $B\subset\{x:\ |f(x)-2|\geq 1\},$ we get $\mu(B)\leq \delta.$ Moreover, if $k\in\mathcal N_{C_\varphi}(f,V)$ then, since
\[\Omega\backslash \varphi^{-k}(B)\subset\bigl\{x\in\Omega:\ |f\circ\varphi^k(x)-4|\geq 1\bigr\}\qquad{\rm and}\qquad \|f\circ\varphi^k -\pmb 4\|<\eta,\]
we get $\mu(\Omega\backslash \varphi^{-k}(B))\leq\veps.$ So $\overline{\rm dens} \bigl( \bigl\{k\in\NN:\ \mu(\Omega\backslash \varphi^{-k}(B))\leq\veps\bigr\}\bigr)\geq\alpha$.
\end{proof}

\subsection{$\mathcal F$-hypercyclicity} In this section, we prove an abstract version of Proposition \ref{prop:ufhcbis}. To state it, we need to recall a few definitions. In what follows, we denote by $2^\NN$ the space of all subsets of $\NN$, identified with the Cantor space.

 A \emph{Furstenberg family} is a family $\mathcal F\subset 2^\NN$ which is hereditary upward for inclusion and such that all sets in $\mathcal F$ are nonempty. An operator $T\in\mathfrak L(X)$ is said to be \pmb{$\mathcal F$}\textbf{-hypercyclic} if there exists $x\in X$ such that $\mathcal N_T(x,V)\in \mathcal F$ for every open set $V\neq\varnothing$. \emph{Hereditary} $\mathcal F$-hypercyclicity is defined in the obvious way.
 
 A Furstenberg family $\mathcal F\subset 2^\NN$ is said to be \emph{right translation-invariant} if $A\in\mathcal F\implies A+n\in\mathcal F$ for every $n\in\NN$; and \emph{left translation-invariant} if $A\in\mathcal F\implies A-n\in\mathcal F$ for every $n\in\NN$, where $A-n=\{ j\in\NN\,:\, j+n\in A\}$.
 
 Note that any Furstenberg family $\mathcal F\subset 2^\NN$ can be written in an artificial way as a union of $G_\delta$ (in fact, closed) Furstenberg families; namely, $\mathcal F=\bigcup_{A\in \mathcal F} \mathcal F_A$ with $\mathcal F_A:=\{ B\in 2^\NN:\, B\supset A\}$. A Furstenberg family is  \textbf{upper} in the sense of \cite{KarlAntonio} precisely when it can be written as a union of {left translation-invariant} $G_\delta$ Furstenberg families. For example, the family $\mathcal F$ of all sets with positive upper density is upper, since $\mathcal F=\bigcup_{r>0} \mathcal F_r$ where $\mathcal F_r=\{ A\subset\NN\,:\, \overline{\rm dens} (A)\geq r\}$.
 
 A Furstenberg family $\mathcal F$ is said to be \emph{partition-regular} if, whenever $A_1,\dots ,A_d\subset\NN$ are such that $A_1\cup\cdots \cup A_d\in\mathcal F$, it follows that at least one $A_i\in\mathcal F$; and $\mathcal F$ is said to partition-regular \emph{with respect to arithmetic progressions} if this holds at least when the sets $A_i$ are arithmetic progressions of the same step $d$ (for any $d$). 
 
\begin{proposition}\label{prop:f-hc}
Let $T\in\mathfrak L(X)$ be such that the periodic vectors of $T$ are dense in $X$. Let also $\mathcal F\subset 2^\NN$ be a Furstenberg family written as $\mathcal F=\bigcup_{r\in R} \mathcal F_r$, where each $\mathcal F_r$ is a $G_\delta$ Furstenberg family. 
\begin{enumerate}
\item[\rm (1)] Assume that each family $\mathcal F_r$ is right translation-invariant and that for any neighbourhood $W$ of $0$ in $X$ and every $d\geq 1$, there exists $r\in R$ such that: for every nonempty open set $U\subset X$, one can find $x\in U$ and $i\in\llbracket 0, d\llbracket$ such that $\mathcal N_T(x,W)\cap (i+d\ZZ_+)\in\mathcal F_r$. Then $T$ is $\mathcal F$-hypercyclic.
 \item[\rm (2)] Assume that $\mathcal F$ is partition regular with respect to  arithmetic progressions and that, for any neighbourhood $W$ of $0$ in $X$ and every   $A\in\mathcal F$, there exists $r$ such that  $\left\{ x\in X:\, \mathcal N_T(x,W)\cap A\in\mathcal F_r\right\}$ 
    is dense in $X$. 
Then $T$ is hereditarily $\mathcal F$-hypercyclic. 
\end{enumerate}
\end{proposition}
\begin{proof}



$(1)$  It suffices to prove that for any nonempty open set $V$, there exists $r=r_V$ such that 
 the $G_\delta$ set 
\[ G_{V}:=\left\{u\in X:\ \mathcal N_T(u,V)\in\mathcal F_{r}\right\}\]
is dense in $X$. Indeed, picking $(V_p)$ a basis of nonempty open subsets of $X,$ any vector in the dense $G_\delta$ set $\bigcap_{p} G_{V_p}$  
will be a $\mathcal F$-hypercyclic 
vector for $T.$

So let $V$ be a nonempty open subset of $X$. Let $v\in V$ be a periodic point of $T$ with period $d$, and let $W_0$ be a neighbourhood of $0$ such that $v+W_0\subset V$. Choose a neighbourhood $W$ of $0$ such that $T^i(W)\subset W_0$ for $i=1,\dots ,d$, and let $r=r_V$  be associated with this choice of $W$. 

Let $U$ be a nonempty open subset of $X$: we need to find $u\in U$ such that   $N_T(u,V)\in \mathcal F_r$. By assumption,
one can find $i\in\llbracket 1,d\rrbracket$ and $x$ in the nonempty open set $U-v$ such that $\mathcal N_T(x,W)\cap (d\NN-i)\in\mathcal F_r$. Then $(\mathcal N_T(x,W)+i)\cap d\NN\in\mathcal F_r$ by right-translation-invariance of $\mathcal F_r$. Now, if $k\in (\mathcal N_T(x,W)+i)\cap d\NN$ then, 
writing $x=u-v$ with $u\in U$ and $k=l+i$ with $l\in \mathcal N_T(x,W)$, we get
\[ T^k u=T^i(T^l x)+v\in W_0+v\subset V.\]
So  $\mathcal N_T(u,V)\in\mathcal F_r$, as required.

\smallskip
$(2)$ It is enough to show that for any nonempty open set $V$ and for any $B\in\mathcal F$, 
there exists $r=r_{V,B}>0$ such that the $G_\delta$ set 
\[ G_{V,B}:=\left\{u\in X:\ \mathcal N_T(u,V)\cap B\in\mathcal F_{r}\right\}\]
is dense in $X$. Picking  any sequence $(V_q)$ of nonempty open subsets of $X$ and any seuence $(B_q)\subset\mathcal F$, every vector $x$ in the dense $G_\delta$-set $\bigcap_{q}G_{V_q,B_q}$ will then be such that $\mathcal N(x,V_q)\cap B_q\in\mathcal F$ for all $q\in\NN$.

So let $V$ and $B$ as above and let $v$ be a periodic vector belonging to $V$. Let $W$ be a neighbourhood of $0$ such that $v+W\subset V$, and let $d$ be the period of $v$. For $i=0,\dots,d-1,$ let $B_i:=B\cap (i+d\ZZ_+)$. Since $\mathcal F$ is partition regular with respect to arithmetic progressions, there exists some $i_0$ such that $B_{i_0}\in \mathcal F$. We apply the assumption with $A:= B_{i_0}$ to get some $r=r_{V,B}$. 

Let $U$ be any nonempty open subset of $X$; we need to find $u\in U$ such that   $N_T(u,V)\cap B\in \mathcal F_r$. By assumption, one can find $x\in U-T^{d-i_0}v$ such that $\mathcal N_T(x,W)\cap B_{i_0}\in\mathcal F_r$. Now, writing $x=u-T^{d-i_0}v$ with $u\in U$, we have $\mathcal N_T(x,W)\cap B_{i_0} \subset \mathcal N_T(u,V)$: indeed, if $k\in \mathcal N_T(x,W)\cap B_{i_0}$, then 
\[ T^ku=T^kx+ T^{k-i_0+d}v=T^kx+v\in v+W\subset V.\]
So we get  $\mathcal N_T(u,V)\cap B_{i_0}\in\mathcal F_r$, and hence $\mathcal N_T(u,V)\cap B\in\mathcal F_r$.
\end{proof}

\smallskip  Proposition \ref{prop:ufhcbis} is a special case of Proposition \ref{prop:f-hc}. Indeed, let $\mathcal F$ be the family of all sets $A\subset\NN$ with positive upper density. Then  $\mathcal F=\bigcup_{r>0} \mathcal F_r$ where each 
$\mathcal F_r:=\{ A\subset\NN\,:\, \overline{\rm dens} (A)\geq r\}$ is $G_\delta$ and translation-invariant. Moreover, $\mathcal F$ is partition-regular by subadditivity of upper density. Let us show that (a) in Proposition \ref{prop:ufhcbis} follows from (1) above. Given $W$ and $d$,  choose $\alpha >0$ associated with $W$ according to the assumption in (a), and let $r:=\alpha/d$. By the Baire category theorem and since any set of the form $\{ x\in X:\, \# \mathcal N_T(x,W)\cap\llbracket 0,m\rrbracket\geq \alpha m\}$ is open in $X$, the  set $\{ x\in X:\, \overline{\rm dens}\bigl( \mathcal N_T(x,W)\bigr)\geq\alpha\}$ is dense in $X$. So we are done by subadditivity of upper density. Similarly, (b) in Proposition \ref{prop:ufhcbis} follows from (1) above.

\smallskip Here is another consequence of Proposition \ref{prop:f-hc}. This is not new, see \cite{GMM21}.

\begin{corollary} Let $T\in\mathfrak L(X)$ be such that ${\rm Per}(T)$ is dense in $X$.
\begin{enumerate}
\item[\rm (a)] Assume that for any neighbourhood $W$ of $0$ and any open set $U\neq\varnothing$, we have $\mathcal N_T(U,W)\neq\varnothing$. Then $T$ is hypercyclic.
    \item[\rm (b)] Assume that for any neighbourhood $W$ of $0$ and any open set $U\neq\varnothing$, the set $\mathcal N_T(U,W)$ is cofinite.  
Then $T$ is topologically mixing.
\end{enumerate}
\end{corollary} 
\begin{proof} (a) Apply (1) above with $\mathcal F:=2^\NN\setminus\{\varnothing\}$. 

\smallskip
(b) We apply (2) above with  $\mathcal F=[\NN]^\infty$, the family of all infinite subsets of $\NN$, which is $G_\delta$ and partition regular. Note that topological mixing is equivalent to hereditary $\mathcal F$-hypercyclicity. 
{
Let $W$ be a neighbourhood of $0$ and let $A\in [\NN]^\infty$. Towards a contradiction with (2), assume that there exists a nonempty open set $U_0$ such that $\mathcal N_T(x,W)\cap A$ is finite for all $x\in U_0$. Let $(F_n)_{n\in\NN}$ be an enumeration of all finite subsets of $\NN.$ Then, the sets $\mathbf F_n:=\{x\in X:\ \mathcal N_T(x,W)\cap A\subset F_n\}$ are closed in $X$, and $U_0\subset \bigcup_{n\in\NN} \mathbf F_n$. 
By the Baire category theorem, it follows that one can find a finite set $F$ and a nonempty open set $U\subset U_0$ such that $\mathcal N_T(x,W)\cap A\subset F$ for all $x\in U$. In other words, $\mathcal N_T(U,V)\cap A\subset F$, a contradiction since $A$ is infinite and $\mathcal N_T(U,V)$ is cofinite. 
}
\end{proof}

\begin{remark} Part (a) can be slightly strengthened, as follows: $T$ is hypercyclic as soon as (i) $\mathcal N_T(U,W)\neq\varnothing$ for any $U, W$ as in (a), and (ii) for any open set $V\neq \varnothing$, one can find $z\in V$ such that $\mathcal N_T(z,V)$ has bounded gaps (we do not assume that ${\rm Per}(T)$ is dense).
\end{remark}
\begin{proof} Let $U,V$ be nonempty open sets in $X$. Choose a nonempty open set $V'\subset V$ and a neighbourhood $W$ of $0$ such that $V'+W\subset V$. By (ii), one can find $z\in V'$ and $d\in\NN$ such that $\mathcal N_T(z,V')$ has gaps of length at most $d$. Let $U':=U-z$, and let $W'$ be a neighbourhood of $0$ such that $T^i(W')\subset W$ for $i=0,\dots ,d$. By (i), one can find $x\in U'$ and $n\in\NN$ such that $T^nx\in W'$. Taking $i\leq d$ such that $i+n\in \mathcal N_T(z,V')$, we see that $u:= z+x\in U$ and $T^{i+n}u\in V'+T^iW'\subset  V$. Hence $\mathcal N_T(U,V)\neq\varnothing$, and $T$ is topologically transitive.
\end{proof}


\section{Composition operators induced by odometers} \label{sec:odometers}
\subsection{Notation} In this section, we use the notation of Subsection \ref{subsec:odometers}. We fix once and for all an odometer $(\Omega,\mathcal B,\mu,\od)$, where 
\[ \Omega=\prod_{i\geq 1} \ZZ/m_i\ZZ=:\prod_{i\geq 1} \Omega_i,\] 
$\mathcal B$ is the Borel $\sigma$-algebra of $\Omega$ 
and $\mu=\otimes_{i\geq 1} \mu_i$ is a product probability measure. We consider the operator $C_\od$ acting on $\Lpmu$, $1\leq p<\infty$; so we are assuming that the boundedness condition (\ref{eq:continuityodometer}) is satisfied. We also assume that the measure $\mu$ is non-atomic, which means that 
\[ \prod_{i=1}^\infty \eta_i=0\]
where
\[ \eta_i=\max\bigl( \mu_i(j)\,:\, j\in\Omega_i\bigr).\]

Finally, we recall that the periodic points of $C_\od$ are dense in $\Lpmu$, so that $C_\od$ is chaotic as soon as it is hypercyclic: this follows from Lemma \ref{periodicgeneral}; or, in a more elementary way, one just need to observe that if $C=[C_1,\dots ,C_n]\subset \Omega$ is an arbitrary cylinder set and $d:=\prod_{i=1}^{n} m_i$, then $\od^{-d}(C)=C$ since $\od^d(x)_i=x_i$ for every $x\in\Omega$ and $i=1,\dots ,n$ by Fact \ref{addition}.





\subsection{Hypercyclicity}


The next theorem gives a simple sufficient condition for hypercyclicity of $C_\od$. The proof relies on a probabilistic argument. For $i\geq 1,$ let us set 
\[\theta_i:=\sup\bigl\{\mu_i(D)-\mu_i(D+k):\ D\subset \Omega_i, \ k\in\ZZ
\bigr\},\]
where  addition is understood in $\Omega_i=\mathbb Z/m_i\mathbb Z$.

\begin{theorem}\label{thm:hc}
Assume that 
there exists an increasing sequence of positive integers $(i_s)_{s\geq 1}$ such that
\begin{enumerate}
\item[\rm (a)] $\displaystyle \sum_{s=1}^{\infty}\prod_{i=i_s+1}^{i_{s+1}-1}\mu_i(m_i-1)<\infty$;
\item[\rm (b)] $\displaystyle \frac 1n\left(\sum_{s=1}^n\theta_{i_s}\right)^{2}\to\infty$ as $n\to\infty$.
\end{enumerate}
Then $C_\od$ is hypercyclic on $\Lpmu$, $1\leq p<\infty$.
\end{theorem}
\begin{proof}
Let $\veps>0.$ Since for all $N\geq 0,$
$$\frac 1n\left(\sum_{s=N+1}^{N+n} \theta_{i_s}\right)^{2}\xrightarrow{n\to\infty}\infty,$$
we can assume, forgetting the first $N$ terms of $(i_s)$ for some large enough $N$,  that we have selected an integer $n\geq 1$ such that 
\[ \sum_{s=1}^{n-1} \prod_{i=i_s+1}^{i_{s+1}-1}\mu_i(m_i-1)<\veps\qquad{\rm and} \qquad \exp\left[-\frac 2{9n}\left(\sum_{s=1}^n  \theta_{i_s}\right)^2\,\right]<\veps.\]
The reason for putting the strange factor $2/9$ will become clear in a few lines.

For $s=1,\dots,n,$ choose $D_{i_s}\subset \Omega_{i_s}$ and $k_{i_s}\in\llbracket 0, m_{i_s}-1\rrbracket$ such that \[ \theta_{i_s}=\mu_{i_s}(D_{i_s})-\mu_{i_s}(D_{i_s}+k_{i_s}).\]
We consider the random variables $X_s$ and $Y_s$ defined on $(\Omega,\mathcal B,\mu)$ as follows:
\[X_s(x):=\left\{
\begin{array}{ll}
1&\textrm{if }x_{i_s}\in D_{i_s},\\
0&\textrm{otherwise}
\end{array}
\right.\quad \textrm{ and }\quad
Y_s(x):=\left\{
\begin{array}{ll}
1&\textrm{ if }x_{i_s}\in D_{i_s}+k_{i_s},\\
0&\textrm{ otherwise.}
\end{array}
\right.\]
Note that $X_1,\dots ,X_n$ are independent Bernoulli variables with $\mathbb E (X_s)=\mu_{i_s}(D_{i_s})$; and $Y_1,\dots ,Y_n$ are independent Bernoulli variables with $\mathbb E (Y_s)=\mu_{i_s}(D_{i_s} +k_{i_s})$.

\smallskip Let us define 
\begin{align*}
B_X&:=\left\{x\in\Omega:\ \sum_{s=1}^n X_s(x) \geq \sum_{s=1}^n \left(\mu_{i_s}(D_{i_s})-\frac{\theta_{i_s}}3\right)\right\},\\
B_Y&:=\left\{x\in\Omega:\ \sum_{s=1}^n Y_s(x) \leq \sum_{s=1}^n \left(\mu_{i_s}(D_{i_s}+k_{i_s})+\frac{\theta_{i_s}}3\right)\right\}
\end{align*}
and, for $s=1,\dots,n-1,$ 
$$E_s:=\{x\in\Omega:\ x_i=m_i-1\textrm{ for  }i=i_s+1,\dots,i_{s+1}-1\}.$$
Let finally 
\begin{align*}
B&:=\big(B_X\cap B_Y\big)\backslash \bigcup_{s=1}^{n-1} E_s\,, \\{\rm and}\qquad 
k&:=\sum_{s=1}^n k_{i_s}\prod_{i=1}^{i_s-1}m_i\,.
\end{align*}

We claim that $\mu(B)>1-3\veps$ and that $B\cap \od^k(B)=\varnothing,$ which will show that $C_\od$ is hypercyclic by Corollary \ref{lem:topologicallytransitive}.

\smallskip
Let $x\in B$. By Fact \ref{addition}, we have $\od^k(x)=x\pmb+(0,\dots ,k_{i_1},0\dots ,k_{i_2},0\dots ,k_{i_n},0\dots)$. Hence $(\od^k(x))_{i_s}=x_{i_s}+k_{i_s}\ (\!\!\!\!\mod m_{i_s})$ for $s=1;$ and this remains true for $s=2,\dots,n$ since $x$ does not belong to $E_{s-1}$: indeed, there exists $i\in(i_{s-1},i_s)$ such that $x_i\neq m_i-1,$ which prevents
the occurence of a carry   at the $i_s$-th position. So we have $Y_s(\od^k(x))=X_s(x)$ for $s=1,\dots ,n$. Since $x\in B_X$, it follows that
\begin{align*}
\sum_{s=1}^n Y_s(\od^k(x))&=\sum_{s=1}^n X_s(x)\\
&\geq \sum_{s=1}^n \left(\mu_{i_s}(D_{i_s})-\frac{\theta_{i_s}}3\right)\\
&>\sum_{s=1}^n  \left(\mu_{i_s}(D_{i_s}+k_{i_s})+\frac{\theta_{i_s}}3\right), 
\end{align*}
so that $\od^k(x)\not\in B_Y$. This shows that $B\cap \od^k(B)=\varnothing.$ 

By definition of $B_X$, we have
\begin{align*}
\mu(B_X)&= 1-\mu\left(\sum_{s=1}^n X_s< \sum_{s=1}^n \left(\mu_{i_s}(D_{i_s})-\frac{\theta_{i_s}}3\right)\right)\\
&=1-\mu \left(\sum_{s=1}^n \bigl(X_s-\mathbb E(X_s)\bigr)<-\frac 13\sum_{s=1}^n \theta_{i_s}\right).
\end{align*}
Since $X_1,\dots ,X_n$ are independent Bernoulli variables, it follows, by Hoeffding's inequality (see e.g. \cite{BLM}), that
$$\mu(B_X)\geq 1-\exp\left(-\frac 2{9n}\left(\sum_{s=1}^n \theta_{i_s}\right)^2\right);$$
and the choice of $n$ ensures that $\mu(B_X)>1-\veps.$ Similarly, $\mu(B_Y)>1-\veps.$ Since
$$\sum_{s=1}^{n-1}\mu(E_s)=\sum_{s=1}^{n-1}\prod_{i=i_s+1}^{i_{s+1}-1}\mu_i(m_i-1)<\veps,$$
we conclude that $\mu(B)>1-3\varepsilon$.
\end{proof}

Theorem \ref{thm:hc} allows us to prove a strengthened version of \cite[Corollary 3.5]{BDDD22}.
For $i\geq 1,$ we set \[\delta_i:=\min\{\mu_i(j):\ j\in \Omega_i\},\]
and we recall that
\[\eta_i=\max\{\mu_i(j):\ j\in\Omega_i\}.\]

Observe that \[ \theta_i\geq \eta_i-\delta_i.\]
 Indeed, assuming that $\eta_i-\delta_i>0$, choose $j,k\in \Omega_i$ such that $\mu_i(j)=\eta_i$ and $\mu_i(j+k)=\delta_i$, and take $D=\{ j\}$ in the definition of $\theta_i$.

\begin{corollary}\label{cor:hc}
If $\limsup\, (\eta_i-\delta_i)>0$, then $C_\od$ is hypercyclic.
\end{corollary}

\begin{proof} By assumption, there exist $c >0$ and an increasing sequence of integers $(i_s)_{s\geq 1}$ such that $\eta_{i_s}-\delta_{i_s}\geq c$ for all $s\geq 1.$ 
Then, Condition (b) of Theorem \ref{thm:hc} is clearly satisfied
since $\theta_{i_s}\geq \eta_{i_s}-\delta_{i_s}\geq c$ for all $s\geq 1.$ Moreover, since $\prod_{i=1}^\infty \eta_i=0$, we may assume, upon extracting a subsequence of $(i_s)$, that 
$\prod_{i=i_s+1}^{i_{s+1}-1}\eta_i\leq 2^{-s}$ for all 
$s\geq 1$, so that Condition (a) is satisfied as well.
\end{proof}

In particular, we may apply Corollary \ref{cor:hc} as soon as $\limsup\,  \eta_i>1/2$ (whereas in \cite{BDDD22} hypercyclicity was only proved when $\limsup \,\eta_i=1$). 

\smallskip 
We may also conclude that $C_\od$ is hypercyclic  as soon as $\limsup \eta_i>0$ and $\delta_i\to 0$.  For example, this holds for the so-called \emph{Ornstein odometer}.

\begin{example}
Let $(\Omega,\mathcal B,\mu,\od)$ be the Ornstein odometer,  defined by 
$\Omega_i=\llbracket 0,i\rrbracket$, $\mu_i(0)=1/2$ and $\mu_i(j)=1/2i$ for  $1\leq j\leq i$. Then 
 $C_\od$ is (bounded  and) hypercyclic on $\Lpmu$.
\end{example}

\smallskip
Finally, Corollary \ref{cor:hc} allows us to completely characterize the hypercyclic composition operators induced by odometers with the same measure on each component.

\begin{corollary}
Let $N\geq 2$, and le $\nu$ be a probability measure on $\llbracket 0,N-1\rrbracket$. Assume that  $\Omega_i=\llbracket 0,N-1\rrbracket$ for all $i\geq 1$, and that all measures $\mu_i$ are equal to $\nu$. 
Then $C_\od$ is bounded on $L_p$ if and only if $\nu(0)\geq \nu({N-1})$; and when this condition is satisfied, $C_\od$ is hypercyclic on $\Lpmu$ if and only if $\nu$ is not the uniform distribution on $\llbracket 0,N-1\rrbracket$, \textit{i.e.} $\mu$ is not the Haar measure of the compact group $(\Omega,\pmb +)$.
\end{corollary}
\begin{proof} Since $\mu_i=\nu$ for all $i$, the boundedness condition (\ref{eq:continuityodometer}) for $C_\od$ reduces to 
\[ \sup_{l\geq 1}\left(\frac{\nu(N-1)}{\nu(0)}\right)^{l-1}<\infty,\]
\textit{i.e.} $\nu({N-1})\leq \nu(0)$. 
If $\nu$ is the uniform distribution, then $\od$ is measure-preserving, so $C_\od$ is an isometry and hence it is not hypercyclic. Otherwise, $\eta_i-\delta_i$ is positive and does not depend on $i$, so we may apply Corollary \ref{cor:hc}.
\end{proof}

\smallskip
Let us now discuss an example of a binary odometer with $\lim \eta_i=1/2=\lim \delta_i.$

\begin{example}\label{alpha} Let $\alpha>0$. Let us assume that  $\Omega_i=\{0,1\}$ for all $i\geq 1$,  and that $\mu_i(0)=\frac 12+\frac1{i^\alpha}\cdot$ 
\begin{itemize}
\item[\sbt] if $\alpha\in(0,1/2)$, then $C_\od$ is hypercyclic;
\item[\sbt] if $\alpha\in(1,\infty)$, then $C_\od$ is not hypercyclic.
\end{itemize}
\end{example}
\begin{proof}
We first observe that $C_\od$ is easily seen to be bounded on $L_p$, for any $\alpha>0.$

Suppose that $\alpha\in(0,1/2)$. Choose $\beta>1$ such that $\alpha \beta<1/2$, and let $i_s=\lfloor s^\beta\rfloor$ for all $s\geq 1$. Then 
$i_{s+1}-i_s\sim \beta s^{\beta-1}$ as $s\to\infty$, so  Condition (a) in Theorem \ref{thm:hc} is satisfied since $\mu_i(m_i-1)=\mu_i(1)\leq 1/2$ for all $i\geq 1$ and the series 
$\sum_s 2^{-c s^{\beta-1}}$ 
is convergent for any $c>0$. 
Moreover, since $\varepsilon_i=\frac2{i^\alpha}$ for all $i\geq 1$, we see that
\[\sum_{s=1}^n \theta_{i_s} 
=\sum_{s=1}^n \frac 2 {i_s^\alpha} \sim  c_\beta n^{1-\alpha\beta}\quad\hbox{as $n\to \infty$},\]
which allows us to conclude that Condition (b) in Theorem \ref{thm:hc} is satisfied since $1-\alpha\beta>1/2$. 

Suppose that $\alpha>1$.  Let us observe that for all $i\geq 1,$
\[ \frac{\eta_i }{\delta_i}=\frac{\frac 12+\frac1{i^\alpha}}{\frac 12-\frac 1{i^\alpha}}=1+\frac 4{i^\alpha} +o\left(\frac1{i^\alpha}\right)\quad \hbox{as $i\to\infty$}.\]
In particular, the infinite product $\prod_{i\geq 1} ({\eta_i }/{\delta_i})$ is convergent, which prevents $C_\od$ from being hypercyclic by \cite[Theorem 3.8]{BDDD22}.
\end{proof}

\begin{question}
What happens in Example \ref{alpha} for $\alpha\in[1/2,1]$?
\end{question}

\begin{question}
A look at the proof of  \cite[Theorem  3.8]{BDDD22} reveals that when the infinite product $\prod_{i\geq 1} ({\eta_i }/{\delta_i})$ 
is convergent, the operator $C_\od$ is in fact power-bounded, \textit{i.e.} $\sup_{n\in\NN} \|C_\od^n\|<\infty$ is bounded. Is it true that $C_\od$ is hypercyclic as soon as it is not power-bounded?
\end{question}



\subsection{Mixing}

In \cite{BDDD22}, it is shown that $C_\od$ is topologically mixing if $\lim_{i\to\infty} \eta_i=1,$ and that the converse is true provided the sequence $(m_i)$ is bounded. In this section, our first aim is to get a characterization working even for unbounded sequences $(m_i)$. For this, we need to introduce a new sequence: for $i\geq 1,$ we set 
\[\kappa_i:=\inf_{j\in\llbracket 1,m_i-1\rrbracket}\sup\bigl\{\mu_i(D):\ D\subset \llbracket 0,m_i-1\rrbracket\textrm{ and }(D+j)\cap D=\varnothing\bigr\}.\]
Here (and in the proof of the next theorem),  addition is understood as  addition in $\ZZ_+$, 
  not in $\mathbb Z/m_i\mathbb Z.$

\begin{theorem}\label{thm:mixing}
 $C_\od$ is topologically mixing if and only if $\lim\limits_{i\to\infty}\kappa_i=1.$
\end{theorem}
\begin{proof} The proof relies on Corollary \ref{lem:topologicallytransitive}.

\smallskip
Assume first that $C_\od$ is mixing. Given $\veps>0$, we need to find $i_0$ such that $\kappa_i\geq 1-\varepsilon$ for all $i\geq i_0$. By Corollary \ref{lem:topologicallytransitive}, there exists $k_0\in\mathbb N$ such that, for all $k\geq k_0,$ we may find $B\in\mathcal B$ with $\mu(B)>1-\veps/2$ and 
$\od^k(B)\cap B=\varnothing.$ Let $i_0\geq 1$ be such that $m_1\cdots m_{i_0-1}>k_0$; we check that $i_0$ works. 

Let us fix $i\geq i_0,$ and let $j\in\llbracket 1 ,m_i-1\rrbracket$. We set $k:=jm_1\cdots m_{i-1}$ and we choose $B\in\mathcal B$ such that $\mu(B)>1-\veps/2$ and 
$\od^k(B)\cap B=\varnothing.$

For $a\in \Omega_i,$ we set $B_a:=\{x\in B:\ x_i=a\}$, and we denote by $\widehat{\,B_a}$ the projection of $B_a$ onto $\prod_{r\neq i}\Omega_r.$ Let us also denote by $\widehat{\,\mu_i}$ the measure $\otimes _{r\neq i}\mu_r.$  We observe that if $0\leq a\leq m_i-1-j$, then $\widehat{\,B_a}\cap \widehat{\,B_{a+j}}=\varnothing.$ Indeed, otherwise there exists $(x_r)_{r\neq i}\in \prod_{r\neq i}\Omega_r$ such that 
\[ x=(x_1,\dots,x_{i-1},a,x_{i+1},\dots)\in B\quad\textrm{ and }\quad x'=(x_1,\dots,x_{i-1},a+j,x_{i+1},\dots)\in B;\]
but $\od^k(x)=x'$ and this contradicts  $\od^k(B)\cap B=\varnothing.$ 

We now define 
$$D:=\big\{a\in\llbracket 0,m_i-1\rrbracket :\ \widehat{\,\mu_i}\bigl(\widehat{\,B_a}\bigr)> 1/2\bigr\}.$$
Note that if $a,a'\in D$, then $\widehat{\,B_a}\cap \widehat{\,B_{a'}}\neq\varnothing$. Hence $D\cap (D+j)=\varnothing$, since otherwise there exists $a\in D\cap\llbracket 0,m_{i}-1-j\rrbracket$ such that $a'=a+j\in D$ and we know that $\widehat{\,B_a}\cap \widehat{\,B_{a+j}}=\varnothing.$ 
Moreover, 
\begin{align*}
1-\frac{\veps}2\leq \mu(B)&=\sum_{a\in\Omega_i}\, \mu_i(a)\widehat{\,\mu_i}\bigl(\widehat{\, B_a}\bigr)\\
&\leq \mu_i(D)+\sum_{a\in \Omega_i\backslash D}\mu_i(a)\,\widehat{\,\mu_i}\bigl(\widehat{\,B_a}\bigr)\\
&\leq 1-\mu_i (\Omega_i\backslash D)+\frac 12\,\mu_i (\Omega_i\backslash D).
\end{align*}
So we get $\mu_i(\Omega_i\backslash D)\leq \veps$, \textit{i.e.} $\mu_i(D)\geq 1-\varepsilon$. Thus, for any $j\in\llbracket 1 ,m_i-1\rrbracket$ we can find a set $D\subset \llbracket 0,m_i-1\rrbracket$ such that $D\cap (D+j)=\varnothing$ and  $\mu_i(D)\geq1-\veps$; in other words, $\kappa_i\geq 1-\veps.$

\smallskip

Conversely, assume that $\lim_{i\to\infty}\kappa_i=1.$ 
This time, given $\veps>0$, we need to find $k_0\in\NN$ such that, for any $k\geq k_0$, there exists $B\subset\Omega$ such that $\mu(\Omega\setminus B)\leq \varepsilon$ and $\od^k(B)\cap B=\varnothing$. 

By assumption, there exists $i_0\geq 1$ such that the following holds true: for every $i\geq i_0$ and all $j\in\llbracket 1,m_{i_0}-1\rrbracket,$ there exists $D_{i,j}\subset \Omega_i$ with $\mu_i(D_{i,j})\geq 1-\veps/3$ and $(D_{i,j}+j)\cap D_{i,j}=\varnothing.$ We show that $k_0:= \sum_{i=1}^{i_0}m_1\cdots m_{i-1}$ works.

Let $k\geq k_0$. By the choice of $k_0$, one can write $k=\sum_{i=1}^l k_im_1\cdots m_{i-1}$, where $l\geq i_0$,  $k_i\in\llbracket 0,m_i-1\rrbracket$  for all $i$ and $k_l\neq 0.$ By our choice of $i_0,$ one can find $D'_l, D''_l\subset \Omega_l$ such that $\mu_l(D'_l),\mu_l(D''_l)\geq 1-\veps/3$ and $(D'_l+k_l)\cap D'_l=\varnothing=(D''_l+(k_l+1))\cap D''_l$ (if $k_l=m_l-1$, just take $D''_l=\Omega_l$). We then set $D_l:=D'_l\cap D''_l$ and observe that $\mu_l(D_l)\geq 1-2\veps/3.$ We also choose $D_{l+1}\subset \Omega_{l+1}$  such that $(D_{l+1}+1)\cap D_{l+1}=\varnothing$ and $\mu_{l+1}(D_{l+1})\geq 1-\veps/3.$ Now, we define
\[B:=[\Omega_1,\dots \Omega_{l-1}, D_l, D_{l+1}],
\]
and observe that $\mu(B)\geq (1-2\veps/3)(1-\veps/3)\geq 1-\varepsilon$. We claim that $\od^k(B)\cap B=\varnothing.$ 

Let $x\in B$: we need to show that $\od^k(x)\not\in B$. Recall that, by Fact \ref{addition}, we have
\[\od^k(x)=x\pmb +(k_1,\dots ,k_l, 0,0,\dots ).\]
We distinguish several cases.
\begin{itemize}
\item[\sbt] if $x_l\geq m_l-k_l,$ then $(\od^k(x))_{l+1}=x_{l+1}+1$ (there is a carry at the $(l+1)$-th position) and $x_{l+1}+1\notin D_{l+1}$, so $\od^k(x)\not\in B.$
\item[\sbt] if $x_l=m_l-k_l-1$ and if there is a carry at the $l$-th position when we compute $\od^k(x),$ then we are in the same situation.
\item[\sbt] if $x_l=m_l-k_l-1$ and if there is no carry  at the $l$-th position when we compute $\od^k(x),$ then $(\od^k(x))_l=x_l+k_l\notin D_l$ because $x_l\in D_l$ and $D_l\cap(D_l+k_l)=\varnothing$, so  $\od^k(x)\notin B$.
\item[\sbt] if $x_l<m_l-k_l-1,$ then $(\od^k(x))_l=x_l+k_l$ or $x_l+k_l+1$ depending on whether there is a carry over at the $l$-th position of not. In both cases, $\od^k(x)\notin B.$
\end{itemize}

So we have shown that $\od^k(B)\cap B=\varnothing$, which concludes the whole proof.
\end{proof}

\begin{remark} From the definition of $\kappa_i$, it is clear that $\kappa_i\geq \eta_i$ for all $i\geq 1$. Moreover, it can also be seen from the definition of $\kappa_i$ that if the sequence $(m_i)$ is bounded and if $\kappa_i\to 1$ as $i\to\infty$, then $\eta_i\to 1$. So Theorem \ref{thm:mixing} is coherent with (and in fact, formally implies) Theorem \ref{EmmaUdayan} recalled in Section \ref{subsec:odometers}. {On the other hand, if $(m_i)$ is unbounded, it may happen that $\kappa_i\to 1$ and $\eta_i\not\to 1$. It is even possible that $\kappa_i\to 1$ and $\eta_i\to 0$, see \cite[Example 3.3]{BDDD22}.}
\end{remark}
\begin{proof} Assume that $(m_i)$ is bounded, that $\kappa_i\to 1$ and, towards a contradiction, that $\liminf \eta_i<1$. Then, one can find an increasing sequence of indices $(i_s)_{s\geq 1}$, an integer $m$, some $\varepsilon >0$ and $a,b\in \llbracket 0,m-1\rrbracket$ with $a<b$, such that $m_{i_s}=m$  and $\mu_{i_s}(a), \mu_{i_s}(b)>\varepsilon$ for all $s\geq 1$. Let $j:=b-a$. If $s\geq 1$, then any set $D\subset \llbracket 0,m-1\rrbracket$ such that $\mu_{i_s}(D)>1-\varepsilon$ must contain $a$ and $b $, so it cannot satisfy $D\cap (D+j)=\varnothing$. So we have $\kappa_{i_s}\leq 1-\varepsilon$ for all $s\geq 1$, a contradiction since $\kappa_i\to 1$ as $i\to\infty$.
\end{proof}

\smallskip
We can now improve Theorems 3.6 and 3.7 of \cite{BDDD22} by exhibiting, for any sequence of integers $(m_i)$, a measure $\mu$ on $\prod_{i\geq 1} \ZZ/m_i\ZZ$ for which $C_\od$ is hypercyclic but not topologically mixing. In \cite{BDDD22}, this was done when $(m_i)$ does not grow too quickly.

\begin{example}
Let $(m_i)$ be any sequence of integers with $m_i\geq 2$ for all $i\geq 1.$ There exists a sequence of probability measures $(\mu_i)$ such that $C_\od$ is hypercyclic but not topologically mixing on $\Lpmu$.
\end{example}
\begin{proof}
Let us define the measures $\mu_i$. If $m_i=2,$ we set $\mu_i(0):=2/3$ and $\mu_i(1):=1/3.$
Otherwise, if  $m_i$ is even, $m_i=2\beta_i$, we set 
$$\mu_i(j):=\frac{c_i}{2^{\beta_i-1-j}}\qquad\hbox{for  $j=0,\dots,\beta_i-1$}$$
 and 
$$\mu_i(j):=\frac{c_i}{2^{j-\beta_i}}\qquad\hbox{for $j=\beta_i,\dots,2\beta_i-1$},$$
where $c_i>0$ is chosen so that $\mu_i$ is a probability measure.
If $m_i$ is odd, $m_i=2\beta_i+1$, we set 
$$\mu_i(j):=\frac{c_i}{2^{\beta_i-j}}\qquad\hbox{for $j=0,\dots,\beta_i$}$$
and 
$$\mu_i(j):=\frac{c_i}{2^{j-\beta_i}}\qquad\hbox{for $j=\beta_i+1,\dots,2\beta_i$},$$
where $c_i>0$ is again chosen so that $\mu_i$ is a probability measure. {Since $\eta_i\leq 2/3$ for all $i\in\NN,$ the product measure $\mu$ is non-atomic.}

We first observe that $C_\od$ is bounded. Indeed, for every $i\in\NN,$ we have
$$\frac{\mu_i(m_i-1)}{\mu_i(0)}\leq 1$$
and  
$$\frac{\mu_i(j-1)}{\mu_i(j)}\in\{1/2;2\}\qquad\hbox{for all $1\leq j\leq m_i-1$.}$$
So it is clear that  the boundedness condition \eqref{eq:continuityodometer} is satisfied.

We now claim that $C_\od$ is hypercyclic. Indeed, let $i\geq 1$ and assume that $m_i$ is odd, $m_i=2\beta_i+1$. Then 
\[\eta_i-\delta_i=\mu_i(\beta_i)-\mu_i(0)=c_i\left(1-\frac 1{2^{\beta_i}}\right)\geq \frac{c_i}{2}\cdot\]

Now, by definition of $\mu_i$, we have
\[ 2c_i\sum_{j=0}^{\infty}\frac 1{2^j}\geq \sum_{k=0}^{2\beta_i+1}\mu_i(k)=1,\]
so that $c_i\geq 1/4.$ Hence $\eta_i-\delta_i\geq 1/8$. We get a similar lower bound when $m_i$ is even, and hence we may conclude that $C_\od$ is hypercyclic by Corollary \ref{cor:hc}.

To show that $C_\od$ is not mixing, we observe that for each $i\geq 1,$ there exist at least two different integers $a<b \in\llbracket 0,m_i-1\rrbracket$ such that $\mu_i(a),\mu_i(b)\geq c_i/2\geq 1/8$ 
(consider for instance $a=\beta_i,$ $b=\beta_{i+1}$ if $m_i=2\beta_i+1$ is odd). If we set $j:=b-a$, then $\mu(D)\leq 7/8$ for any $D\subset \Omega_i$ such that $D\cap (D+j)=\varnothing$, since $D$ cannot contain both $a$ and $b$. So $\kappa_i\leq 7/8$ for all $i$, and hence  $C_\od$ is not mixing by Theorem \ref{thm:mixing}.
\end{proof}

We also complete the picture by exhibiting, for any sequence of integers $(m_i),$ a sequence of  measures
$(\mu_i)$ such that $C_\od$ is topologically mixing.

\begin{example}\label{ex:mixing}
Let $(m_i)$ be any sequence of integers with $m_i\geq 2$ for all $i\in\mathbb N.$ There exists a sequence of probability measures $(\mu_i)$ such that $C_\od$ is topologically mixing on $\Lpmu.$
\end{example}
\begin{proof} {
For \( i > 1 \) and \( j \in \llbracket 0, m_i - 1\rrbracket \), we set
\[
\mu_i(j) := \left(1 - \frac{1}{i+1}\right) c_i^j = \frac{i}{i+1}\, c_i^j,
\]
where $c_i$ is chosen so that
\[
\sum_{j=0}^{m_i-1} c_i^j = \frac{i+1}{i}\cdot
\]
Observe first that
\[
\frac{1}{i+1} \leq c_i \leq \frac{1}{i},
\]
the ``limit'' cases corresponding to \( m_i = 2 \) and \( m_i = \infty \).
Note also that \( \eta_i = \mu_i(0) =\frac{i}{i+1}\cdot\) In particular $\prod_{i=1}^\infty \eta_i=0,$
so the measure \( \mu \) is non-atomic. We check that \( C_\od \) is bounded on \( L_p(\Omega,\mu) \). Let \( l > 1 \)  and \( j \in \llbracket0, m_i - 1\rrbracket \). Then
\[
\begin{aligned}
\frac{\mu_l(j-1)}{\mu_l(j)} \prod_{i=1}^{l-1} \frac{\mu_i(m_i-1)}{\mu_i(0)}& \leq c_l^{-1} \prod_{i=1}^{l-1} c_i^{m_i-1} \\
& \leq c_l^{-1} \prod_{i=1}^{l-1} c_i \\
& \leq (l+1) \prod_{i=1}^{l-1} \frac{1}{i},
\end{aligned}
\]
so  the boundedness condition (\ref{eq:continuityodometer}) is satisfied. Since \( \eta_i \to  1 \), $C_\od$ is topologically mixing by Theorem \ref{EmmaUdayan}. }
%
%
\end{proof}

{We point out that the previous known examples of topologically mixing odometers were obtained with measures having atoms (see the proof of \cite[Theorem 3.2]{BDDD22}).}


\subsection{Frequent hypercyclity}

In this section, we show that some composition operators induced by odometers are frequently hypercyclic. We need to introduce two new quantities. For $i\in\mathbb N$ and $\kappa\in(0,1),$ we set 
\begin{align*}
\gamma_i:=\sup\bigl\{\gamma\in(0,1):\ \exists D\subset \Omega_i\; \,&\exists j\in\llbracket 1,m_i-1\rrbracket\,:\,\\
 &\mu_i(D)>\gamma\textrm{ and }\mu_i(D+j)<1-\gamma\bigr\}
\end{align*}
where addition is understood in $\ZZ/m_i\ZZ$; and 
\[ \omega_i(\kappa):=\mu_i\bigl(\llbracket m_i-1-\kappa m_i m_{i+1} ,m_i-1\rrbracket \bigl).\]

\smallskip To help digesting these definitions (and also for future reference), we quote the following remark.
\begin{remark}\label{01} Assume that $\Omega_i=\{ 0,1\}$ for all $i\geq 1$. Then $\gamma_i=\max\bigl( \mu_i(0), \mu_i(1)\bigr)$ and $\omega_i(\kappa)=\mu_i(1)=1-\mu_i(0)$ for all $i$ and any $\kappa <1/4$.
\end{remark}

\smallskip

\begin{theorem}\label{thm:fhc}
Assume that there exists $\kappa\in (0,1)$ such that 
\[ \limsup_{i\to\infty} \min(1-\omega_{i-1}(\kappa),\gamma_{i})=1.\]
Then $C_\od$ is frequently hypercyclic on $\Lpmu$, $1\leq p<\infty$.
\end{theorem}

\begin{proof} We apply Theorem \ref{thm:fhcco}, taking as $\mathcal C$  the family of all cylinder sets of $\Omega$. Note that $\textrm{span}(\mathbf 1_B:\ B\in\mathcal C)$ is dense in $\Lpmu$.

In order to check that Condition $({\rm H}_\kappa)$ is satisfied, let us fix $\veps\in(0,1)$, $m\in\NN$, and $B_1,\dots ,B_r\in\mathcal C$.  Let us also set $\delta:=\veps/2.$ 

 Let $N=i$ be a very large integer such that \[ \max(\omega_{N-1}(\kappa),1-\gamma_N)<\delta,\] 
and choose $D_N\subset \Omega_N$ and $j_N\in\llbracket1,m_N-1\rrbracket$ such that $\mu_N(D_N)>1-\delta$ and $\mu_N(D_N+j_N)<\delta.$ We set 
\begin{align*}
B&:=[\Omega_1,\dots,\Omega_{N-1},D_N+j_N],\\
B'&:=[\Omega_1,\dots,\Omega_{N-1},D_N],\\
n&:=j_N\prod_{i=1}^{N-1}m_i \quad{\rm and}\\
d&:=\prod_{i=1}^{N}m_i.
\end{align*}

Observe that $\od^{-d}(B)=B$, and that if $N$ is large enough then $d\geq m$ and $\od^{-d}(B_j)=B_j$ for $j=1,\dots ,r$. Note also that $\od^{-n}(B)=B'.$ 

We show that $\forall 0\leq k\leq \kappa d\,:\, \mu\bigl(\od^{-k}(B)\bigr)\leq \veps$ and $\mu\bigl( \od^{-n-k}(B)\bigr)\geq1-\veps$. In what follows, we fix $0\leq k\leq \kappa d$. Observe that $\kappa m_{N-1}m_N< m_{N-1}$ since $\omega_{N-1}(\kappa)<1$, so that $k\leq \prod_{i=1}^{N-2}m_i \times(\kappa m_{N-1}m_N)< \prod_{i=1}^{N-1}m_i $.

\smallskip Let us first show that $\mu\bigl( \od^{-k}(B)\bigr)\leq\varepsilon$.  If $x\in\od^{-k}(B)$, then either the computation of $\od^k(x)$ has led to a carry at the $N$-th coordinate, or not. If no carry has occured, then $x_N=(\od^k(x))_N\in D_N+j_N$. If a carry has occured, this can be so only if $x_{N-1}\geq  m_{N-1}-1-\kappa m_{N-1}m_N$ since $k\leq \prod_{i=1}^{N-2}m_i \times(\kappa m_{N-1}m_N)$. So we see that 
\begin{align*}
\mu\bigl( \od^{-k}(B)\bigr)&\leq \mu_N(D_N+j_N)+\mu_{N-1} \bigl( \llbracket m_{N-1}-1-\kappa m_{N-1}m_N,m_{N-1}-1\rrbracket\bigr)\\
&\leq \delta +\omega_{N-1}(\kappa)\\
&\leq \varepsilon.
\end{align*}

Now, let us show that  $\mu\bigl( \od^{-n-k}(B)\bigr)\geq1-\veps$, \textit{i.e.} $\mu\bigl( \od^{-k}(B')\bigr)\geq1-\veps$. Observe that if $x\in B'\setminus \od^{-k}(B')$, then the computation of $\od^k(x)$ has led to a carry at the $N$-th coordinate, so that (as above) $x_{N-1}\geq  m_{N-1}-1-\kappa m_{N-1} m_{N}$. It follows that \[\mu\bigl( B'\setminus \od^{-k}(B')\bigr)\leq \mu_{N-1} \bigl( \llbracket m_{N-1}-1-\kappa m_{N-1}m_N,m_{N-1}-1\rrbracket\bigr)=\omega_{N-1}(\kappa),\]
and hence that 
\[ \mu\bigl( \od^{-n-k}(B)\bigr)=\mu\bigl( \od^{-k}(B')\bigr)\geq \mu(B')-\omega_{N-1}(\kappa)=\mu_N(D_N)-\delta\geq 1-\veps.\]

So we have checked Condition $({\rm H}_\kappa)$, which concludes the proof.
\end{proof}

\begin{remark}
Alternatively we could apply Corollary \ref{cor:fhcgroup} with $d_i:=\prod_{j=1}^i m_j,$ but the heart of the proof remains the same.
\end{remark}

\smallskip The easiest way to estimate $\gamma_i$ is to observe that $\gamma_i\geq \eta_i$ (choose $D=\{a\}$ where
$\mu_i(a)=\eta_i$, and $j=1$), which leads to the following corollary:
\begin{corollary}\label{cor:fhc0} If $\limsup_{i\to\infty}\min(1-\omega_{i-1}(\kappa),\eta_{i})=1$ for some $\kappa>0$, 
then $C_\od$ is frequently hypercyclic on $\Lpmu$.
\end{corollary}

When the sequence $(m_i)$ is bounded, we can simplify the statement.

\begin{corollary}\label{cor:fhc1}
Suppose that  the sequence $(m_i)$ is bounded. Then $C_\od$ is frequently hypercyclic on $\Lpmu$ as soon as \[\limsup_{i\to\infty} \min(1-\mu_{i-1}(m_{i-1}-1),\eta_{i})=1.\] 
\end{corollary}
\begin{proof}
Let $M:=\sup_i m_i$, and take any $\kappa< 1/M^2$. We have $\omega_i(\kappa)=\mu_i(m_i-1)$ for all $i\geq 1$, so Corollary \ref{cor:fhc0} applies.
\end{proof}

\begin{corollary}\label{cor:fhc2}
If  the sequence $(m_i)$ is bounded and  $\lim_{i\to\infty} \eta_i=1,$
 then $C_\od$ is frequently hypercyclic on $\Lpmu$.
\end{corollary}
\begin{proof}
By Corollary \ref{cor:fhc2}, we have just to prove that $\limsup \,(1-\mu_i(m_i-1))=1$. Otherwise, there exists $c>0$ such that $\mu_i(m_i-1)>c$
for all $i\geq 1.$ Since $\eta_i\to 1$ as $i\to\infty$, it follows that $\lim_{i\to\infty} \mu_i(m_i-1)=1$, 
so that $\lim_{i\to\infty} \mu_i(0)=0;$ but this contradicts the boundedness condition (\ref{eq:continuityodometer}). 
\end{proof}

In view of Thorem \ref{EmmaUdayan}, we immediately deduce the following corollary.

\begin{corollary}\label{cor:fhc3}
Assume that the sequence $(m_i)$ is bounded. If $C_\od$ is topologically mixing, then 
it is frequently hypercyclic.
\end{corollary}

\smallskip
\begin{example}\label{ex:mixingfhc}
Assume that $\Omega_i=\{0,1\}$ for all $i\geq 1$, and let $\mu_i(0):=1-\frac 1{i+1}=\frac{i}{i+1}$ and $\mu_i(1):=\frac 1{i+1}\cdot$ 
Then $C_\od$ is  frequently hypercyclic on $\Lpmu$.
\end{example}

\begin{proof}
The boundedness of $C_\od$ follows from the identity
$$\prod_{i=1}^{l-1}\frac{\mu_i(1)}{\mu_i(0)}\times \max\left(\frac{\mu_l(0)}{\mu_l(1)}, \frac{\mu_l(1)}{\mu_l(0)}\right)= \frac{l}{(l-1)!}$$
which is valid for all $l\in\NN$,  whereas  frequent hypercyclicity is a consequence of Corollary \ref{cor:fhc2}.
\end{proof}

%
%
\smallskip
Theorem \ref{thm:fhc} allows us to exhibit very simple examples of chaotic and frequently hypercyclic Hilbert space operators that are not topologically mixing. 
The first construction of such operators dates back to \cite{BadGri07}, and other examples can be found in \cite{GMM21}. Composition operators induced by odometers are arguably ``easier'' examples.

\begin{example}
Let $\Omega_i=\{0,1\}$ for all $i\geq 1$, and define the measures $\mu_i$ as follows: $\mu_{3k+1}(0)=\mu_{3k+2}(0):=1-\frac 1{k+1}$ and $\mu_{3k+3}(0):=1/2$ for all $k\geq 0$. Then $C_\od$ is  frequently hypercyclic on $\Lpmu$, but not topologically mixing.
\end{example}
\begin{proof} The boundedness of $C_\od$ is easily checked, as in Example \ref{ex:mixingfhc}. As for frequent hypercyclicity, observe that if we set $n_k:=3k+2$, then  $\gamma_{n_k}=1-1/(k+1)$ and $\omega_{n_k-1}(1/5)=1/(k+1)$ for all $k\geq 0$ (see Remark \ref{01}), so Theorem \ref{thm:fhc} applies. Finally, $C_\od$ is not topologically mixing by Theorem \ref{EmmaUdayan}, since $\eta_{3k+3}=1/2$ for all $k\geq 0$. 
\end{proof}

\smallskip Our examples leave open the following basic question. 

\begin{question}
Does there exist a composition operator induced by an odometer which is hypercyclic but not frequently hypercyclic ?  
\end{question}

If the answer to this question was ``yes'', this would give natural examples of chaotic operators which are not frequently hypercyclic. The existence of such operators was a longstanding problem in linear dynamics, which was solved by Q. Menet  \cite{Men17}. Since the construction of  \cite{Men17} is rather difficult, it would be nice to have a ``simple'' example.   It seems that there is some flexibility in the odometer setting. Natural candidates are  $\Omega_i=\{0,1\},$ $\mu_i(0)=3/4$ for all $i\geq 1$, or  an odometer for which $m_i\to\infty$. In both cases, the assumptions of Theorem \ref{thm:fhc} are clearly not satisfied: we have $\gamma_i=3/4$ for all $i\geq 1$ in the first case (as well as $\omega_i(\kappa)\geq 1/4$ for any $\kappa$), whereas in the second case, $\omega_i(\kappa)$ is eventually equal to $1$ for any $\kappa\in(0,1).$


\subsection{$\mathcal U$-frequent hypercyclicity}

If we want to prove that the composition operator induced by an odometer is $\mathcal U$-frequently hypercyclic, we can slightly weaken the assumptions of Theorem \ref{thm:fhc}.

\begin{theorem}\label{thm:ufhc}
Assume that there exists $\kappa>0$ such that the following holds true:  for every $\delta>0$, there exist $i\geq 1$ arbitrarily large, $j\in\llbracket 1,m_{i}-1\rrbracket$  and $D\subset\Omega_{i}$ such that 
$$\mu_{i}(D)>1-\delta\;,\;\ \mu_{i}(D+j)<\delta\qquad{\rm and}$$
\begin{equation}\label{eq:thmufhc}
\mu_{i-1}\bigl(\llbracket m_{i-1}-\kappa j m_{i-1},m_{i-1}-1\rrbracket\bigr)<\delta.
\end{equation}
Then $C_\od$ is $\mathcal U$-frequently hypercyclic on $\Lpmu$.
\end{theorem}

\begin{remark}
In Theorem \ref{thm:fhc}, the assumptions  are exactly the same except that \eqref{eq:thmufhc} is replaced by the stronger inequality
$$\mu_{i-1}\bigl(\llbracket m_{i-1}-\kappa m_{i} m_{i-1},m_{i-1}-1\rrbracket\bigr)<\delta.$$
\end{remark}
\begin{proof}[\it Proof of Theorem \ref{thm:ufhc}]
The proof is almost identical to that of Theorem \ref{thm:fhc}. Given $\varepsilon>0$, we set 
\[
B:=[\Omega_1,\dots,\Omega_{N-1},D_N+j_N],\quad\quad 
n:=j_N\prod_{i=1}^{N-1}m_i
\]
for a very large $N=i,$ where $j_N\in\llbracket 1, m_N-1\rrbracket$ and $D_N\subset \Omega_{N}$ are such that $\mu(D_{N}+j_N)<\veps/2,$  $\mu(D_{N})>1-\veps/2$ and 
$\mu_{N-1}\bigl(\llbracket m_{N-1}-\kappa j_N m_{N-1},m_{N-1}-1\rrbracket\bigr)<\veps/2.$ Then $\mu(B)<\veps$; and for any $k\in\llbracket 0,\kappa n\rrbracket,$ the proof of Theorem \ref{thm:fhc} shows that  $\mu(\od^{-(n+k)}(B))>1-\veps.$ 
Hence the assumptions of Theorem \ref{thm:ufhcgeneral} are satisfied with $m:=(1+\kappa n)$ and $\alpha:=\kappa/(1+\kappa).$
\end{proof}

Here is a corollary to be compared with Corollary \ref{cor:fhc0}.
\begin{corollary}\label{cor:ufhc0}
If 
$\limsup_{i\to\infty} \min\left(1-\mu_{i-1}\bigl(\llbracket\kappa m_{i-1},m_{i-1}-1\rrbracket\bigr),\eta_{i}\right)=1$ for some $\kappa<1$, 
then $C_\od$ is $\mathcal U$-frequently hypercyclic on $\Lpmu$. This holds in particular if $\mu_i(0)\to 1$ as $i\to\infty$ or, more generally, if there exists an increasing sequence of integers $(i_s)$ such that $\mu_{i_s-1}(0)\to 1$.
\end{corollary}

\smallskip
\begin{example}
Let $(m_i)$ be any sequence of integers with $m_i\geq 2$ for all $i\in\mathbb N.$ There exists a sequence of probability measures $(\mu_i)$ such that $C_\od$ is $\mathcal U$-frequently hypercyclic on $\Lpmu.$
\end{example}
\begin{proof}
The odometer exhibited in Example \ref{ex:mixing} does the job since $ \mu_i(0)\to 1.$
\end{proof}

\section{Diagonal translation operators} \label{sec:sumoperator}

\subsection{Notation} In this section, we use the notation of Section \ref{subsec:sumoperators}. So we still have $\Omega=\prod_{i\geq 1} \ZZ/m_i\ZZ$, but this time addition is performed coordinatewise. We consider the composition operator $C_{\trans}$ defined by the translation $\trans(x)=x+a$, where $a=(1,1,\dots )$. We view $C_\trans$ as acting on $\Lpmu$, $1\leq p<\infty$, so we assume that Condition (\ref{eq:continuity}) is satisfied. Recall that to get hypercyclic examples, the sequence $(m_i)$ must be unbounded. Note also that the periodic vectors of $C_{\trans}$ are dense in $\Lpmu$, so that $C_\trans$ is chaotic as soon as it is hypercyclic. This follows from Lemma \ref{periodicgeneral}, or just by observing that if $n\in\NN$ and if we set $N=\prod_{i=1}^n m_i$, then $\trans^N(x)_i=x_i$ for all $x\in\Omega$ and $i=1,\dots ,n$, from which it follows that $\mathbf 1_B$ is a periodic vector of $C_{\trans}$ for any cylinder set $B\subset\Omega$.

\subsection{Hypercyclicity and mixing}
 Let us introduce (again) some useful quantities. For $i,n\geq 1$, we set 
\[ \alpha_{i,n}:= \sup\bigl\{\mu_i(D):\ D\subset\Omega_i,\ (D+n)\cap D=\varnothing\bigr\}\]
where addition is performed in $\ZZ/ m_i\ZZ$, and we define 
\[ \beta_i:= \sup_{n\geq 1} \,\alpha_{i,n}\qquad{\rm and}\qquad \gamma_n:=\sup_{i\geq 1} \,\alpha_{i,n}.\]

\begin{lemma}\label{betagamma}
We have $\limsup_{i\to\infty}\beta_i=1$ if and only if $\limsup_{n\to\infty}\gamma_n=1.$
\end{lemma}
\begin{proof} It is clear that $\beta_i<1$ for all $i\geq 1$. Hence,
\[ \limsup \beta_i=1\iff \sup_i \beta_i=1\iff \sup_i\,\sup_n \alpha_{i,n}=1.\] 

Moreover, it is also true that $\gamma_n< 1$
 for all $n\geq 1$. 
 Indeed, since $C_{\trans}$ is bounded on $\Lpmu$, there exists a constant $K$ such that $\mu_i(D-1)\leq K \mu_i(D)$ for all $i\geq 1$ and every $D\subset\Omega_i$. Then $\mu_i(D)\leq K^n \mu_i(D+n)$ for all $i$ and every $D\subset\Omega_i$, so that $\mu_i(D)\leq K^n(1-\mu_i(D))$ if $D\cap (D+n)=\varnothing$; and this shows that $\gamma_n\leq\frac{ K^n}{1+K^n}\cdot$ Therefore,
 \[ \limsup \gamma_n=1\iff\sup_{n} \gamma_n=1\iff \sup_n\,\sup_i \alpha_{i,n}=1.\]
\end{proof}

\smallskip
\begin{proposition}\label{hcsum} If $\limsup_{n\to\infty} \gamma_n=1$, then $C_{\trans}$ is hypercyclic on $\Lpmu$; and if $ \gamma_n\to 1$, then $C_{\trans}$ is topologically mixing. 
\end{proposition}
\begin{proof} Assume that $\limsup\, \gamma_n=1$, and let us show that $C_{\trans}$ is hypercyclic.  By Corollary \ref{lem:topologicallytransitive} (1), it is enough to show that for any $\varepsilon >0$, one can find $B\subset \Omega$ and $n\in\NN$ such that $\mu(B)\leq \varepsilon$ and $\mu(B-na)\geq 1-\varepsilon$. By assumption, there exist $i\geq 1$, $D\subset\Omega_i$ and $n\in\NN$ such that $\mu_i(D)\geq 1-\veps$ and $(D+n)\cap D=\varnothing$. We take that $n$, and $B:=\Omega_1\times\cdots\times\Omega_{i-1}\times (D+n)\times\Omega_{i+1}\times \cdots.$

\smallskip The proof of the mixing case is the same, using Corollary \ref{lem:topologicallytransitive} (2).
\end{proof}

We now provide some examples.

\begin{example} \ \label{ex:trans} Assume that the sequence $(m_i)$ is unbounded. 
\begin{enumerate}
\item[(a)]  There exists a sequence of measures $(\mu_i)$ such that $C_\trans$ is hypercyclic on $\Lpmu.$
\item[(b)]  If $\sum_{i=1}^\infty m_i^{-1}<\infty$ and $m_{i+1}=O(m_i),$ then there exists a sequence of measures $(\mu_i)$ such that $C_\trans$ is topologically mixing on $L_p(\Omega,\mu).$
\end{enumerate}
\end{example}
\begin{proof}(a) 
Let \((i_s)_{s\geq 1}\) be an increasing sequence of integers such that 
$ \sum_{s=1}^\infty m_{i_s}^{-1}<\infty$, 
and choose a sequence of positive numbers $(\delta_s)$ such that 
\[ \delta_s m_{i_s}\to\infty\qquad{\rm and}\qquad \sum_{s=1}^\infty \delta_s<\infty.\] We define \((\mu_i)\) as follows. If \(i \neq i_s\) for all $s$, then \(\mu_i\) is the uniform distribution on \(\llbracket 0, m_i - 1 \rrbracket\). If \(i = i_s\) for some \(s\), we set \(n_{i_s} := \lfloor m_{i_s}/2 \rfloor\) and we write \(\llbracket 0, m_{i_s} - 1 \rrbracket\) as \(J_{s,1} \cup J_{s,2}\), where \(J_{s,1}\) and \(J_{s,2}\) are two consecutive intervals  with \(\# J_{s,2} = n_{i_s}\) (so that \(\# J_{s,1} = n_{i_s}\) or \(n_{i_s} + 1\)). We set \(\rho_s := 1 + \delta_s\), and we choose \(\varepsilon_s\) such that
\begin{equation}\label{eq:sumhc}
\# J_{s,1} \varepsilon_s + \frac{\rho_s^{n_{i_s}} - 1}{\rho_s - 1} \varepsilon_s = 1.
\end{equation}
The measure \(\mu_{i_s}\) is defined by
\[
\begin{cases}
\mu_{i_s}(k) :=\varepsilon_s & \text{if } k \in J_{s,1}, \\
\mu_{i_s}(m_{i_s} - 1 - n_{i_s} + k) :=\rho_s^{n_{i_s} - k} \varepsilon_s & \text{if } k=1,\dots,n_{i_s}.
\end{cases}
\]
It is indeed a probability measure on \(\llbracket 0, m_{i_s} - 1\rrbracket\).
Moreover, we have
\[
\sup_{j \in \Omega_{i_s}} \frac{\mu_{i_s}(j - 1)}{\mu_{i_s}(j)} = \rho_s\qquad\hbox{for all $s\geq 1$},
\]
whereas $\sup_{j\in\Omega_i} \frac{\mu_{i}(j - 1)}{\mu_{i}(j)} =1$ if $i$ is not an $i_s$. So the boundedness condition \eqref{eq:continuity} is satisfied since the infinite product \(\prod_{s \geq 1} \rho_s\) is convergent. Note also that the measure $\mu$ is non-atomic since $\eta_{i_s}=\rho_s^{n_{i_s}-1}\varepsilon_s\leq \delta_s +\varepsilon_s$ by \eqref{eq:sumhc} and hence $\eta_{i_s}\to 0$ as $s\to\infty$.

Now, we observe that $\mu_{i_s}(J_{s,2})\to 1$ as $s\to\infty$. Indeed, we have $(1+\delta_s)^{n_{i_s}}\gg n_{i_s}\delta_s$ since $n_{i_s}\delta_s\to \infty$, so that 
\[ \frac{\mu_{i_s}(J_{s,2})}{\mu_{i_s}(J_{s,1})}=\frac{(1+\delta_s)^{n_{i_s}}-1}{n_{i_s}\delta_s}\to \infty.\]
Hence, if we set $D_{i_s}:=J_{s,2}$, then $\mu_{i_s}(D_{i_s})\to 1$
and $D_{i_s}\cap (D_{i_s}+n_{i_s})=\varnothing.$ It follows that 
 $\limsup_i \beta_i\geq \limsup_s \mu_{i_s}(D_{i_s})=1$, so that $C_\trans$ is topologically transitive by Proposition \ref{hcsum} and Lemma \ref{betagamma}.

\smallskip (b) Let $\kappa>0$ be such that $\kappa m_{i+1}+1\leq (1-\kappa) m_i$ for all $i\geq 1$. 
We set \(n_{i} := \lfloor \kappa m_{i} \rfloor\) and we split \(\llbracket 0, m_i - 1 \rrbracket\) into two consecutive intervals \(J_{i,1} \cup J_{i,2}\)  with \(\# J_{i,2} = n_{i}\). We choose a sequence of positive numbers $(\delta_i)$ such that $m_i\delta_i\to\infty$ and $\sum_{i=1}^\infty \delta_i<\infty$, and we set $\rho_i:=1+\delta_i.$ We define $\veps_i$ by 
\[\# J_{i,1} \varepsilon_i + \frac{\rho_i^{n_{i} }- 1}{\rho_i - 1} \varepsilon_i = 1,\]
and $\mu_i$ by
\[
\begin{cases}
\mu_{i}(k) :=\varepsilon_i & \text{if } k \in J_{i,1}, \\
\mu_{i}(m_{i} - 1 - n_{i} + k) :=\rho_i^{n_{i} - k} \varepsilon_i & \text{if } k=1,\dots,n_{i}.
\end{cases}
\]
As before, the convergence of $\prod_{i\geq 1} \rho_i$ ensures the boundedness of $C_\trans$, and the measure $\mu$ is non-atomic. By Proposition \ref{hcsum}, to show that $C_\trans$ is mixing it is enough to find a sequence $(D_i)_{i\geq 1}$ with $D_i\subset\Omega_i$ such that $\mu_i(D_i)\to 1$ as $i\to\infty$ and, for any $i_0\geq 1$, every large enough integer $n$ is such that $D_i\cap (D_i+n)=\varnothing$ for some $i\geq i_0$.

We set $D_i:=J_{i,2}$ and observe that $\mu_i(D_i)\to 1.$ Moreover, if $n\in\llbracket n_i,m_i-n_i-1\rrbracket$, then $(D_i+n)\cap D_i=\varnothing.$ Now,
$$m_i-n_i-1\geq (1-\kappa)m_i-1\geq  \kappa m_{i+1}\geq n_{i+1};$$
so, for any $i_0\in\NN,$ the set $\bigcup_{i\geq i_0} \llbracket n_i, m_i-n_i-1\rrbracket$ has the form $\llbracket n_0,\infty)$ for some $n_0\in\ZZ_+$. Hence, every large enough integer $n$ is such that $D_i\cap (D_i+n)=\varnothing$ for some $i\geq i_0$.
\end{proof}

\begin{question} Assuming only that the sequence $(m_i)$ is unbounded, is it possible to find a sequence of measures $(\mu_i)$ such that $C_\trans$ is topologically mixing on $\Lpmu$?
\end{question}

\smallskip
Let us now give another criterion for hypercyclicity or mixing, which is based on the behaviour of $\trans$ on several coordinates. For $i, n\geq 1$  we  set 
\[
\theta_{i,n}:=\sup\bigl\{\mu_i(D)-\mu_i(D+n):\ D\subset \Omega_i
\bigr\}\]
where addition is understood in $\Omega_i$, 
and 
\[ \widetilde \gamma_n:=\sup\left\{ \frac1{\#I}\,\Biggl(\sum_{i\in I}\theta_{i,n}\Biggr)^2\,:\, I\subset\NN\;{\rm finite}\right\}.
\]

\smallskip
\begin{theorem}\label{Hoeffbis}  If $\limsup_{n\to\infty} \widetilde\gamma_n=\infty,$ then $C_{\trans}$ is hypercyclic on $\Lpmu$; and if $\widetilde\gamma_n\to \infty,$ then $C_{\trans}$ is topologically mixing.
\end{theorem}
\begin{proof}
We only prove the hypercyclic case, the mixing case being completely similar. The proof follows that of Theorem \ref{thm:hc}, but it is easier since we do not have to worry about carries.

By Corollary \ref{lem:topologicallytransitive}, it is enough to show that given $\varepsilon >0$, one can find a Borel set $B\subset\Omega$ and $n\in\NN$ such that $\mu(B)\geq 1-\varepsilon$ and $B\cap \trans^n(B)=\varnothing$. 

Let $\varepsilon>0.$ By assumption, there exist $n,N\geq 1$ and  $1\leq i_1<\cdots <i_N$ such that 
$$\exp\left(-\frac{2}{9N}\left(\sum_{s=1}^N \theta_{i_s,n}\right)^2\right)<\veps.$$
For $s=1,\dots ,N$, let $D_s\subset \Omega_{i_s}$ be such that $(D_s+n)\cap D_s=\varnothing$ and $\mu_{i_s}(D_s)-\mu_{i_s}(D_s+n)=\theta_{i_s,n}.$ Define random variables $X_s,Y_s:(\Omega,\mathcal B,\mu)\to \{ 0 , 1\}$ by
\[ X_s:=\mathbf 1_{\{x_{i_s}\in D_s\}}\qquad{\rm and}\qquad Y_s:=\mathbf 1_{\{x_{i_s}\in D_s+n\}}.\]

Now, consider 
\begin{align*}
    B_X&:=\left\{ \sum_{s=1}^N X_s\geq \sum_{s=1}^N \left(\mu_{i_s}(D_s)-\frac {\theta_{i_s,n}}3\right)\right\},\\
    B_Y&:=\left\{\sum_{s=1}^N Y_s\leq \sum_{k=1}^N \left(\mu_{i_s}(D_s+n)+\frac {\theta_{i_s,n}}3\right)\right\},\\
\end{align*}
and
\[    B=B_X\cap B_Y.\]

Observe that since $\mu_{i_s}(D_s)-\frac {\theta_{i_s,n}}3>\mu_{i_s}(D_s+n)+\frac {\theta_{i_s,n}}3$ (by the choice of $D_s$) and  $Y_s(\trans^n x)=X_s(x)$ for all $x\in\Omega$, we have $\trans^n(B)\cap B=\varnothing.$ Moreover,
\begin{align*}
    \mu(B_X)&=1-\mu \left( \sum_{s=1}^N X_s< \sum_{s=1}^N \left(\mu_{i_s}(D_s)-\frac {\theta_{i_s,n}}3\right)\right)\\
    &=1-\mu\left( \sum_{s=1}^N \bigl( X_s-\mathbb E(X_s)\bigr)<-\frac{1}3\sum_{s=1}^N \theta_{i_s,n}\right).
\end{align*}
Since $X_1,\dots ,X_s$ are independent Bernoulli variables, it follows (by Hoeffding's inequality) that
\[ \mu(B_X)\geq 1-\exp\left(-\frac2{9N}\left(\sum_{s=1}^N \theta_{i_s,n}\right)^2\right)\geq 1-\veps.\]
Similarly $\mu(B_Y)\geq 1-\varepsilon$, and hence $\mu(B)\geq 1-2\veps.$
\end{proof}

\smallskip  We point out the following consequence of Theorem \ref{Hoeffbis}, to be compared with Theorem \ref{thm:hc}. Recall the notation 
\[ \theta_i=\sup\bigl\{\mu_i(D)-\mu_i(D+k):\ D\subset \Omega_i,\ k\in\ZZ
\bigr\}.\]

\begin{corollary}\label{coprime} Assume that the integers $m_i$ are pairwise coprime, and that 
\[ \sup\,\left\{ \frac1{\# I}\left(\sum_{i\in I} \theta_i\right)^2\,:\, I\subset\NN\;\,{\rm finite}\right\}= \infty.\]

Then $C_\trans$ is hypercyclic on $\Lpmu$, $1\leq p<\infty$.
\end{corollary}
\begin{proof} It is enough to show that given $A>0$ and $n_0\in\NN$, one can find $n\geq n_0$ such that $\widetilde\gamma_n>A$. By assumption, one can find a finite set $I\subset\NN$ such that 
\[  \frac1{\# I}\left(\sum_{i\in I} \theta_i\right)^2>A.\]
For $i\in I$, choose a set $D_i\subset\Omega_i$ and an integer $k_i$ such that 
$\mu_i(D_i)-\mu_i(D_i+k_i)=\theta_i$. By the Chinese Remainder Theorem, one can find an integer $n\geq n_0$ such that $n\equiv k_i\; [m_i]$ for all $i\in I$. Then $\theta_{i,n}=\theta_i$ for all $i\in I$; hence $\widetilde \gamma_n\geq  \frac1{\# I}\left(\sum_{i\in I} \theta_i\right)^2>A$. 
\end{proof}

\smallskip

Theorem \ref{Hoeffbis} allows us to give examples of hypercyclic $C_\trans$ for which Proposition \ref{hcsum} cannot be applied.
\begin{example} Let $m_i=2^{l+1}$ if $l^2\leq i<(l+1)^2$ 
for some $l\geq 1$. Then, one can find a sequence of measures $(\mu_i)$ such that $C_\trans$ is hypercyclic on $\Lpmu$ and $\limsup \beta_i<1$.
\end{example}
\begin{proof}
We write $m_i=2n_i$, so $n_i=2^l$ if $l^2\leq i<(l+1)^2$ 
for some $l\geq 1$.  Let $\delta_i:=1/n_i$ and $\rho_i:=1+\delta_i$, let 
$\varepsilon_i$ be such that 
\[ n_i\varepsilon _i+ \frac{\rho_i^{n_i}-1}{\rho_i-1}\,\varepsilon_i=1,\]
and define a probability measure $\mu_i$ on $\Omega_i$ as follows:
\[\left\{
\begin{array}{rcll}
\mu_i(k)&:=&\veps_i&\textrm{ for }k=0,\dots,n_i-1,\\
\mu_i(n_i+k)&:=&\rho_i^{n_i-1-k}\veps_i&\textrm{ for }k=0,\dots,n_i-1.
\end{array}
\right.\]
First, observe that 
$$\sum_{i= 1}^\infty \delta_i=\sum_{l= 1}^\infty\sum_{l^2\leq i <(l+1)^2} 2^{-l}<\infty,$$
so that $C_{\trans}$ is bounded on $\Lpmu$.

Next, let $l\geq 1$ and consider $l^2\leq i<(l+1)^2$. 
Taking $D:=\llbracket 2^l,2^{l+1}-1\rrbracket\subset\Omega_i$, we see that 
\begin{align*}
\theta_{i,2^l}&\geq\left(\frac{\rho_{i}^{n_{i}}-1}{\rho_{i}-1}-n_{i}\right)\veps_i\\
&\geq\left(\left(1+2^{-l}\right)^{2^l}-2\right)2^l\veps_i.
\end{align*}
Since we also have
$$\left(\frac{\rho_i^{n_i}-1}{\rho_i-1}+n_i\right)\veps_i=1,$$
we get
$$\left(1+2^{-l}\right)^{2^l}2^l\veps_i=1$$
which yields
$$\theta_{i,2^l}\geq\frac{\left(1+2^{-l}\right)^{2^l}-2}{\left(1+2^{-l}\right)^{2^l}}\xrightarrow{l\to\infty} \frac{e-2}{e}\cdot$$

So, there exists $c>0$ such that $\theta_{i,2^l}\geq c$ for any $l\geq 1$ and $l^2\leq i<(l+1)^2$. Taking $n:=2^l$, $N:=2l+1$ and $i_1:=l^2, \dots ,i_N:=l^2+N-1$ in the definition of $\widetilde \gamma_{n}$, it follows that $\widetilde\gamma_{2^l}\to \infty$ as $l\to\infty$. Hence, $C_{\trans}$ is hypercyclic by Theorem \ref{Hoeffbis}.  

To see that Proposition \ref{hcsum} does not apply, observe that if $l^2\leq i<(l+1)^2$ and if $D\subset \Omega_i=\llbracket 0, 2^{l+1}-1\rrbracket$ is such that $D\cap (D+n)=\varnothing$ for some $n\in\ZZ_+$, then $\# D\leq 2^l$, and that the set $D\subset \llbracket 0, 2^{l+1}-1\rrbracket$ with cardinality $2^l$ and greatest $\mu_i\,$-$\,$measure is $D=\llbracket 2^l,2^{l+1}-1\rrbracket$, for which 
\[ \mu_i (\llbracket 2^l,2^{l+1}-1\rrbracket)=1-n_i\varepsilon_i=\frac{\left(1+2^{-l}\right)^{2^l}-1}{\left(1+2^{-l}\right)^{2^l}}\cdot\]Hence, we have 
\begin{align*}
\beta_i&\leq \frac{\left(1+2^{-l}\right)^{2^l}-1}{\left(1+2^{-l}\right)^{2^l}} \qquad \hbox{if} \quad l^2\leq i<(l+1)^2,
\end{align*}
so that $\limsup_i \beta_i\leq(e-1)/e<1$.
\end{proof}

\subsection{Frequent hypercyclicity}
The next theorem provides a simple sufficient condition for frequent hypercyclicity of $C_\trans$.
\begin{theorem}\label{thm:fhcsumoperator}
  Let $(d_i)_{i\geq 1}$ be an increasing sequence of integers such that $d_i$ is a multiple of $m_1,\dots ,m_i$. Assume that there exists $\kappa>0$ such that for any $\varepsilon >0$, the following holds true: 
{for all $i_0\in\NN,$ there exist $i\geq i_0,$  }
$n\in\NN$  and $D\subset\Omega_i$ such that 
\[\forall 0\leq k\leq \kappa d_i\;:\; \mu_i(D-k)\leq \varepsilon\quad{\rm and}\quad \mu_i(D-(n+k))\geq 1-\varepsilon.\]
Then $C_{\trans}$ is frequently hypercyclic.
\end{theorem}
\begin{proof} We apply Corollary \ref{cor:fhcgroup}. The assumption on $d_i$ implies that $d_ia\to 0$. 
For $i\geq 1$, let $B_i:=[\Omega_1,\dots,\Omega_{i-1},D_i]$, the cylinder defined by $D_i$. Then $B_i-d_ia=B_i$ because $d_i$ is a multiple of $m_i$. So the assumptions of Corollary \ref{cor:fhcgroup} are satisfied.
\end{proof}

\smallskip
{
\begin{example}\label{ex:fhcnotmixing}
Assume $m_{i+1}$ is a multiple of $m_i$ for all $i\geq 1$.
\begin{enumerate}
\item[(a)] If $m_i\to\infty$, there exists a sequence of measures $(\mu_i)$ such that $C_\trans$ is frequently hypercyclic on $L_p(\Omega,\mu)$. 
\item[(b)] If the sequence $(m_{i+1}/m_i)$ is unbounded, there exists a sequence of measures $(\mu_i)$ such that $C_\trans$ is frequently hypercyclic on $L_p(\Omega,\mu)$ and \emph{topologically rigid} along $(m_i)$, \textit{i.e.} $C_\trans^{m_i}f\to f$ for all $f\in\Lpmu$. In particular, $C_\trans$ is not topologically mixing.
\end{enumerate}
\end{example}
\begin{proof} (a) Since $m_i\to\infty$, one can choose a sequence of positive numbers $(\delta_i)$ such that $\sum_{i=1}^\infty \delta_i<\infty$ and $\limsup \delta_im_i=\infty$. We set $\rho_i:=1+\delta_i$. For $i< 5$, we take as $\mu_i$ any probability measure on $\Omega_i$. For $i\geq 5$, we define $\mu_i$ as follows. Let $n_i:=\lfloor m_{i}/5\rfloor$. We split $\Omega_{i}$ into three consecutive intervals, 
$\Omega_i=J_{i,1}\cup J_{i,2}\cup J_{i,3}$ with $\#J_{i,2}=\#J_{i,3}=n_{i}$. This implies that $3n_{i}\leq \#J_{i,1}\leq 3n_{i}+4.$ We now choose $\veps_i$ such that
\[(\#J_{i,1}+\#J_{i,3})\veps_i+\frac{\rho_i^{n_{i}}-1}{\rho_i-1}\veps_i=1,\]
and we define $\mu_{i}$ by
\[
\begin{cases}
\mu_{i}(k) :=\varepsilon_i & \text{if } k \in J_{i,1}\cup  J_{i,3}, \\
\mu_{i}(\# J_{i,1} + k) :=\rho_s^{n_{i} - k-1} \varepsilon_s & \text{if } k=0,\dots,n_{i}-1.
\end{cases}
\]

As in the proof of  Example \ref{ex:trans}, the convergence of the infinite product $\prod_{i\geq 1} \rho_i$ ensures that $C_\trans$ is bounded on $\Lpmu$; and the measure $\mu$ is non-atomic because $\eta_i\leq \delta_i+\varepsilon_i\leq \delta_i+1/2n_i$ for all $i\geq 5$ and hence $\eta_i\to 0$. Moreover, since $\limsup \delta_i n_i=\infty$, we have $\limsup \mu_i(J_{i,2})=1$, again as in the proof of Example \ref{ex:trans}.

Let $d_i:=m_i$ and $D_{i}:= J_{i,2}\cup  J_{i,3}$. For $k=0,\dots ,n_{i}-1$, we see that \[ D_i-k\supset J_{i,2}\qquad{\rm and}\qquad D_i-(2n_i+k)\subset J_{i,1}.\]
Since $\limsup \mu_i(J_{i,2})=1$ and $n_i\geq m_i/6$, it follows that given $\varepsilon>0$, one can find $i$ arbitrarily large such 
\[ \forall 0\leq k\leq d_i/6\;:\; \mu_i(D_i-k)\geq 1-\varepsilon\qquad{\rm and}\qquad \mu_i(D_i-(2n_i+k)\bigr)\leq \varepsilon.\]
Hence, $C_\trans$ is frequently hypercyclic by the ``flipped'' version of Theorem \ref{thm:fhcsumoperator} (see Remark \ref{rem:fhcco}).

\smallskip
(b) Since the sequence $(m_{i}/m_{i-1})$ is unbounded, one can choose a sequence of positive numbers $\delta_i$ such that $\limsup \delta_i m_i=\infty$ and $\sum_{i=2}^\infty \delta_i m_{i-1}<\infty$. We take as $(\mu_i)$ the same sequence as in (a), so that $C_\trans$ is frequently hypercyclic on $\Lpmu$.  As for topological rigidity, note that since $m_{i+1}$ is a multiple of $m_i$ for all $i$, we already know that $C_\trans^{m_i} \mathbf 1_B\to \mathbf 1_B$ for every cylinder set $B\subset\Omega$. Hence, to show that $C_\trans$ is topologically rigid along $(m_i)$,   it is enough to show that $\sup_{i\geq 1} \Vert C_\trans^{m_i}\Vert <\infty$.

Let us first observe that 
\[ K:=\sup_{i\geq 1} \prod_{j=i}^{\infty}\rho_j^{m_{i-1}}<\infty.\]
Indeed, since $\log(\rho_j)\leq \delta_j$, we have for all $i\geq 2$:
\begin{align*}
\log\left( \prod_{j=i}^{\infty}\rho_j^{m_{i-1}}\right)\leq m_{i-1} \sum_{j\geq i} \delta_j\leq \sum_{j\geq 2} \delta_j m_{j-1}.
\end{align*} 

Now, let us show that $\|C_{\trans}^{m_{i-1}}\|\leq K$  for all $i\geq 1$. It is enough to check that $\mu\bigl( \trans^{-m_{i-1}}(B)\bigr)\leq K\mu(B)$ for every basic cylinder $B=[x_1,\dots,x_n]\subset \Omega$. Indeed, it then follows that $\mu(\trans^{-m_{i-1}}(B))\leq K\mu(B)$ for all Borel sets $B\subset\Omega$, which gives $\|C_{\trans}^{m_{i-1}}\|\leq K$.

If $n\leq {i-1}$,  then $\trans^{-m_{i-1}}(B)=B$ and there is nothing to do. If $n\geq i$, then  
\[ \trans^{-m_{i-1}}(B)=[x_1,\dots,x_{i-1},x_i-m_{i-1},\dots,x_n-m_{i-1}],\] so that 
\begin{align*}
\mu\big(\trans^{-m_{i-1}}(B)\big)&\leq \mu([x_1,\dots,x_n])\times \prod_{j=i}^{\infty}\sup_{k\in\Omega_j}\frac{\mu_j(k-m_{i-1})}{\mu_j(k)}\\
&\leq\prod_{j=i}^{\infty}\rho_j^{m_{i-1}} \times \mu(B)\leq K\mu(B).
\end{align*}

This concludes the proof.
\end{proof}

}

\begin{remark}
In Example \ref{ex:fhcnotmixing} (b), the translation $\trans$ is conservative, and hence $C_\trans$ cannot satisfy the Frequent Hypercyclicity Criterion if $p\geq 2$ by \cite{DP21}. Indeed, the proof has shown that $\mu(B)\leq K\mu(\trans^{m_{i-1}}(B))$ for all $i\geq 1$ and any Borel set $B\subset\Omega$.  In particular, if $\mu(B)>0$ then the series $\sum \mu(\trans^n(B))$ is divergent. Since $\mu$ is a probability measure, it follows that for any such $B$, one can find $n<n'$ such that $\trans^n(B)\cap \trans^{n'}(B)\neq\varnothing$; which implies that $\trans$ is conservative. 
\end{remark}

\smallskip
{Example \ref{ex:fhcnotmixing} (b) calls for some comments. It follows from the main results of \cite{SophieTanja}  that if $(m_i)$ is a sequence of integers such that $m_{i+1}$ is a multiple of $m_i$ for all $i\geq 1$ and the sequence $(m_{i+1}/m_i)$ is unbounded, then there exists a Hilbert space operator $T$ which is  weakly mixing with respect to some nondegenerate invariant Gaussian measure (hence frequently hypercyclic) and uniformly rigid along $(m_i)$, \textit{i.e.} $\Vert T^{m_i}-{\rm Id}\Vert \to 0$. The operator $T$ in \cite{SophieTanja} is specifically constructed in order to satisfy these requirements, and the construction is quite nontrivial. We find it interesting that the very simply defined operator $C_\trans$ from Example \ref{ex:fhcnotmixing} (b) happens to have similar properties. Incidentally, it is plausible that in fact, $C_\trans$ is weakly mixing with respect to some nondegenerate invariant Gaussian measure. Since (for Hilbert spaces operators at least) this property is equivalent to having a perfectly spanning set of unimodular eigenvectors, this leads to the following question, that could  be asked for odometers as well.
\begin{question} What are the eigenvalues and the eigenvectors of  $C_\trans$? 
\end{question}

Concerning this question, the two extreme situations are \textit{a priori} possible: it may be that all eigenvalues of $C_\trans$ are roots of unity (regardless of the measure $\mu$), or that $C_\trans$ has a perfectly spanning set of unimodular eigenvectors as soon as it is hypercyclic. We are not ready to bet on any of the two alternatives. However, one can observe the following: if $\lambda$ is an eigenvalue of $C_\trans$ and has an associated eigenvector which happens to be ($\mu$-almost everywhere equal to) a continuous function, then $\lambda=\gamma(a)$ for some continuous character $\gamma$ of $(\Omega, +)$; and hence $\lambda$ has to be a root of unity due to the form of $\Omega$. Indeed, if $f(x+a)=f(x)$ $\mu$-almost everywhere for some non-zero continuous function $f:\Omega\to \CC$, then in fact $f(x+a)=f(x)$ everywhere because the measure $\mu$ has full support (recall that $\mu_i(j)>0$ for all $i$ and every $j\in\Omega_i$). It follows that $\gamma(a)\widehat f(\gamma)=\lambda \widehat f(\gamma)$ for every character $\gamma$, which gives the result. But this seems far from telling the whole story.
}

\subsection{$\mathcal U$-frequent hypercyclicity}
In our context, Corollary \ref{cor:ufhcgroup} implies the following easy sufficient condition.
\begin{theorem}\label{ufhcsum}
Assume as usual that $C_{\trans}$ is bounded on $\Lpmu$.
\begin{enumerate}
\item[\rm (1)] Suppose that there exists $\alpha>0$ such that the following holds true: for any $\veps>0$, there exists an arbitrarily large $m\geq 1,$ $i\geq 1$ and $D\subset\Omega_i$ such that 
$\mu_i(D)<\veps$ and $\#\bigl\{k\in\llbracket 1,m\rrbracket:\ \mu_i(D-k)>1-\veps\bigr\}\geq\alpha m.$
Then $C_{\trans}$ is $\mathcal U$-frequently hypercyclic.
\item[\rm (2)] Suppose that for any $A\subset \NN$ with $\overline{\rm dens}(A)>0$, there exists $\alpha>0$ such that: for any $\veps>0$, there exists an arbitrarily large $m\geq 1,$ $i\geq 1$ and $D\subset\Omega_i$ such that 
$\mu_i(D)<\veps$ and $\#\bigl\{k\in\llbracket 1,m\rrbracket:\ k\in A\textrm{ and }\mu_i(D-k)>1-\veps\bigr\}\geq\alpha m.$
Then $C_{\trans}$ is hereditarily $\mathcal U$-frequently hypercyclic.
\end{enumerate}
\end{theorem}
\begin{proof}
Both proofs follow the same argument as the proof of Theorem \ref{thm:fhcsumoperator}.
\end{proof}

{
\begin{example}
Assume that the sequence $(m_i)$ is unbounded.  Then there exists a sequence of measures $(\mu_i)$ such that $C_\trans$ is $\mathcal U$-frequently hypercyclic on $L_p(\Omega,\mu).$
\end{example}
\begin{proof}
We follow the proof of Example \ref{ex:trans} but we now choose $n_{i_s}:=\lfloor m_{i,s}/3\rfloor.$ Then $(D_{i_s}-k)\cap D_{i_s}=\varnothing$ for $n_{i_s}\leq k\leq 2n_{i_s}$ so that we can apply the ``flipped'' version of Theorem \ref{ufhcsum} with $m:=n_{i_s}$ and $\alpha:=1/2$; see Remark \ref{hum}.
\end{proof}
}
 
\smallskip
Theorem \ref{ufhcsum} also allows us to give new examples of hereditarily $\mathcal U$-frequently hypercyclic operators. We first need 
an elementary lemma.
\begin{lemma}
Let $A\subset\NN$ have positive upper density. There exists $\delta>0$ such that, for any $N\geq 1,$ there exists $n\geq N$ such that
$$\frac{\#\left(\llbracket 2^n,2^{n+1}\llbracket\,\cap \,A\right)}{2^n}\geq\delta.$$
\end{lemma}
\begin{proof}
Assume that this is not the case. Let $\delta>0$. By assumption, there exists $N\geq 1$ such that
\[ \forall n\ge N\,:\,\frac{\#\left(\llbracket 2^n,2^{n+1}\llbracket\,\cap\, A\right)}{2^n}<\delta.\]
Let $m\geq 2^N$ and let $n\geq N$ be such that $2^n\leq m<2^{n+1}$. Then
\begin{align*}
\frac{\#(\llbracket 1,m\rrbracket\cap A)}{m}&\leq \frac{\#(\llbracket 1,2^{n+1}\llbracket\,\cap\, A)}{2^n}\\
&=  \frac{\#(\llbracket 1,2^{N}-1\rrbracket\cap A)}{2^n}
+\sum_{j=N}^n  \frac{\#(\llbracket 2^j,2^{j+1}\llbracket\, \cap\, A)}{2^j}\times\frac 1{2^{n-j}}\\
&\leq \frac{2^{N+1}}m+\sum_{l=0}^{\infty}\frac{\delta}{2^l}\cdot
\end{align*}
Since $m\geq 2^N$ is arbitrary, it follows that $\overline{\textrm{dens}}(A)\leq 2\delta$ for any $\delta >0$, a contradiction.
\end{proof}

\begin{example}
Let $n_i:=2^i$ and set $m_i:=3n_i$. Write $\Omega_i=\llbracket 0,m_i-1\rrbracket$ 
as $J_{1,i}\cup J_{2,i}$ where $J_{1,i}$ and $J_{2,i}$ are consecutive intervals with
respective length $2n_i$ and $n_i$. As in Example \ref{ex:trans}, it is possible to define $\mu_i$ on $\Omega_i$ such that $C_{\trans}$ is bounded on $\Lpmu$  and $\mu_i(J_{2,i})\to 1.$ Then, $C_{\trans}$ is hereditarily $\mathcal U$-frequently hypercyclic.
\end{example}
\begin{proof} 
Let $A\subset \NN$ have positive upper density, and choose $\delta >0$ according to the above lemma, \textit{i.e.}  such that for arbitrarily large $n\in\NN$, we have
\begin{equation}
    \frac{\#(\llbracket 2^n,2^{n+1}\llbracket\, \cap\, A)}{2^n}\geq\delta.\label{eq:hufhc}
\end{equation}
We show that the assumption of the ``flipped'' version of Theorem \ref{ufhcsum} is satisfied with $\alpha:=\delta/2$; see Remark \ref{hum}.

Let $\veps>0$, and consider $n$ arbitrarily large so that 
\eqref{eq:hufhc} is satisfied. Set also $m:=2^{n+1}$, $i:=n$ and $D:=J_{2,i}$. Then
 $J_{2,i}-k\subset J_{1,i}$ for any $k\in\llbracket2^n,2^{n+1}\rrbracket$, 
so that $\mu_i(D)\geq 1-\veps$ and $\mu_i(D-k)\leq \veps$ provided $n$ is large enough.
Therefore,
\begin{align*}
    \#\bigl\{k\in\llbracket 1,m\rrbracket:\ k\in A\textrm{ and }\mu_i(D-k)\leq \veps\bigr\}&\geq \#(\llbracket 2^n,2^{n+1}\llbracket\, \cap\, A)\geq \delta 2^n\geq \alpha m.
\end{align*}
\end{proof}

\subsection{A semigroup variant}
We can give a semigroup variant of all what we have done, in the following way. 
Assume now that \[ \Omega=\prod_{i=1}^{\infty}\RR/m_i\ZZ=:\prod_{i=1}^\infty \Omega_i,\] and let $\nu=\prod_{i=1}^{\infty}\nu_i$ be a product probability measure on $\Omega$, where each $\nu_i$ is equivalent to $\mathcal L_i$, the normalized Lebesgue measure on $\Omega_i=\RR/m_i\ZZ$. Define $T_tf(x)=f(x+ta)$ for $f\in\Lpmu$ and $t>0$. Then provided 
$$\prod_{i=1}^{\infty}\sup_{A\subset\Omega_i,\ t\in[0,1]}\frac{\nu_i(A-t)}{\nu_i(A)}<\infty,$$
the semigroup $(T_t)$ is (well-defined and) strongly continuous  on $\Lpmu$.

All of our previous statements admit semigroup versions. In the other direction, it is easy to adapt the examples given in this section so that they fit into this semigroup context. Starting from  probability measures $\mu_i$ on $\mathbb Z/m_i\mathbb Z,$ we simply define for any Borel set $A\subset[0,m_i),$
\[ \nu_i(A)=\sum_{j=0}^{m_i-1}\mu_i(j)\, \mathcal L_i\bigl(A\cap [j,j+1)\bigr).\]
In particular, if we start from Example \ref{ex:fhcnotmixing}, we get an example of a chaotic and frequently hypercyclic strongly continuous semigroup which fails to be topologically mixing. Previous examples have been obtained in \cite{BadGri07}, with a more difficult construction based on a renorming of some weighted $\ell_2\,$-$\,$space.

\providecommand{\bysame}{\leavevmode\hbox to3em{\hrulefill}\thinspace}
\providecommand{\MR}{\relax\ifhmode\unskip\space\fi MR }
\providecommand{\MRhref}[2]{%
  \href{http://www.ams.org/mathscinet-getitem?mr=#1}{#2}
}
\providecommand{\href}[2]{#2}

\end{document}